\documentclass[a4paper,10pt]{amsart}
\usepackage{amsthm,amsmath,amsfonts,amssymb}
\usepackage{bm}
\usepackage{mathtools,mathrsfs,braket}
\usepackage[pdfdisplaydoctitle,colorlinks,breaklinks,urlcolor=blue,linkcolor=blue,citecolor=blue]{hyperref} 
\usepackage[nameinlink,capitalise]{cleveref}
\AddToHook{env/theorem/begin}
  {\crefalias{section}{theorem}}

\AddToHook{env/lemma/begin}
  {\crefalias{theorem}{lemma}}

\AddToHook{env/proposition/begin}
  {\crefalias{theorem}{proposition}}

\AddToHook{env/corollary/begin}
  {\crefalias{theorem}{corollary}}

\AddToHook{env/definition/begin}
  {\crefalias{theorem}{definition}}
  \AddToHook{env/example/begin}
  {\crefalias{theorem}{example}}

\newcommand{\C}{\mathbb{C}}

\newcommand{\HH}{\mathcal{H}}

\newcommand{\N}{\mathbb{N}}

\newcommand{\PP}{\mathbb{P}}

\newcommand{\Q}{\mathbb{Q}}
\newcommand{\R}{\mathbb{R}}

\DeclareMathOperator{\trace}{Tr}

\renewcommand{\epsilon}{\varepsilon}

\newcommand{\one}{\bm{1}}
  
\newcommand{\pa}[1]{\left(#1\right)}

\newcommand{\norm}[1]{\left\|#1\right\|}
\newcommand{\brak}[1]{\left\langle#1\right\rangle}
\newcommand{\wick}[1]{:\mathrel{#1}:}
\newcommand{\expt}[1]{\mathbb{E}\left[#1\right]}

\newcommand{\Bk}{\color{black}}

\newcommand{\Bl}{\color{blue}}

\newtheorem{theorem}{Theorem}[section]
\newtheorem{definition}[theorem]{Definition}
\newtheorem{hypothesis}[theorem]{Hypothesis}
\newtheorem{corollary}[theorem]{Corollary}
\newtheorem{lemma}[theorem]{Lemma}
\newtheorem{proposition}[theorem]{Proposition}
\newtheorem{example}[theorem]{Example}

\theoremstyle{remark}
\newtheorem{remark}[theorem]{Remark}
\numberwithin{equation}{section}

\allowdisplaybreaks

\newenvironment{acknowledgements}{%
  \section*{Acknowledgements}
}{}

\title[Effective Models for Schr\"odinger Equations]{An Effective SPDE Model for a Schr\"odinger Equation with a Fluctuating Magnetic Potential}
\author[F. Flandoli]{Franco Flandoli}
\address{Scuola Normale Superiore, Piazza dei Cavalieri, 7, 56126 Pisa, Italia}
\email{\href{mailto:franco.flandoli at sns.it}{franco.flandoli at sns.it}}
\author[R. Fukuizumi]{Reika Fukuizumi}
\address{Department of Mathematics, Faculty of Fundamental Science and Engineering,\linebreak Waseda University, 169-8555 Tokyo, Japan}
\email{\href{mailto:fukuizumi at waseda.jp}{fukuizumi at waseda.jp}}
\author[F. Grotto]{Francesco Grotto}
\address{Università di Pisa, Dipartimento di Matematica, 5 Largo Bruno Pontecorvo, 56127 Pisa, Italia}
\email{\href{mailto:francesco.grotto at unipi.it}{francesco.grotto at unipi.it}}
\author[E. Luongo]{Eliseo Luongo}
\address{Fakultät für Mathematik, Universität Bielefeld, 33501 Bielefeld, Germany}
\email{\href{mailto:eluongo at math.uni-bielefeld.de}{eluongo at math.uni-bielefeld.de}}
\date\today

\begin{document}
\begin{abstract}
    We consider the Schr\"odinger equation for a charged particle in a randomly fluctuating external magnetic field and establish a singular perturbation limit in which the solution converges to a deterministic Schr\"odinger equation with Gaussian fluctuations. We propose a stochastic Schr\"odinger equation incorporating the behavior of both the deterministic average and the fluctuations and show that it has the same limiting behavior as the original model, while providing a simpler way to study the distribution of relevant observables.
\end{abstract}

\maketitle

\section{Introduction}\label{sec:introduction}
Quantum particles in time-varying magnetic fields are described for $t \geq 0$ by the Schr\"{o}dinger equation on $\R^d$
\begin{equation*}
i\partial_{t}\psi^{\tau}=(-i\nabla-A^{\tau})^{2}\psi^{\tau},\tag{RMSE}%
\label{eq:rmse}%
\end{equation*}
where $\psi^{\tau}\left(  t,x\right)  $ is the particle wave function and $A^{\tau}\left(  t,x\right)  $ is the vector magnetic potential (for the subscript $\tau$ see below). 
It is an interesting topic due to special phenomena that may occur, especially the emergence of the so-called ponderomotive force and the particle trapping by effective potentials, both generated by fast varying electro-magnetic fields, see for instance the pioneering works \cite{Kap,KugPau,LovMeh,Pri} and the analysis of the averaging effective potential for classical and quantum systems in \cite{RhaGil}. Effective potential traps and ponderomotive force find applications in Bose-Einstein condensation, quantum information and plasma particle acceleration\ \cite{Leg,MonCam,RibPer}. The effective potential is however only a first approximation of the rapidly varying electro-magnetic field;\ corrections and additional phenomena have been investigated in the Physics literature, see for instance \cite{BesPen}, \cite{RidDav}, \cite{MulMor}. We propose here a new conservative stochastic model which incorporates random corrections and may throw light on the modifications of the tunneling properties due to fluctuations around the effective potential. 

Here we consider a stochastic magnetic potential $A^{\tau}\left(  t,x\right) $, fast varying in time, given by a centered Gaussian random field with a small time decorrelation rate $\tau>0$,
\begin{equation*}
    \mathbb{E}[A^{\tau}(t,x)\otimes A^{\tau}(s,y)]=e^{-|t-s|/\tau}Q(x,y),\quad
Q(x,y)=\sum_{k}\sigma_{k}(x)\otimes\sigma_{k}(y),
\end{equation*}
where $\sigma_{k}\in C^{\infty}(\mathbb{R}^{d})$ are vector fields satisfying suitable assumptions, cf. \cref{hp:noise} below. Equivalently, $(A^{\tau}_t)_{t\geq 0}$ is the stationary Ornstein-Uhlenbeck process solving
\begin{equation}
dA^{\tau}=-\frac{1}{\tau}A^{\tau}dt+\sqrt{\frac{2}{\tau}}\sum_{k\in I}%
\sigma_{k}dW^{k},\label{eq:OU}%
\end{equation}
where the $W^{k}$'s are independent 1D Brownian motions. The system
\eqref{eq:rmse}-\eqref{eq:OU} is an \textit{open} quantum system.

Inspired also by Klaus Hasselmann proposal developed for climate sciences
\cite{Has}, we look for a \textit{closed} quantum system, namely a stochastic
Schr\"{o}dinger equation only for the wave function $\psi\left(  t,x\right)  $
(also in the spirit of model reduction), with delta-correlated noise in
multiplicative (Stratonovich) form, which represents as close as possible the
system $\left(  \psi^{\tau}\left(  t,x\right)  ,A^{\tau}\left(  t,x\right)
\right)  $. Among the requirements, we want $L^{2}$-norm preservation for
$\psi\left(  t,x\right)  $. The new Stochastic Partial Differential Equation
(SPDE) should incorporate at the same time the averaging behavior of the fast
fluctuating field $A^{\tau}\left(  t,x\right)  $ and its fluctuations. The
philosophy is somewhat also similar to Dean-Kawasaki approach to particle
systems \cite{CorFis}, \cite{DirFeh}, \cite{DjuKre} where an SPDE is found
which closely represents the particle system, but the formulation and
techniques are very different. What we take from the known research on Hasselmann proposal \cite{Arn} and from Dean-Kawasaki research is a criterion for deciding when a model is close to another:\ they must have the same deterministic averaging limit and the same Gaussian fluctuations. Large deviations have also been proposed as a further selection criterion \cite{Arn,fehrman2023non} (see also \cite{BerDes}). We do not pursue this point of view further here, apart from the discussion in \autoref{ssec:ldpsse} relating the large deviation principle of \autoref{sec_ldp} to the original model; further comments are provided at the end of this introduction.

In order to devise the correct closed SPDE representing the system
\eqref{eq:rmse}-\eqref{eq:OU}, we first analyze the averaging principle and the Gaussian fluctuations of such a system. Our first result (averaging principle, \autoref{thm_averaging}) claims that in the fast oscillation limit $\tau\rightarrow0$ solutions converge to those of the deterministic Schr\"{o}dinger equation
\begin{equation*}
i\partial_{t}\bar{\psi}=(-\Delta+V)\bar{\psi},\tag{ASE}\label{eq:averaging}%
\end{equation*}
where the ``effective potential'' $V$, given by%
\begin{align}\label{def_potential}
    V(x)=\sum_{k\in I}|\sigma_{k}(x)|^{2} =\mathbb{E} \left[|A_t^{\tau}|^2\right] ,
\end{align}
is the local mean-square intensity of the magnetic vector potential. This is a rigorous formalization of an information known in the Physics literature listed above. The object defined by \eqref{def_potential} is the confining potential which should trap the particle. Due to our assumptions, $V$ has to decay at infinity, hence the confinement (in this rigorous setting) is only partial. A concrete class of divergence-free magnetic fields producing a radial, compactly supported effective potential is given in \cref{ex:magnetic}; depending on the choice of its profile, the resulting potential can model an annular barrier and thus provides a simple setting in which partial confinement and tunneling may occur.

The limit $\tau\to0$ presents a basic modeling tension. On the one hand, the fast field should disappear from the leading-order dynamics; on the other hand, for small but finite $\tau$, its cumulative effect remains visible in the fluctuations of the wavefunction. We shall then consider the fluctuation field
\begin{equation}
\xi^{\tau}=\frac{\psi^{\tau}-\bar{\psi}}{\sqrt{\tau}}%
\label{def:fluctuations1}%
\end{equation}
and establish (Gaussian fluctuations, \autoref{thm_fluctuations}) its
convergence to the Gaussian process solving the SPDE
\begin{equation*}
id\xi=\left(  -\Delta+V\right)  \xi dt+\sqrt{8}i\sum_{k\in I}\sigma_{k}%
\cdot\nabla\bar{\psi}dW_{t}^{k}+\bar{\psi}dB_{t}.\tag{GFSE}\label{eq:gaussianlimit}%
\end{equation*}
Let us stress the presence of a further stochastic forcing $B$, a Gaussian field independent of $W$ with covariance
\begin{equation}
\mathbb{E}\left[  B(t,x)B(s,y)\right]  =2(t\wedge s)R(x,y),\quad\text{with
}R(x,y)=\trace\left[  Q(x,y)Q(y,x)\right]  ,\label{second_covariance}%
\end{equation}
corresponding to the limiting law of 
\begin{align*}
 \sqrt{\frac{1}{\tau}}%
\int_0^{\cdot}:\mathrel{|A^\tau_s|^2}:ds \,\,  \text{where}\,\,       :\mathrel{|A^\tau_s|^2}:=|A^{\tau}_s|^{2}-V.
\end{align*}
In terms of Karhunen-Lo\`{e}ve decomposition,
\begin{equation*}
    B_{t}(x)=\sqrt{2}\sum_{k,j\in I}\sigma_{k}(x)\cdot\sigma_{j}(x)B_{t}^{k,j},
\end{equation*}
where $\{B_{t}^{k,j}\}_{k,j\in I}$ are real independent Brownian motions. The transport term $\sqrt{8}i\sum_{k\in I}\sigma_{k}\cdot\nabla\bar{\psi}%
dW_{t}^{k}$ obviously emerges from any formal computation, but the term $\bar{\psi}dB_{t}$ is less trivial. The two noises are
generated respectively by the linear and centered quadratic parts of the magnetic Hamiltonian, which belong to the first and second Wiener chaoses. The mechanism can be read from the equation for the fluctuations. Set
\begin{equation}\label{eq:WNapprox1}
    Z^{\tau}=\frac{A^{\tau}}{\sqrt{\tau}}.
\end{equation}
Expanding \eqref{eq:rmse} about $\bar\psi$ gives the exact identity
\begin{equation}\label{eq:fluctuations}
    i\partial_t\xi^{\tau}=-\Delta\xi^{\tau}+ 2iA^{\tau}\cdot\nabla \xi^{\tau}+ 2iZ^{\tau}\cdot\nabla \bar{\psi}+|A^\tau|^2\xi^{\tau}+\frac{\left(|A^\tau|^2-V\right)}{\sqrt{\tau}}\bar{\psi}.
\end{equation}
After integration in time, the two rescaled forcing terms in
\eqref{eq:fluctuations} converge jointly to the independent noises displayed in \eqref{eq:gaussianlimit}.  Thus the fluctuation theorem determines not only the size of the residual noise, but also its covariance and the differential and multiplicative forms through which it acts.

Based on these\ results, a good approximation of $\psi^{\tau}\left(
t,x\right)  $ is the complex-valued function
\begin{equation*}
    \psi^{\tau}\left(  t,x\right)  \sim\bar{\psi}\left(  t,x\right)  +\sqrt{\tau
}\xi^{\tau}\left(  t,x\right)  .
\end{equation*}
For certain practical purposes this may be satisfactory, but $\bar{\psi}\left(  t,x\right)  +\sqrt{\tau
}\xi^{\tau}\left(  t,x\right) $ is no longer a wave function and it does not satisfy any Schr\"{o}dinger type equation. 

It is here that the motivation of this work arises, following the intuitions of Hasselmann program. We introduce the new (closed) SPDE of Schr\"{o}dinger type
\begin{equation*}
id\Psi^{\tau}=(-\Delta+V)\Psi^{\tau}dt+i\sqrt{8\tau}\sum_{k\in I}\sigma
_{k}\cdot\nabla\Psi^{\tau}\circ dW^{k}+\sqrt{\tau}\Psi^{\tau}\circ
dB_{t}\tag{SSE}\label{eq:simplified_model}%
\end{equation*}
and prove, with quantitative rates (\autoref{thm_effective_model}), that in the fast-oscillation limit $\tau\rightarrow0$ its solutions converge to those of \eqref{eq:averaging} and its fluctuations
\begin{equation}
\Xi^{\tau}=\frac{\Psi^{\tau}-\bar{\psi}}{\sqrt{\tau}},%
\label{def:fluctuations2}%
\end{equation}
governed by the equation
\begin{multline}\label{eq:fluctuations2}
    i d \Xi^{\tau}=(-\Delta+V)\Xi^{\tau}dt+i\sqrt{8\tau}\sum_{k\in I}\sigma_k\cdot\nabla \Xi^{\tau}\circ dW^k +2\sqrt{2}i\sum_{k\in I}\sigma_k\cdot\nabla \bar{\psi} dW^k\\ +\sqrt{\tau} \Xi^{\tau}\circ dB_t+\bar{\psi}dB_t,
\end{multline}
converge to the Gaussian process \eqref{eq:gaussianlimit}. Moreover, the Stratonovich formulation in \eqref{eq:simplified_model} respects the skew-adjoint structure of the random generators.  Consequently,
\eqref{eq:simplified_model} preserves the $L^2$ norm exactly, cf.
\cref{prop:well_effective}, and hence evolves normalized initial data on the unit sphere.

These are the criteria mentioned at the beginning of the Introduction,
that we invoked\ to claim that equation \eqref{eq:simplified_model} is a good
model for system \eqref{eq:rmse}-\eqref{eq:OU}. Heuristically, notice that,
since also
\begin{equation*}
    \Psi^{\tau}\left(  t,x\right)  \sim\bar{\psi}\left(  t,x\right)  +\sqrt{\tau
}\xi^{\tau}\left(  t,x\right)
\end{equation*}
the difference $\psi^{\tau}-\Psi^{\tau}$ is infinitesimal in $\tau$ of higher order than $\sqrt{\tau}$, hence $\Psi^{\tau}$ is a better approximation of $\psi^{\tau}$ than just the deterministic function $\bar{\psi}$ (which differs from $\psi^{\tau}$ by order $\sqrt{\tau}$). Therefore the stochastic model \eqref{eq:simplified_model} it is more precise and informative than just the deterministic one \eqref{eq:averaging}. Finally, the structural properties of \eqref{eq:simplified_model} connect the reduction to a familiar construction in open quantum dynamics \cite{breuer2002theory,barchielli2009quantum}.  As explained formally in \cref{ssec:lindbladian}, it is a diffusive random-unitary
unraveling: the ensemble density operator associated with its solutions satisfies a Lindblad-type master equation.  We use this only as a structural interpretation, not as a microscopic system--bath derivation of the model.

Let us also give a brief overview of related averaging and fluctuation results in the literature. Averaging principles and their Gaussian fluctuation counterparts have a long history in the theory of multiscale stochastic dynamics, going back to the foundational work of Khasminskii \cite{khasminskii1968} on finite-dimensional stochastic differential equations, with the classical treatment of the diffusive fluctuation limit developed by Freidlin and Wentzell \cite[Chapter 7]{freidlinwentzell3}. We also refer to the seminal works of Pardoux and Veretennikov on Poisson equation techniques \cite{pardoux_1, pardoux2003poisson, pardoux_3}. In the infinite-dimensional setting of stochastic partial differential equations, the averaging principle and the structure of the associated Gaussian fluctuations have been systematically investigated over the last two decades; we refer, among many others, to the works of Cerrai and Freidlin \cite{cerrai1, cerrai2009normal, cerrai2009averaging, cerrai3} for parabolic fast-slow systems. More recently, fluctuation limits driven by non-Markovian or long-range correlated noise have been treated by Hairer and Li \cite{Hairer_1, li2022slow, hairer2022generating}. Averaging principles have also been studied for Schrödinger-type equations \cite{gao2017averaging, gao2018averaging}. All of these works, however, deal with either parabolic dynamics or a fast oscillating term entering as a zeroth-order, multiplicative coupling; in both cases, some regularizing structure is available and exploited in the analysis. The noise in \eqref{eq:rmse} and \eqref{eq:fluctuations} is instead of transport type: it enters through a first-order term that is singular for the Schrödinger dynamics, as it does not regularize the solution. Handling this term at low regularity requires commutator-type estimates in the spirit of Kato and Ponce \cite{kato1988commutator}, cf. \cref{lem:commutator}, together with renormalization arguments à la DiPerna–Lions \cite{diperna1989ordinary}, cf. \cref{lem_uniqueness_fluct_eq}.

The paper is organized as follows. In \autoref{preliminaries} we collect the well-posedness results
for \eqref{eq:rmse}, \eqref{eq:simplified_model} and the main results \autoref{thm_averaging}, \autoref{thm_fluctuations}, \autoref{thm_effective_model} for the two models. The analysis of the averaging and the Gaussian fluctuations for \eqref{eq:rmse} is presented in \autoref{sec:proof_main_original}, culminating in the proofs of \autoref{thm_averaging}, \autoref{thm_fluctuations}. The corresponding analysis for the proposed effective model \eqref{eq:simplified_model} is carried out in \autoref{sec:proof_effective}. In \autoref{sec_ldp} we prove a path-space large deviation principle for solutions of \eqref{eq:simplified_model} with Laplace scale $\tau$, extending the effective description from $O(\sqrt{\tau})$ Gaussian fluctuations to $O(1)$ atypical trajectories of the the proposed effective model. 
We conclude the paper with two appendices. In \autoref{sec:appendix_formal} some further formal justifications for \eqref{eq:simplified_model} as an effective model for \eqref{eq:rmse} are provided. More precisely, a formal calculation in \autoref{ssec:ldpsse} shows that the rate function of \eqref{eq:simplified_model} given in \autoref{sec_ldp} is the quadratic expansion, around the averaged minimizer, of the action associated with \eqref{eq:rmse}. Links with Lindbladian unravelings are also discussed in \autoref{ssec:lindbladian}. 
Finally, proofs of the well-posedness results mentioned in \autoref{preliminaries} are provided in \autoref{app:well-posed}.

\subsection*{Notation}
Throughout the paper, $T > 0$ denotes an arbitrary, fixed, positive time horizon and complex conjugates are denoted by the superscript $^*$. We use the convention
\begin{align*}
    \langle f,g\rangle=\int_{\R^d}f(x)g(x)^* dx.
\end{align*}
Thus the complex inner product is linear in its first argument.\\
For each $R>0$ we denote by $B^{L^2}_R$ the closed ball of $L^2(\R^d;\C)$ centered at $0$ with radius $R$ endowed with the strong topology and by $B^{L^2}_{R,w}$ the same closed ball endowed with the weak topology. Both are Polish: the first is closed in $L^2$, while the second is compact and metrizable in the weak topology.
We also denote by $\tilde{H}^s$, $s \in \R$, the weighted Sobolev space on $\R^d$, that is the closure of compactly supported smooth functions with values in $\C$ under the norm $\| \phi \|_{\tilde{H}^s} := \| \phi w \|_{H^s}$, where the weight $w$ is defined as $w(x) := (1+|x|^2)^{-d/2-1}$. 
Moreover, we introduce, for $s\in \R$, the Polish spaces
\begin{align*}
\tilde{H}^{s-} := \bigcap_{n \in \N} \tilde{H}^{-1/n},
\quad
\mbox{with distance }\,
d_{\tilde{H}^{s-}}(u,v):=\sum_{n\geq 1}\frac{1}{2^n}\left(1\wedge \norm{u-v}_{\tilde{H}^{s-1/n}}\right).
\end{align*}
In case of $s=0$ we simply write $\tilde{H}^{-}$ in place of $\tilde{H}^{0-}$.\\
If $H$ is a separable Hilbert space we denote by $C_w([0,T];H)$ the space of weakly continuous functions with values in $H.$ \\ Finally, let $(\Omega,\mathcal{F},(\mathcal{F}_t)_{t\ge0},\mathbb{P})$ be a filtered probability space. Throughout the paper, we assume that the filtration $(\mathcal{F}_t)_{t\ge0}$ satisfies the usual conditions, i.e., it is right-continuous and $\mathcal{F}_0$ contains all $\mathbb{P}$-null sets of $\mathcal{F}$. 
\begin{acknowledgements}
The research of F.F. is funded by the European Union (ERC, NoisyFluid, No. 101053472). E.L. has received funding from the European Research Council (ERC) under the European Union’s Horizon 2020 research and innovation programme (grant agreement No. 949981).
\end{acknowledgements}

\section{Preliminaries}\label{preliminaries}
 In this section we state some well-posedness results for the equations that are the main object of our analysis, namely \eqref{eq:rmse}, \eqref{eq:simplified_model},\ \eqref{eq:averaging}, and \eqref{eq:gaussianlimit}.

Let $\{W^k\}_{k\in I}$ and $\{B^{k,j}\}_{k,j\in I}$ be mutually independent families of standard real Brownian motions adapted to $(\mathcal{F}_t)_{t\ge0}$, we shall assume from now on the following regularity on the noise coefficients. 
\begin{hypothesis}\label[hypothesis]{hp:noise}
The index set $I$ is countable and $d\geq 2$ is an integer. The coefficients
$\{\sigma_k\}_{k\in I}$ are real-valued, divergence-free vector fields and
satisfy
\begin{align*}
    \sum_{k\in I}\|\sigma_k\|^2_{H^{5+\frac{d+\theta}{2}}}<+\infty
\end{align*}
for some $\theta>0.$ For each $\tau\in (0,1)$, the initial data of \eqref{eq:OU}, $A^{\tau}_0$,  is a centered
Gaussian $\mathcal{F}_0$-measurable random variable on
$H^{5+\frac{d+\theta}{2}}(\R^d;\R^d)$, independent of the future Brownian
increments, such that
\begin{align*}
    \expt{A^{\tau}_0(x)\otimes A^{\tau}_0(y)}=Q(x,y).
\end{align*}
\end{hypothesis}
The following example shows that \cref{hp:noise} accommodates magnetic fields whose mean-square intensity generates the localized confining potential discussed in the Introduction.
\begin{example}\label{ex:magnetic}
Let $g:[0,\infty)\rightarrow\mathbb{R}$ be a smooth function, with $g\left(
r\right)  =0$ for $r\in\left[  0,r_{0}\right]  $ and for $r\geq r_{1}$, with
$r_{1}>r_{0}>0$, and $g\left(  r\right)  >0$ in $\left(  r_{0},r_{1}\right)
$.
Set
\begin{equation*}
    I=\{(i,j):1\leq i<j\leq d\}.
\end{equation*}
For $(i,j)\in I$, define the real-valued vector field
\begin{align*}
    \sigma_{ij}(x)
  =\frac{g(|x|^2)}{\sqrt{d-1}}
    \bigl(-x_j e_i+x_i e_j\bigr),
  \qquad x=(x_1,\ldots,x_d)\in\R^d,
\end{align*}
where $e_1,\ldots,e_d$ is the canonical basis of $\R^d$. Every $\sigma_{ij}$ is divergence-free. Indeed, since $i\neq j$,
\begin{align*}
  \operatorname{div}\sigma_{ij}(x)
  &=\frac{1}{\sqrt{d-1}}
    \left(
      \partial_i\left[-x_jg(|x|^2)\right]
      +\partial_j\left[x_i g(|x|^2)\right]
    \right)
   =0.
\end{align*}
Moreover, the total intensity is radial:
\begin{align*}
  V(x)
  :=\sum_{1\leq i<j\leq d}|\sigma_{ij}(x)|^2
  &=\frac{g^2(|x|^2)}{d-1}
    \sum_{1\leq i<j\leq d}(x_i^2+x_j^2) \\
  &=|x|^2g^2(|x|^2)
   =:W(|x|^2).
\end{align*}
Consequently,
\begin{equation*}
    W(r)=r g^2(r),
\end{equation*}
so that $W=0$ on $[0,r_0]$ and on $[r_1,\infty)$, while $W>0$ on
$(r_0,r_1)$.    \\
Since $g(|x|^2)$ is smooth and compactly supported and the set $I$ is finite, this is an example of confining
potential satisfying our assumptions. We may modulate $W$ as we want in $\left(
r_{0},r_{1}\right)  $, by choosing $g$. Tunneling is of course possible, hence
the confinement is only partial. Choosing $r_{0}$ large, it may also be
considered as an approximation of a rectangular barrier as the one of
\cite{Moh}, and choosing $r_{0}$ very small as an approximation of the
circular barrier of \cite{RakYus}.
 
\end{example}

We use the shorthand
\begin{align*}
    Q(x)&:=Q(x,x)=\sum_{k\in I}\sigma_k(x)\otimes\sigma_k(x),\\
    R(x)&:=R(x,x)=\sum_{k,j\in I}(\sigma_k(x)\cdot\sigma_j(x))^2
\end{align*}
whenever the diagonal kernels occur as coefficients of differential or
multiplication operators.
Under the above assumptions the process
\begin{align}\label{eq:OU_prel}
       A^{\tau}_t=e^{-\frac{t}{\tau}}A^{\tau}_0+\sqrt{\frac{2}{{\tau}}}\sum_{k\in I}\int_{0}^t e^{-\frac{(t-s)}{\tau}}\sigma_k dW^k_s
   \end{align}
   is a stationary Gaussian process with paths in
   $C([0,T];H^{5+\frac{d+\theta}{2}})$ $\mathbb{P}-a.s.$ We refer to
   \cref{lem:stochastic_conv} for additional standard properties.
   We construct analytically weak $L^2$ solutions of the original model
   \eqref{eq:rmse} and the effective model \eqref{eq:simplified_model}
   according to the following definitions.
\begin{definition}\label[definition]{def:weaksoloriginal}
   Given $\psi_0\in L^2(\R^d)$, we say that a progressively measurable
   process $\psi^{\tau}$ with paths in $C_w([0,T];L^2)$
   $\mathbb{P}-a.s.$ is a weak solution of \eqref{eq:rmse}  with $\psi^{\tau}_0=\psi_0$ if the following
   identity holds for each $\phi\in \mathscr{S}(\R^d)$ and $t\in [0,T]$:
   \begin{align}\label{eq:weak:original}
       i\langle \psi^{\tau}_t,\phi\rangle-i\langle \psi_0,\phi\rangle&=-\int_0^t \langle\psi^{\tau}_s, \Delta\phi\rangle ds\notag\\ &-2i\int_0^t\langle \psi^{\tau}_s,A^{\tau}_s\cdot\nabla\phi\rangle ds+\int_0^t \langle|A^{\tau}_s|^2 \psi^\tau_s, \phi\rangle ds\quad \mathbb{P}-a.s.,
   \end{align}
   where $A^{\tau}_t$ is the stationary Ornstein–Uhlenbeck process defined by \eqref{eq:OU_prel}.
\end{definition}
\begin{definition}\label[definition]{def:weak_sol_effective}
    Given $\psi_0\in L^2(\R^d)$, we say that a progressively measurable process $\Psi^{\tau}$ with paths in $C_{w}([0,T];L^2)\ \mathbb{P}-a.s.$ is a weak solution of \eqref{eq:simplified_model} with $\Psi^{\tau}_0=\psi_0$ if there exists $\bar{R}>0$ such that
    \begin{align}\label{boundedness_spde}
        \sup_{t\in [0,T]}\|\Psi^{\tau}_t\|\leq \bar{R}\quad \mathbb{P}-a.s.
    \end{align} 
    and the following identity holds for each
    $\phi\in \mathscr{S}(\R^d)$ and $t\in [0,T]$:
   \begin{align}
       i\langle \Psi^{\tau}_t,\phi\rangle&=i\langle \psi_0,\phi\rangle+\int_0^t \langle\Psi^{\tau}_s, (-\Delta+V)\phi\rangle ds\notag\\
       &\quad+4i\tau \int_0^t\langle \Psi^{\tau}_s, \operatorname{div}\left(Q\nabla\phi\right)\rangle ds-i\tau \int_0^t\langle R\Psi^{\tau}_s,\phi\rangle ds\notag\\
       &\quad-i\sqrt{8\tau}\sum_{k\in I }\int_0^t
       \langle \Psi^{\tau}_s,\sigma_k\cdot\nabla\phi\rangle dW^k_s\notag\\
       &\quad+\sqrt{2\tau}\sum_{k,j\in I }\int_0^t
       \langle \Psi^{\tau}_s\sigma_k\cdot\sigma_j,\phi\rangle dB^{k,j}_s
       \quad \mathbb{P}-a.s.\label{eq:weak_effective}
   \end{align}
\end{definition}
\begin{remark}
  Assuming \cref{hp:noise}, it is easy to check that a weak solution of \eqref{eq:simplified_model} in the sense of \cref{def:weak_sol_effective} has paths in $C([0,T];H^{-2})\ \mathbb{P}-a.s.$
\end{remark}
\begin{remark}
    By density of Schwartz functions in $H^2(\R^d)$ and the regularity of $\psi^{\tau}$ (resp. $\Psi^{\tau}$) in \cref{def:weaksoloriginal} (resp. \cref{def:weak_sol_effective}) identity \eqref{eq:weak:original} (resp. \eqref{eq:weak_effective}) holds for each $\phi\in H^2(\R^d).$
\end{remark}
Existence and uniqueness of solutions for \eqref{eq:rmse} and
\eqref{eq:simplified_model} are provided by the following propositions.
\begin{proposition}\label[proposition]{prop:well_confinement}
    Assuming \cref{hp:noise} and $\psi_0\in L^2(\R^d)$, there exists a unique solution of \eqref{eq:rmse} in the sense of \cref{def:weaksoloriginal}. Moreover, 
    \begin{align}\label{eq:bound_rmse}
        \|\psi^{\tau}_t\|^2=\|\psi_0\|^2\quad \text{for all }t\geq 0, \quad  \mathbb{P}-a.s. 
    \end{align}
    In particular the solution $\psi^{\tau}$ has paths in $C([0,T];L^2)\ \mathbb{P}-a.s.$
\end{proposition}
\begin{proposition}\label[proposition]{prop:well_effective}
    Assuming \cref{hp:noise} and $\psi_0\in L^2(\R^d)$, there exists a unique solution of \eqref{eq:simplified_model} in the sense of \cref{def:weak_sol_effective}. Moreover, 
    \begin{align}\label{bounds_effective}
        \|\Psi^{\tau}_t\|^2=\|\psi_0\|^2 \quad \text{for all }t\geq 0, \quad  \mathbb{P}-a.s.
    \end{align}
   In particular, $\Psi^{\tau}$ has paths in $C([0,T];L^2)\ \mathbb{P}-a.s.$
\end{proposition}
The claims of \cref{prop:well_confinement}, \cref{prop:well_effective} follow by standard arguments via vanishing viscosity approximations and commutator estimates in the spirit of \cite{diperna1989ordinary}. For completeness, we provide a proof in \autoref{app:well-posed} below.

\begin{remark}
Equation \eqref{eq:simplified_model} preserves the $L^2$ norm. Thus normalized
initial data remain normalized. 
\end{remark}

The averaging equation \eqref{eq:averaging} and the fluctuation equation
\eqref{eq:gaussianlimit} are also well posed, as detailed by
\cref{prop:well_posed_averaging_limit,prop:well_posed_fluctuations}.
These results are direct consequences of semigroup theory for $e^{it\Delta}$ and the equivalence between weak and mild formulation \cite{pazy2012semigroups, da2014stochastic}, with $V$ treated as a linear perturbation. In particular, under \cref{hp:noise}, $V\in H^{5+\frac{d}{2}}\cap W^{5,\infty}$.   
\begin{proposition}\label[proposition]{prop:well_posed_averaging_limit}
   Assume \cref{hp:noise}. For each $\epsilon\in [-5,5]$ and
   $\psi_0\in H^\epsilon$, there exists a unique
   $\bar{\psi}\in C([0,T];H^{\epsilon})$ solving \eqref{eq:averaging} with $\bar{\psi}_0 =\psi_0 $ in the sense that
   the following identity holds for each $\phi\in \mathscr{S}(\R^d)$ and
   $t\in [0,T]$:
   \begin{align*}
       i\langle \bar\psi_t,\phi\rangle&=i\langle \psi_0,\phi\rangle+\int_0^t \langle\bar\psi_s, (-\Delta+V)\phi\rangle ds.
   \end{align*}
   Moreover, if $\epsilon\geq 0$, 
   \begin{align*}
       \|\bar{\psi}_t\|=\|\psi_0\|\quad \text{for all }t\geq 0.
   \end{align*}
\end{proposition}
\begin{proposition}\label[proposition]{prop:well_posed_fluctuations}
    Assume \cref{hp:noise} and $\psi_0\in L^2$, and let $\bar{\psi}$ be the
    unique solution in $C([0,T];L^2)$ given by
    \cref{prop:well_posed_averaging_limit}. For each
    $\epsilon\in [-5,-1]$, there exists a unique progressively measurable
    process ${\xi}$ with paths in
    $C([0,T];H^{\epsilon})\ \mathbb{P}$-a.s. solving
    \eqref{eq:gaussianlimit} with $\xi_0=0$ in the sense that
   the following identity holds for each $\phi\in \mathscr{S}(\R^d)$ and
   $t\in [0,T]$:
   \begin{align*}
       i\langle \xi_t,\phi\rangle&=\int_0^t \langle\xi_s, (-\Delta+V)\phi\rangle ds\\ &+2\sqrt{2}i\sum_{k\in I}\int_0^t \langle \sigma_k\cdot\nabla\bar{\psi}_s,\phi\rangle dW^k_s+\sqrt{2}\sum_{k,j\in I}\int_0^t \langle \sigma_k\cdot \sigma_j\bar{\psi}_s,\phi\rangle dB^{k,j}_s\quad \mathbb{P}-a.s.
   \end{align*} 
\end{proposition}

\subsection{Main Results}\label{sub:main_results}
We are now ready to prove our main results. The first two theorems  identify the averaging profile and the Gaussian fluctuations of $\psi^{\tau}$ solving \eqref{eq:rmse}.
\begin{theorem}[Averaging]\label{thm_averaging}
    Given $R>0$ and assuming \cref{hp:noise}, for each $\psi_0\in L^2$ such that $\|\psi_0\|\leq R$, the unique solution $\psi^{\tau}$ of \eqref{eq:rmse} given by \cref{prop:well_confinement} converges in probability in $C([0,T];B^{L^2}_{R,w})
    \cap C([0,T];\tilde{H}^-)$\footnote{Recall that $B^{L^2}_{R,w}$ is the closed ball of $L^2(\R^d;\C)$ centered at $0$ with radius $R$ endowed with the weak topology.} to the unique weak solution $\bar{\psi}$ of \eqref{eq:averaging} given by \cref{prop:well_posed_averaging_limit}.
\end{theorem}
\begin{remark}
    Let $d_{\psi_0}$ be a bounded metric inducing the topology of $C([0,T];B^{L^2}_{R,w})
    \cap C([0,T];\tilde{H}^-)$, Then, by \autoref{thm_averaging}, we also obtain, for each $p\geq 1$,
    \begin{align*}
       \lim_{\tau\rightarrow 0} \mathbb{E}\left[d_{\psi_0}(\psi^\tau,\bar{\psi})^p\right]=0.
    \end{align*}
\end{remark}
The proof of \cref{thm_averaging} is the content of \cref{sec:averaging}. Concerning the Gaussian fluctuations the following holds.
\begin{theorem}[Gaussian Fluctuations]\label{thm_fluctuations}
    Given $\psi_0\in L^2$ and assuming \cref{hp:noise}, let $\psi^{\tau}$ be the unique solution of \eqref{eq:rmse} given by \cref{prop:well_confinement} and let $\bar{\psi}$ be the unique weak solution of \eqref{eq:averaging} given by \cref{prop:well_posed_averaging_limit}. Then $\xi^{\tau}=\frac{\psi^\tau-\bar{\psi}}{\sqrt{\tau}}$ converges in law in the space $\bar Z_T$ defined in \eqref{topology_fluctuations} to the unique weak solution of \eqref{eq:gaussianlimit} given by \cref{prop:well_posed_fluctuations}.
\end{theorem}
The proof of \cref{thm_fluctuations} is the content of
\cref{sec:fluct_original}, which also specifies the topology of convergence
in law; see \cref{subsec:compactness_fluct} and
\eqref{topology_fluctuations}. This topology roughly corresponds to that of
weakly continuous functions on $[0,T]$ with values in $H^{-3}(\R^d)$.

Next we show quantitatively that the effective model
\eqref{eq:simplified_model} has the same averaging profile and Gaussian
fluctuations. This is the content of our final main result.
\begin{theorem}[Averaging and Fluctuations of the Effective Model]\label{thm_effective_model}
 Assume $\psi_0\in L^2$ and the validity of \cref{hp:noise}. Let $\Psi^{\tau}$ be the unique solution of \eqref{eq:simplified_model} given by \autoref{prop:well_effective} and let  $\Xi^{\tau}=\frac{\Psi^\tau-\bar{\psi}}{\sqrt{\tau}}$. Then for each $\theta_1\in [-2,0),\theta_2\in [-3,-\frac{5}{3})$ and $q\geq 1$
 \begin{align}\label{eq_rate_averaging}
     \mathbb{E}\left[\sup_{t\in [0,T]}\|\Psi_t^{\tau}-\bar{\psi}_t\|_{H^{\theta_1}}^q\right]&\lesssim \tau^{-\frac{q\theta_1}{4}},\\
     \mathbb{E}\left[\sup_{t\in [0,T]}\|\Xi_t^{\tau}-\xi_t\|_{H^{\theta_2}}^q\right]&\lesssim \begin{cases}\tau^{-\frac{q(3\theta_2+5)}{4}}\quad &\text{if }\theta_2\in [-2,-\frac{5}{3}),\\
        \tau^{-\frac{q(\theta_2+1)}{4}}\quad &\text{if }\theta_2\in [-3,-2].
     \end{cases}\label{eq_rate_fluctuctions}
 \end{align}
 Above we denote by $\bar{\psi}$ the unique weak solution of \eqref{eq:averaging} given by \cref{prop:well_posed_averaging_limit} and by $\xi$ the unique weak solution of \eqref{eq:gaussianlimit} given by \cref{prop:well_posed_fluctuations}, driven by the same Brownian motions as $\Psi^\tau$.
\end{theorem}

\section{Averaging and Fluctuations of \texorpdfstring{\eqref{eq:rmse}}{(RMSE)} }\label{sec:proof_main_original}
The section is split into two parts. We prove \cref{thm_averaging} in
\cref{sec:averaging} and \cref{thm_fluctuations} in
\cref{sec:fluct_original}, respectively.
\subsection{Averaging}\label{sec:averaging}
Since the proof of \cref{thm_averaging} is mainly a simplified version of that of \cref{thm_fluctuations} we only describe the main steps, referring to the arguments developed elsewhere in the paper for the missing details.

In order to show tightness of the laws of $\psi^{\tau}$, $\tau\in (0,1)$, we need a control on time increments. This is provided by the following lemma.
\begin{lemma}\label[lemma]{time:compactness_averaging}
        For each $\tau\in(0,1)$ and $0\leq s\leq t\leq T$,
    \begin{align*}
        \|\psi^{\tau}_t-\psi^{\tau}_s\|_{H^{-2}}&\leq  
        K|t-s|^{1/2}\|\psi_0\|
        \left(\int_0^T
        (1+\|A^{\tau}_r\|^4_{L^{\infty}})\,dr\right)^{1/2}
        \quad \mathbb{P}-a.s.
\end{align*}
for some positive non-random constant $K>0$ independent of $\tau,s,t.$
\end{lemma}
\begin{proof}
    Integrating equation \eqref{eq:rmse} between $s$ and $t$, we obtain by
    H\"older's inequality
    \begin{align*}
        \|\psi^{\tau}_t-\psi^{\tau}_s\|_{H^{-2}}
        &\lesssim |t-s|^{1/2}
        \sup_{r\in[0,T]}\|\psi^{\tau}_r\|\left(\int_0^T
        (1+\|A^{\tau}_r\|^4_{L^\infty})\,dr\right)^{1/2}
        \quad\mathbb{P}-a.s.
    \end{align*}
    The claim then follows from equation \eqref{eq:bound_rmse}.
\end{proof}
By \cref{time:compactness_averaging}, with $R=\|\psi_0\|$, the family of
laws of the processes
\begin{align*}
    \left(\psi^{\tau},\ A_0^\tau,\ W=\{W^k\}_{k\in I}\right)_{\tau\in (0,1)}
\end{align*}
in the Polish space
\begin{align*}
\left(C([0,T];B^{L^2}_{R,w})\cap C([0,T];\tilde{H}^-)\right)
\times H^{5+\frac{d+\theta}{2}}\times C([0,T];\mathbb{R}^I)
\end{align*}
is tight; this follows by arguing as in \cref{lem_compactness_viscous}; see
also
\cref{subsec:compactness_fluct}. Fix an arbitrary sequence
$\tau_n\rightarrow0$. By the Skorokhod representation theorem and standard
arguments (see for example \cite[Chapter 2]{flandoli2023stochastic}), after
passing to a subsequence that is not relabeled, we find an auxiliary
probability space, which for simplicity we continue to call
$(\Omega,\mathcal{F},\mathbb{P})$, and processes
\begin{align*}
(\hat{\psi}^{\tau_n}, {A}^{n}_0,W^n=(W^{k,n})_{k\in I}),\ ( \bar{\psi} ,\bar{A}_0,\bar{W}=(\bar{W}^{k})_{k\in I})
\end{align*}
such that, for each $n\in\N$,
\begin{align*}
(\hat{\psi}^{\tau_n}, {A}^{n}_0,W^n)\sim \left({\psi}^{\tau_n}, A^{\tau_n}_0,W\right),\end{align*}
and it holds
\begin{align}
 \hat{\psi}^{\tau_n}&\rightarrow \bar{\psi}\quad\text{in }C([0,T];B^{L^2}_{R,w})
    \cap C([0,T];\tilde{H}^-) \mathbb{P}-a.s.,\label{convergence_psi_n}\\
 A^{n}_0&\rightarrow \bar{A}_0 \quad\hspace{-0.14cm}\text{in }H^{5+\frac{d+\theta}{2}}\quad \mathbb{P}-a.s.,\notag\\
 W^n&\rightarrow \bar{W}\quad \hspace{-0.1cm}\text{in }C([0,T];\R^I)\quad \mathbb{P}-a.s.\notag
\end{align}
Moreover, $\mathcal{W}^n_t:=\sum_{k\in I}\sigma_k W^{k,n}$ is a Brownian motion with covariance structure $Q$ with respect to auxiliary probability filtration $(\mathcal{F}^n_t)_{t\geq 0}$, which is complete, right-continuous, and generated by $(\hat{\psi}^{\tau_n},A^{n}_0,W^n).$ Similarly, $\mathcal{W}_t:=\sum_{k\in I}\sigma_k \bar{W}^{k}$ is a Brownian motion with covariance structure $Q$ with respect to the complete right-continuous filtration $(\mathcal{F}_t)_{t\geq 0}$ generated by $(\bar{\psi}, \bar{A}_0,\bar{W}).$
Also, calling
\begin{align*}
    A^{n}_t:=e^{-\frac{t}{\tau_n}}A_0^{n}+\sqrt{\frac{2}{\tau_n}}\sum_{k\in I}\int_0^t e^{-\frac{t-s}{\tau_n}}\sigma_k dW^{k,n}_s,
\end{align*}
the following identity holds for each $\phi\in\mathscr{S}(\R^d)$ and
$t\in[0,T]$:
   \begin{align}\label{eq:weak_formulation_changed_space_prelimit_averaging}
       i\langle \hat{\psi}^{\tau_n}_t,\phi\rangle-i\langle\psi_0,\phi\rangle&=-\int_0^t \langle\hat{\psi}^{\tau_n}_s, \Delta\phi\rangle ds\notag\\ &-2i\int_0^t\langle \hat{\psi}^{\tau_n}_s,A^{n}_s\cdot\nabla\phi\rangle ds+\int_0^t \langle|A^{n}_s|^2 \hat{\psi}^{\tau_n}_s, \phi\rangle ds\quad \mathbb{P}-a.s.,
   \end{align}
namely, $\hat\psi^{\tau_n} \in C_{w}([0,T];L^2)$ is the unique adapted weak solution
of \eqref{eq:rmse} on the new probability space. In particular
$\|\hat{\psi}^{\tau_n}_t\|=\|\psi_0\|$ $\mathbb{P}-a.s.$, and in fact the processes
are continuous in $L^2$.

In order to identify the limit process $\bar{\psi}$ as the unique solution of \eqref{eq:averaging}, the following It\^o formula plays a role. 
\begin{lemma}\label{perturbed_test_function_averaging}
For each $n\in\mathbb N$ the following holds $\mathbb{P}-a.s.$: for every
$\phi\in \mathscr{S}(\R^d)$ and $t\in [0,T]$
\begin{align*}
    i\langle \hat{\psi}^{\tau_n}_t, \phi\rangle
    &+\frac{\tau_n}{2}\langle \hat{\psi}^{\tau_n}_t,|A^n_t|^2\phi\rangle
    -2i\tau_n\langle \hat{\psi}^{\tau_n}_t,A^n_t\cdot\nabla\phi\rangle\\
    &=i\langle \psi_0,\phi\rangle
    +\frac{\tau_n}{2}\langle \psi_0,|A^n_0|^2\phi\rangle
    -2i\tau_n\langle \psi_0,A^n_0\cdot\nabla\phi\rangle\\
    &\quad+\int_0^t\langle\hat{\psi}^{\tau_n}_s,(-\Delta+V)\phi\rangle ds+\frac{i\tau_n}{2}\int_0^t
    \langle\hat{\psi}^{\tau_n}_s,\Delta(|A^n_s|^2\phi)\rangle ds\\
    &\quad+2\tau_n\int_0^t
    \langle\hat{\psi}^{\tau_n}_s,\Delta(A^n_s\cdot\nabla\phi)\rangle ds-\tau_n\int_0^t
    \langle\hat{\psi}^{\tau_n}_s,A^n_s\cdot\nabla(|A^n_s|^2\phi)\rangle ds\\
    &\quad-\frac{i\tau_n}{2}\int_0^t
    \langle\hat{\psi}^{\tau_n}_s,|A^n_s|^4\phi\rangle ds+4i\tau_n\int_0^t
    \langle\hat{\psi}^{\tau_n}_s,A^n_s\cdot\nabla(A^n_s\cdot\nabla\phi)\rangle ds\\
    &\quad-2\tau_n\int_0^t
    \langle\hat{\psi}^{\tau_n}_s,|A^n_s|^2A^n_s\cdot\nabla\phi\rangle ds+\sqrt{2\tau_n}\sum_{k\in I}\int_0^t
    \langle\hat{\psi}^{\tau_n}_s,A^n_s\cdot\sigma_k\phi\rangle dW^{k,n}_s\\
    &\quad-i\sqrt{8\tau_n}\sum_{k\in I}\int_0^t
    \langle\hat{\psi}^{\tau_n}_s,\sigma_k\cdot\nabla\phi\rangle dW^{k,n}_s.
\end{align*}
\end{lemma}
\begin{proof}
    The claim and its proof are a simplified version of \cref{perturbed_test_function_fluctuations} below to which we refer for details.
\end{proof}

We are now ready to complete the proof of \cref{thm_averaging}.
\begin{proof}[Proof of \cref{thm_averaging}]
We want to pass to the limit in the equation satisfied by $\hat{\psi}^{\tau_n}$  by \cref{perturbed_test_function_averaging}.\\
Let $\phi $ as in the statement of \cref{perturbed_test_function_averaging} and $t\in [0,T]$. Thanks to \eqref{convergence_psi_n} we easily have
\begin{align*}
    i\langle \hat{\psi}^{\tau_n}_t, \phi\rangle&\rightarrow i\langle \bar{\psi}_t,\phi\rangle\quad\mathbb{P}-a.s.,\\
    \int_0^t \langle \hat{\psi}^{\tau_n}_s, (-\Delta+V)\phi\rangle ds&\rightarrow \int_0^t \langle \bar{\psi}_s, (-\Delta+V)\phi\rangle ds\quad\mathbb{P}-a.s.
\end{align*}
We are left to show that all the other terms go to zero as
$n\rightarrow+\infty$. By conservation of the $L^2$ norm,
\cref{lem:stochastic_conv}, and H\"older's inequality, every boundary or
deterministic remainder in \cref{perturbed_test_function_averaging} converges
to zero in $L^1(\Omega)$. For example,
\begin{align*}
 \mathbb E\left[
 \tau_n\left|\left\langle\hat\psi_t^{\tau_n},|A_t^n|^2\phi\right\rangle\right|
 \right]
 &\lesssim \tau_n\|\psi_0\|\|\phi\|
 \mathbb E\left[\|A_t^n\|_{H^{5+\frac{d+\theta}{2}}}^2\right]\longrightarrow0,\\
 \mathbb E\left[
 \tau_n\left|\int_0^t
 \left\langle\hat\psi_s^{\tau_n},
 \Delta(A_s^n\cdot\nabla\phi)\right\rangle ds\right|\right]
 &\lesssim \tau_n\|\psi_0\|\|\phi\|_{H^3}
 \mathbb E\left[\int_0^T\|A_s^n\|_{W^{2,\infty}}\,ds\right]\longrightarrow0.
\end{align*}
For the stochastic integrals, by \cite[Lemma 4.3]{bagnara2025no} it is
enough to show that
\begin{align*}
    \tau_n\sum_{k\in I }\int_0^T \left|\langle\hat{\psi}^{\tau_n}_s, A^n_s\cdot \sigma_k\phi\rangle\right|^2 ds&\rightarrow 0\quad\mbox{in probability,} \\
    \tau_n\sum_{k\in I }\int_0^T \left|\langle\hat{\psi}^{\tau_n}_s, \sigma_k\cdot\nabla\phi\rangle\right|^2 ds&\rightarrow 0\quad\mbox{in probability.} 
\end{align*}
We just show the first one, the second being analogous and simpler.
Indeed,
\begin{align*}
 &\tau_n\sum_{k\in I }\int_0^T
 \left|\langle\hat{\psi}^{\tau_n}_s,A^n_s\cdot\sigma_k\phi\rangle\right|^2ds\\
 &\qquad\lesssim
 \tau_n\|\psi_0\|^2\|\phi\|^2
 \left(\sum_{k\in I}\|\sigma_k\|_{L^\infty}^2\right)
 \int_0^T\|A_s^n\|_{L^\infty}^2ds,
\end{align*}
which converges to zero in $L^1(\Omega)$ by stationarity and \cref{lem:stochastic_conv}. Computations above
imply that for each $\phi \in \mathscr{S}$ there exists
$N\subset \Omega$, $\mathbb{P}(N)=0$, such that on $N^c$ the following holds
for each $t\in \mathbb Q\cap [0,T]$:
\begin{align}\label{eq_limit_averaging}
    i\langle \bar{\psi}_t, \phi\rangle
    &=i\langle\psi_0,\phi\rangle
      +\int_0^t \langle \bar{\psi}_s, (-\Delta+V)\phi\rangle ds.
\end{align}
By continuity in time of all the terms above and the existence of a
countable dense subset of $H^2$ consisting of functions in $\mathscr{S}$,
there exists a null set $N_0\subset\Omega$ such that, on $N_0^c$,
\eqref{eq_limit_averaging} holds for every $t\in[0,T]$ and $\phi\in H^2$.
In particular,  $\bar{\psi} \in C([0,T];H^{-2})$ $\mathbb{P}-a.s.$ Since
the initial sequence was arbitrary and the limit is deterministic,
uniqueness of the limit equation (cf.
\cref{prop:well_posed_averaging_limit}), gives convergence in probability for the full sequence
and completes the proof.
\end{proof}

\subsection{Fluctuations}\label{sec:fluct_original}
\subsubsection{Uniform Estimates}\label{sec:uniform_estimates_fluct}
In order to prove the required uniform bounds on $\xi^{\tau}$ we rely on the following Kato-Ponce type commutator estimate.
\begin{lemma}\label[lemma]{lem:commutator}
    Let $k\geq 0$ and $f\in H^{-k}(\R^d)$ be a complex valued function and $u\in H^{s}(\R^d;\R^d)$ for $s>\frac{d}{2}+k+2$, be a divergence free real vector field. Then
\begin{align*}
    \|u\cdot(I-\Delta)^{-k/2}\nabla f-(I-\Delta)^{-k/2}\left(u\cdot\nabla f\right)\|& \lesssim \|f\|_{H^{-k}}\|u\|_{H^s}.
\end{align*}
    In particular
    \begin{align*}
        |\langle(I-\Delta)^{-k}f,u\cdot\nabla f\rangle+\langle u\cdot\nabla f ,(I-\Delta)^{-k}f\rangle|&\lesssim \|u\|_{H^s}\|f\|_{H^{-k}}^2.
    \end{align*}
\end{lemma}
\begin{proof}
    Since $u$ is divergence free, the left-hand side can be rewritten as
    \begin{align*}
    &\langle(I-\Delta)^{-k}f,u\cdot\nabla f\rangle
    +\langle u\cdot\nabla f,(I-\Delta)^{-k}f\rangle\\
    &\quad=\langle (I-\Delta)^{-k/2}f,
    [u\cdot,(I-\Delta)^{-k/2}]\nabla f\rangle\\
    &\qquad+\langle [u\cdot,(I-\Delta)^{-k/2}]\nabla f,
    (I-\Delta)^{-k/2}f\rangle,
    \end{align*}
    Here and below we use the convention $[L,M]:=ML-LM$ for commutators.
    Therefore we have
    \begin{align*}
        |\langle(I-\Delta)^{-k}f,u\cdot\nabla f\rangle+\langle u\cdot\nabla f ,(I-\Delta)^{-k}f\rangle&\lesssim \|f\|_{H^{-k}}\|[u\cdot,(I-\Delta)^{-k/2}]\nabla f\|
    \end{align*}
    and we are left to estimate $\|[u\cdot,(I-\Delta)^{-k/2}]\nabla f\|$. By Plancherel's theorem it is enough to study $\|\mathscr{F}\left([u\cdot,(I-\Delta)^{-k/2}]\nabla f\right)\|$. For each $\eta\in \R^d$,
    \begin{align*}
      \mathscr{F}\left([u\cdot,(I-\Delta)^{-k/2}]\nabla f\right)(\eta)=i\int_{\R^d} \mathscr{F}u(\eta-\theta)\cdot\theta \mathscr{F}f(\theta)\left(m(\eta)-m(\theta)\right)d\theta, 
    \end{align*}
    where $m(\eta)=(1+|\eta|^2)^{-k/2}$. The mean-value theorem and
    Peetre's inequality give
   \begin{align*}
        |m(\eta)-m(\theta)|&\leq \frac{k|\zeta|} {(1+|\zeta|^2)^{(k+2)/2}}|\eta-\theta|\\ & \lesssim \frac{1}{(1+|\zeta|^2)^{(k+1)/2}}|\eta-\theta|\\ & \lesssim \frac{1}{(1+|\theta|^2)^{(k+1)/2}}|\eta-\theta|(1+|\theta-\zeta|^2)^{(k+1)/2}\\ &\leq \frac{(1+|\eta-\theta|^2)^{(k+2)/2}}{(1+|\theta|^2)^{(k+1)/2}},
    \end{align*}
    Therefore,
    \begin{align*}
        &\left|\mathscr{F}\left([u\cdot,(I-\Delta)^{-k/2}]
        \nabla f\right)(\eta)\right|\\
        &\quad\lesssim \int_{\R^d}|\mathscr{F}u(\eta-\theta)|
        (1+|\eta-\theta|^2)^{(k+2)/2}
        \frac{|\mathscr{F}f(\theta)|}{(1+|\theta|^2)^{k/2}}d\theta.
    \end{align*}
    Young's convolution inequality and Cauchy--Schwarz now imply
    \begin{align*}
        \|[u\cdot,(I-\Delta)^{-k/2}]\nabla f\|
        &\lesssim
        \left\|(1+|\cdot|^2)^{(k+2)/2}\mathscr{F}[u]\right\|_{L^1}
        \left\|(1+|\cdot|^2)^{-k/2}\mathscr{F}[f]\right\|_{L^2}\\
        &\lesssim \|u\|_{H^s}\|f\|_{H^{-k}}
        \left(\int_{\R^d}
        (1+|\eta|^2)^{k+2-s}\,d\eta\right)^{1/2},
    \end{align*}
    which yields the claim since $s>\frac{d}{2}+k+2$.
\end{proof}
Secondly, we prove some uniform controls on the singular terms $ Z^{\tau}\cdot\nabla\bar{\psi},\ \frac{\wick{|A^\tau|^2}\bar{\psi}}{\sqrt{\tau}} $.
\begin{proposition}\label[proposition]{prop:increments_stochastic}
    For each $\kappa\geq2$ and $0\leq s\leq t\leq T$,
    \begin{align*}
         \sup_{\tau \in (0,1)} \expt{\norm{\int_s^t Z^{\tau}_r\cdot\nabla\bar{\psi}_r dr}_{H^{-1}}^\kappa }&\lesssim |t-s|^{\frac{\kappa}{2}}\\
          \sup_{\tau \in (0,1)} \expt{\norm{\int_s^t \frac{\wick{|A^\tau_r|^2}\bar{\psi}_r}{\sqrt{\tau}}dr}_{L^2}^\kappa}&\lesssim |t-s|^{\frac{\kappa}{2}}.
    \end{align*}
\end{proposition}
\begin{proof}
    By Nelson's estimates, \cite{nelson1973free}, \cite[Theorem 3.50]{janson1997gaussian}, it is enough to consider the case $\kappa=2.$ For the first inequality we have, integrating by parts in the third step,  
    \begin{align*}
         &\expt{\norm{\int_s^t Z^{\tau}_r\cdot\nabla\bar{\psi}_r\,dr}_{H^{-1}}^2}\\
         &\quad=-\frac{1}{\tau}\sum_{k\in I}\int_s^t\int_s^t
         e^{-\frac{|r-r'|}{\tau}}
         \left\langle\nabla(I-\Delta)^{-1}
         \operatorname{div}(\sigma_k\bar\psi_r),
         \sigma_k\bar\psi_{r'}\right\rangle dr'\,dr\\
         &\quad=-\frac{1}{\tau}\sum_{k\in I}\int_s^t\int_s^r
         e^{-\frac{r-r'}{\tau}}
         \left\langle\sigma_k\cdot\nabla(I-\Delta)^{-1}
         \operatorname{div}(\sigma_k\bar\psi_r),
         \bar\psi_{r'}\right\rangle dr'\,dr\\
         &\qquad-\frac{1}{\tau}\sum_{k\in I}\int_s^t\int_s^r
         e^{-\frac{r-r'}{\tau}}
         \left\langle\bar\psi_{r'},
         \sigma_k\cdot\nabla(I-\Delta)^{-1}
         \operatorname{div}(\sigma_k\bar\psi_r)\right\rangle dr'\,dr\\
         &\quad\lesssim
         \frac{\|\bar{\psi}\|_{L^{\infty}_tL^2_x}^2}{\tau}
         \sum_{k\in I}\|\sigma_k\|_{L^\infty}^2
         \int_s^t\int_s^r e^{-\frac{r-r'}{\tau}}dr'\,dr
         \lesssim t-s,
    \end{align*}
    thanks to \cref{hp:noise}. For the second one, Wick's formula gives
    \begin{align*}
       &\expt{\norm{\int_s^t
       \frac{\wick{|A^\tau_r|^2}\bar{\psi}_r}{\sqrt{\tau}}\,dr}_{L^2}^2}\\
       &\quad=\frac{2}{\tau}\sum_{k,j\in I}\int_s^t\int_s^t
       e^{-\frac{2|r-r'|}{\tau}}
       \left\langle
       (\sigma_k\cdot\sigma_j)\bar\psi_r,
       (\sigma_k\cdot\sigma_j)\bar\psi_{r'}
       \right\rangle dr'\,dr\\
       &\quad\lesssim
       \|\bar\psi\|_{L^\infty_tL^2_x}^2
       \left(\sum_{k\in I}\|\sigma_k\|_{L^\infty}^2\right)^2
       \frac{1}{\tau}\int_s^t\int_s^t
       e^{-\frac{2|r-r'|}{\tau}}dr'\,dr
       \lesssim t-s.
    \end{align*}
    This proves the claim.
\end{proof}
For brevity, set
\begin{align*}
\mathcal N_T^\tau
&:=\sup_{t\in[0,T]}
\norm{\int_0^t Z^{\tau}_r\cdot\nabla\bar{\psi}_r\,dr}_{H^{-1}}+\sup_{t\in[0,T]}
\norm{\int_0^t
\frac{\wick{|A^\tau_r|^2}\bar{\psi}_r}{\sqrt{\tau}}\,dr}_{H^{-1}},\\[1mm]
\mathcal A_T^\tau
&:=\int_0^T\|A^\tau_r\|_{H^{5+\frac{d+\theta}{2}}}\,dr
+\int_0^T\|A^\tau_r\|_{W^{3,\infty}}^2\,dr.
\end{align*}


Now we are ready to control the behaviour of the $H^{-3}$ norm of $\xi^{\tau}$ uniformly in $\tau$.
\begin{proposition}\label[proposition]{prop:compactness_space_fluct}
    For each $\tau\in(0,1)$,
    \begin{align*}
        \sup_{t\in [0,T]}\|\xi^{\tau}_t\|_{H^{-3}}&\leq
        Ke^{K\mathcal A_T^\tau}
        \left(1+\int_0^T\|A^\tau_r\|^4_{W^{2,\infty}}dr\right)
        \mathcal N_T^\tau
         \quad \mathbb{P}-a.s.
\end{align*}
for some positive non-random constant $K>0$ independent of $\tau.$
\end{proposition}

\begin{proof}
    Let us introduce the auxiliary processes for each $\tau \in (0,1),$  
    \begin{align*}
        v^{\tau}_t&=2\int_0^t Z^{\tau}_r\cdot\nabla\bar{\psi}_r dr-i\int_0^t \frac{\wick{|A^\tau_r|^2}\bar{\psi}_r}{\sqrt{\tau}} dr\in C([0,T];H^{-1})\ \mathbb{P}-a.s.,\\
        \gamma^{\tau}_t&=\xi^{\tau}_t-v^{\tau}_t\in C([0,T];H^{-1})\ \mathbb{P}-a.s.
    \end{align*}
    Then, $\gamma^{\tau}$ satisfies the equation
    \begin{align*}
        i\partial_t\gamma^{\tau}
        &=-\Delta\gamma^{\tau}+2iA^{\tau}\cdot\nabla\gamma^{\tau}
        +|A^\tau|^2\gamma^{\tau}\\
        &\quad-\Delta v^{\tau}+2iA^{\tau}\cdot\nabla v^{\tau}
        +|A^\tau|^2v^{\tau},
    \end{align*}
    Since the right-hand side belongs to $C([0,T];H^{-3})$
    $\mathbb{P}-a.s.$, we can apply the Lions--Magenes lemma in $H^{-3}$.
    For each $t\in[0,T]$, simple manipulations and
    \cref{lem:commutator} give
    \begin{align*}
        \|\gamma^{\tau}_t\|_{H^{-3}}^2& \lesssim \int_0^t  \|\gamma^{\tau}_r\|_{H^{-3}}^2 \left(\|A^\tau_r\|_{H^{5+\frac{d+\theta}{2}}}+\|A^\tau_r\|^2_{W^{3,\infty}}\right) dr\\ & +\int_0^t \|\gamma^{\tau}_r\|_{H^{-3}}\|v^{\tau}_r\|_{H^{-1}}\left(1+\|A^\tau_r\|_{W^{2,\infty}}+\|A^\tau_r\|_{W^{1,\infty}}^2\right) dr.
    \end{align*}
    Therefore, Young's and Gr\"onwall's inequality yield
    \begin{align*}
        \|\gamma^{\tau}_t\|_{H^{-3}}^2
        &\leq K e^{K\mathcal A_T^\tau}
        \int_0^T \|v^{\tau}_r\|^2_{H^{-1}}
        \left(1+\|A^\tau_r\|^4_{W^{2,\infty}}\right) dr.
    \end{align*}
    The latter implies the claim by triangle inequality.
\end{proof}
Lastly we prove a control on time increments of $\xi^{\tau}_t$.
\begin{proposition}\label[proposition]{prop:compactness_time_fluct}
    For each $\tau\in(0,1)$ and $0\leq s\leq t\leq T$,
    \begin{align*}
        \|\xi^{\tau}_t-\xi^{\tau}_s\|_{H^{-5}}
        &\leq\norm{\int_s^t
        Z^{\tau}_r\cdot\nabla\bar{\psi}_r\,dr}_{H^{-3}}+\norm{\int_s^t
        \frac{\wick{|A^\tau_r|^2}\bar{\psi}_r}{\sqrt{\tau}}\,dr}_{H^{-3}}\\
        &\quad+
        K|t-s|^{1/2}e^{K\mathcal A_T^\tau}
        \left(1+\int_0^T\|A^{\tau}_r\|^6_{W^{3,\infty}}dr\right)
        \mathcal N_T^\tau
         \quad \mathbb{P}-a.s.
\end{align*}
for some positive non-random constant $K>0$ independent of $\tau,s,t.$
\end{proposition}


\begin{proof}
    Integrating equation \eqref{eq:fluctuations} between $s$ and $t$ we obtain
    \begin{align*}
        \|\xi^{\tau}_t-\xi^{\tau}_s\|_{H^{-5}}
        &\leq |t-s|^{1/2}
        \sup_{r\in[0,T]}\|\xi^{\tau}_r\|_{H^{-3}}
        \left(\int_0^T
        (1+\|A^{\tau}_r\|^4_{W^{3,\infty}})\,dr\right)^{1/2}\\
        &\quad+\norm{\int_s^t
        Z^{\tau}_r\cdot\nabla\bar{\psi}_r\,dr}_{H^{-3}}
        +\norm{\int_s^t
        \frac{\wick{|A^\tau_r|^2}\bar{\psi}_r}{\sqrt{\tau}}\,dr}_{H^{-3}}.
    \end{align*}
    The claim then follows from \cref{prop:compactness_space_fluct} and
    H\"older's inequality.
\end{proof}
\subsubsection{Compactness of the laws}\label{subsec:compactness_fluct}
Let us consider the Gelfand triple
\begin{align*}
    H^{-1}(\R^d)\hookrightarrow H^{-3}(\R^d)\hookrightarrow H^{-5}(\R^d),
\end{align*}
where $H^{-5}$ is identified with the dual of $H^{-1}$ using $H^{-3}$ as
the pivot space.
By \cite[Lemma C.1]{brzezniak2013existence}, there exists a separable Hilbert space $U$ such that the embedding $U\hookrightarrow H^{-1}(\R^d)$ is dense and compact. Therefore, denoting by $U'$ the dual of $U$ with respect to the Hilbert triple structure above, we have
\begin{align*}
    U\stackrel{c}{\hookrightarrow} H^{-1}(\R^d)\hookrightarrow H^{-3}(\R^d)\hookrightarrow H^{-5}(\R^d)\stackrel{c}{\hookrightarrow} U'.
\end{align*}
Let us denote by $C_w([0,T];H^{-3}(\R^d))$ the space of weakly continuous functions $u:[0,T]\rightarrow H^{-3}(\R^d) $ endowed with the weakest topology such that for all $h\in H^{-3}(\R^d)$ the maps
\begin{align*}
    F_h:C_w([0,T];H^{-3}(\R^d))\rightarrow C([0,T];\mathbb{C})\mbox{ defined as } F_h[u](t)=\langle u,h\rangle_{H^{-3}} 
\end{align*}
are continuous\footnote{In particular, $u_n\rightarrow u $ in $C_w([0,T];H^{-3}(\R^d))$ if and only if for each $h\in H^{-3}(\R^d)$
\begin{align*}
    \lim_{n\rightarrow +\infty}\sup_{t\in [0,T]}\|F_h[u_n](t)-F_h[u](t)\|=0.
\end{align*}}. For each $R>0$, let us denote by \begin{align*}
    \mathbb{B}_R=\{x\in H^{-3}(\R^d): \norm{x}_{H^{-3}}\leq R\}
\end{align*}
and $q$ the metric on $\mathbb{B}_R$ compatible with the weak topology and introduce the space
\begin{align*}
    C([0,T];\mathbb{B}_R)=\{f\in C_w([0,T];H^{-3}(\R^d)):\ \norm{f(t)}_{H^{-3}}\leq R\quad \forall t\in [0,T]\}.
\end{align*}
The space $C([0,T];\mathbb{B}_R)$ is metrizable and complete if endowed with the metric
\begin{align*}
d(u,v)=\sup_{t\in [0,T]}q(u(t),v(t))    
\end{align*}
see \cite{Brez_book_f}, 
\cite{brzezniak2013existence}.
Next, for $t\in [0,T]$, set \begin{align}\label{topology_fluctuations}
    \bar{Z}_t:=C([0,t];U')\cap C_w([0,t];H^{-3}(\R^d))
\end{align}
endowed with the supremum of the corresponding topologies, i.e. the coarsest topology on $\bar{Z}_t$ that is finer than both the topology of $C([0,t];U')$ and that of $C_w([0,t];H^{-3}(\R^d)).$
To introduce compact sets in $\bar{Z}_T$, we recall the following result;
see \cite[Lemma 2.1]{brzezniak2014existence}.
\begin{lemma}\label[lemma]{lemma_abstract_convergence_1}
 Let $u_k\in C_w([0,T];H^{-3}(\R^d))$ satisfy
 \begin{align*}
     \sup_k \sup_{t\in [0,T]}\norm{ u_k(t)}_{H^{-3}}\leq R,\quad u_k\rightarrow u\quad\text{in } C([0,T];U').
    \end{align*}
     Then $u_k,u \in C([0,T];\mathbb{B}_R)$ and
     $u_k\rightarrow u $ in $C([0,T];\mathbb{B}_R)$.
\end{lemma}
As a corollary of the result above, we obtain the following compactness criteria.
\begin{corollary}\label[corollary]{corollary_compactness}
Let $\mathcal{K}\subset \bar{Z}_T$. Then $\mathcal K$ is relatively compact in $\bar{Z}_T$ if
\begin{enumerate}
    \item $\sup_{u\in \mathcal{K}}\sup_{t\in [0,T]}\|u(t)\|_{H^{-3}}<+\infty$,
    \item $\lim_{\delta\rightarrow 0}\sup_{u\in \mathcal{K}}\sup_{s,t\in [0,T], |s-t|\leq \delta}\|u(t)-u(s)\|_{U'}=0.$
\end{enumerate}
In particular, for each $R>0$, $s\in(0,1)$, and $p\in(1,\infty)$ such
that $sp>1$, the set
\begin{align*}
    X_{R,s,p}:=\big\{f\in &C_w([0,T];H^{-3}(\R^d))
    \cap W^{s,p}(0,T;H^{-5}(\R^d)): \\
     &\sup_{t\in [0,T]}\|f(t)\|_{H^{-3}}
     +\|f\|_{W^{s,p}(0,T;H^{-5})}\leq R\big\},
\end{align*}
is relatively compact in $\bar{Z}_T.$
\end{corollary}
\begin{proof}
Let $R=\sup_{u\in \mathcal{K}}\sup_{t\in [0,T]}\|u(t)\|_{H^{-3}}.$
Therefore, compactness in $\bar{Z}_T$ is equivalent to compactness in the
Polish space
$\tilde{Z}_T=C([0,T];U')\cap C([0,T];\mathbb{B}_R).$
In particular, compactness is equivalent to sequential compactness. Let
$(f_k)_{k\in \N}\subset \mathcal{K}$. Since the embedding
$H^{-3}(\R^d)\hookrightarrow U'$ is compact, the Ascoli--Arzelà theorem implies
the existence of a non-relabeled subsequence and
$f\in C([0,T];U')$ such that
\begin{align*}
 f_k\rightarrow f\quad \mbox{ in }   C([0,T];U').
\end{align*}
Then \cref{lemma_abstract_convergence_1} implies that also 
\begin{align*}
    f_k\rightarrow f\quad \text{in }C([0,T];\mathbb{B}_R).
\end{align*}
This completes the proof of the first claim. The second one then follows by
the Sobolev embedding
$W^{s,p}(0,T;H^{-5})\hookrightarrow C^{s-\frac1p}(0,T;U')$.
\end{proof}
In order to move forward we need to show tightness of the laws of $\xi^{\tau}.$ This is the content of the following lemma.
\begin{lemma}\label[lemma]{compactness_fluctuations}
    Let $s\in(0,1/2)$ and $p\in(2,\infty)$ satisfy $sp>1$. Then for each
    $\epsilon>0$ there is $R:=R(s,p,\epsilon,T,\psi_0)$ such that, for
    each $\tau\in(0,1)$,
    \begin{align*}
        \mathbb{P}\left(\xi^{\tau}\notin \overline{X_{R,s,p}}\right)\leq \epsilon,
    \end{align*}
    where we denote by $\overline{X_{R,s,p}}$ the closure of ${X_{R,s,p}}$ with respect to the topology in $\bar{Z}_T.$
\end{lemma}
\begin{proof}
    The following chain of inequalities trivially holds by Chebyshev’s inequality
    \begin{align}\label{tightness_0_fluct}
        \mathbb{P}\big(\xi^{\tau}&\notin \overline{X_{R,s,p}}\big)\leq \mathbb{P}\left(\xi^{\tau}\notin {X_{R,s,p}}\right)\notag\\ & \leq \mathbb{P}\left(\sup_{t\in [0,T]}\|\xi^{\tau}_t\|_{H^{-3}}>\frac{R}{2} \right)+\mathbb{P}\left(\|\xi^{\tau}\|_{W^{s,p}(0,T;H^{-5})}>\frac{R}{2} \right)\notag\\ & \leq \frac{\expt{\log\left(1+\sup_{t\in [0,T]}\|\xi^{\tau}_t\|_{H^{-3}}\right)}}{\log(1+\frac{R}{2})} + \frac{\expt{\log\left(1+\|\xi^{\tau}\|_{W^{s,p}(0,T;H^{-5})}\right)}}{\log(1+\frac{R}{2})}.
    \end{align}
    By \cref{prop:increments_stochastic}, hypercontractivity, and the
    Garsia--Rodemich--Rumsey inequality, the two integrated forcing processes
    have uniformly bounded moments of their suprema. Hence
    \cref{prop:compactness_space_fluct}, basic properties of the logarithm,
    and Young's inequality give
    \begin{align}\label{tightness_1_fluct}
        \expt{\log\left(1+\sup_{t\in [0,T]}\|\xi^{\tau}_t\|_{H^{-3}}\right)}& \lesssim\expt{1+\int_0^T\|A^\tau_t\|^8_{H^{5+\frac{d+\theta}{2}}}dt}\notag\\ & +\expt{\sup_{t\in [0,T]}\norm{\int_0^t Z^{\tau}_r\cdot\nabla\bar{\psi}_r dr}_{H^{-1}}^2}\notag\\ &+\expt{\sup_{t\in [0,T]}\norm{\int_0^t \frac{\wick{|A^\tau_r|^2}\bar{\psi}_r}{\sqrt{\tau}} dr}^2_{H^{-1}}}
        \notag\\ & \lesssim 1 
    \end{align}
    uniformly in $\tau\in (0,1)$ due to \cref{lem:stochastic_conv} and \cref{prop:increments_stochastic}. In order to bound the second term uniformly in $\tau\in (0,1)$ we argue similarly. Indeed, by \cref{prop:compactness_time_fluct} 
    \begin{align*}
        \|\xi^{\tau}\|_{W^{s,p}(0,T;H^{-5})}^p&\lesssim\int_0^T\int_0^t \frac{\sup_{r\in [0,T]}\|\xi^{\tau}_r\|_{H^{-3}}^p\left(1+\int_0^T\|A^{\tau}_r\|^{2p}_{W^{3,\infty}}dr\right)}{|t-t'|^{1+p(s-1/2)}}dt' dt\\
        &\quad+\int_0^T\int_0^t \frac{\norm{\int_{t'}^t Z^{\tau}_r\cdot\nabla\bar{\psi}_r dr}_{H^{-3}}^p+\norm{\int_{t'}^t \frac{\wick{|A^\tau_r|^2}\bar{\psi}_r}{\sqrt{\tau}} dr}_{H^{-3}}^p}{|t-t'|^{1+ps}}dt' dt\\
        &\lesssim \sup_{r\in [0,T]}\|\xi^{\tau}_r\|_{H^{-3}}^p\left(1+\int_0^T\|A^{\tau}_r\|^{2p}_{W^{3,\infty}}dr\right)\\
        &\quad+\int_0^T\int_0^t \frac{\norm{\int_{t'}^t Z^{\tau}_r\cdot\nabla\bar{\psi}_r dr}_{H^{-3}}^p+\norm{\int_{t'}^t \frac{\wick{|A^\tau_r|^2}\bar{\psi}_r}{\sqrt{\tau}} dr}_{H^{-3}}^p}{|t-t'|^{1+ps}}dt' dt.
    \end{align*}
    Therefore, by basic properties of logarithm,
    \begin{align*}
        \expt{\log\left(1+\|\xi^{\tau}\|_{W^{s,p}(0,T;H^{-5})}\right)}& \lesssim 1+ \expt{\log\left(1+\sup_{r\in [0,T]}\|\xi^{\tau}_r\|_{H^{-3}}\right)}\\ & +\expt{\int_0^T\|A^{\tau}_r\|^{2p}_{W^{3,\infty}}dr}^{1/p}\\ & + \expt{\int_0^T\int_0^t \frac{\norm{\int_{t'}^t Z^{\tau}_r\cdot\nabla\bar{\psi}_r dr}_{H^{-3}}^p}{|t-t'|^{1+ps}}dt' dt}^{1/p}\\ &
        +\expt{\int_0^T\int_0^t
        \frac{\norm{\int_{t'}^t
        \frac{\wick{|A^\tau_r|^2}\bar{\psi}_r}{\sqrt{\tau}}\,dr}_{H^{-3}}^p}
        {|t-t'|^{1+ps}}dt' dt}^{1/p}
    \end{align*}
    and the first term can be bounded by \eqref{tightness_1_fluct}, the second one by \cref{lem:stochastic_conv}, while for the last two, by \cref{prop:increments_stochastic}, it holds
    \begin{align*}
        \expt{\int_0^T\int_0^t
        \frac{\norm{\int_{t'}^t
        Z^{\tau}_r\cdot\nabla\bar{\psi}_r\,dr}_{H^{-3}}^p}
        {|t-t'|^{1+ps}}dt'\,dt}
        &\lesssim1,\\
        \expt{\int_0^T\int_0^t
        \frac{\norm{\int_{t'}^t
        \frac{\wick{|A^\tau_r|^2}\bar{\psi}_r}{\sqrt{\tau}}\,dr}_{H^{-3}}^p}
        {|t-t'|^{1+ps}}dt'\,dt}
        &\lesssim1.
    \end{align*}
    In conclusion we also showed
    \begin{align}\label{tightness_2_fluct}
         \expt{\log\left(1+\|\xi^{\tau}\|_{W^{s,p}(0,T;H^{-5})}\right)}&\lesssim 1.
    \end{align}
    Combining \eqref{tightness_0_fluct}, \eqref{tightness_1_fluct}, \eqref{tightness_2_fluct} the claim follows up to choosing $R$ sufficiently large.
\end{proof}
Thanks to \cref{compactness_fluctuations}, \cref{prop:increments_stochastic} and \cite[Theorem 23.7]{kallenberg2002}
the family of laws 
\begin{align*}
    \left(\xi^{\tau},\int_0^\cdot Z^{\tau}_r\cdot\nabla\bar{\psi}_r dr,\int_0^\cdot \frac{\wick{|A^\tau_r|^2}\bar{\psi}_r}{\sqrt{\tau}}dr,A^\tau_0,W=(W^{k})_{k\in I}\right)
\end{align*}
is tight in the space
\begin{align*}
    \mathcal{Z}:=\bar{Z}_T\times  C([0,T];H^{-1})\times C([0,T];L^{2})\times H^{5+\frac{d+\theta}{2}}\times  C([0,T];\R^I).
\end{align*}
Therefore, by the Jakubowski version of the Skorokhod representation theorem
\cite{jakubowski1998almost,Brzezniak_skoro}, and standard arguments (see, for
example, \cite[Chapter 2]{flandoli2023stochastic} or
\cite[Section 5.3]{brzezniak2013existence}), we find a sequence
$\tau_n\rightarrow 0$,
an auxiliary probability space, which for simplicity we continue to call
$(\Omega,\mathcal{F},\mathbb{P})$, and processes
\begin{align*}
(\xi^{\tau_n},\Theta^n, \Gamma^n, A^{n}_0,W^n=(W^{k,n})_{k\in I}),\ (\xi,\Theta, \Gamma, A_0,\bar{W}=(\bar{W}^{k})_{k\in I})
\end{align*}
such that, for each $n\in\N$,
\begin{align*}
(\xi^{\tau_n},\Theta^n, \Gamma^n, A^{n}_0,W^n)\sim \left(\xi^{\tau_n},\int_0^\cdot Z^{\tau_n}_r\cdot\nabla\bar{\psi}_r dr,\int_0^\cdot \frac{\wick{|A^{\tau_n}_r|^2}\bar{\psi}_r}{\sqrt{\tau_n}}dr,A^{\tau_n}_0,W\right),\end{align*}
and it holds
\begin{align}
 \xi^{\tau_n}&\rightarrow \xi\quad\text{ in }\bar{Z}_T\quad \mathbb{P}-a.s.,\label{convergence_xi}\\
 \Theta^n&\rightarrow \Theta\quad\hspace{-0.1cm}\text{ in } C([0,T];H^{-1})\quad \mathbb{P}-a.s.,\label{convergence_theta}\\
 \Gamma^n&\rightarrow \Gamma\quad\hspace{-0.05cm}\text{ in } C([0,T];L^2)\quad \mathbb{P}-a.s.,\label{convergence_gamma}\\
 A^{n}_0&\rightarrow A_0 \quad\hspace{-0.23cm}\text{ in }H^{5+\frac{d+\theta}{2}}\quad \mathbb{P}-a.s.,\notag\\
 W^n&\rightarrow \bar{W} \quad\hspace{-0.2cm}\text{ in }C([0,T];\R^I)\quad \mathbb{P}-a.s.\notag
\end{align}
Moreover calling $\mathcal{W}^n_t:=\sum_{k\in I}\sigma_k W^{k,n}$, it is a Brownian motion with covariance structure $Q$ with respect to the complete right-continuous filtration $(\mathcal{F}^n_t)_{t\geq 0}$ generated by $(\xi^{\tau_n},\Theta^n, \Gamma^n, A^{n}_0,W^n),$ while $\mathcal{W}_t:=\sum_{k\in I}\sigma_k \bar{W}^{k}$ is a Brownian motion with covariance structure $Q$ with respect to the complete right-continuous filtration $(\mathcal{F}_t)_{t\geq 0}$ generated by $(\xi,\Theta, \Gamma, A_0,\bar{W}).$
Also, calling
\begin{align*}
    A^{n}_t:=e^{-\frac{t}{\tau_n}}A_0^{n}+\sqrt{\frac{2}{\tau_n}}\sum_{k\in I}\int_0^t e^{-\frac{t-s}{\tau_n}}\sigma_k dW^{k,n}_s,\ Z^n_t:=\frac{A^{n}_t}{\sqrt{\tau_n}},
\end{align*}
it holds
\begin{align} \label{eq:TG_n}
    \Theta^n_t=\int_0^t  Z^n_s\cdot\nabla\bar{\psi}_s ds,\quad \Gamma^n_t=\int_0^t  \frac{\wick{|A^n_s|^2}\bar{\psi}_s}{\sqrt{\tau_n}} ds\quad \mathbb{P}-a.s.
\end{align}
and for each $\phi\in \mathscr{S}(\R^d),\ t\in [0,T]$ 
   \begin{align}\label{eq:weak_formulation_changed_space_prelimit}
       i\langle \xi^{\tau_n}_t,\phi\rangle&=-\int_0^t \langle\xi^{\tau_n}_s, \Delta\phi\rangle ds\notag\\ &-2i\int_0^t\langle \xi^{\tau_n}_s,A^{n}_s\cdot\nabla\phi\rangle ds+2i\int_0^t\langle Z^{n}_s\cdot\nabla\bar{\psi}_s,\phi\rangle ds\notag\\ &+\int_0^t \langle|A^{n}_s|^2 \xi^{\tau_n}_s, \phi\rangle ds+\int_0^t \left\langle\frac{\wick{|A^n_s|^2}\bar{\psi}_s}{\sqrt{\tau_n}}, \phi\right\rangle ds\quad \mathbb{P}-a.s.,
   \end{align}
namely, the $\mathcal{F}^n_t$-adapted process $\xi^{\tau_n}$, with paths in
$C_w([0,T];H^{-3})$, is a weak solution of
\eqref{eq:fluctuations}.

\subsubsection{Passage to the limit}\label{subsec:limit_fluct}
We first show uniqueness of weak solutions of \eqref{eq:fluctuations} in the class of adapted processes with paths in $C_w([0,T];H^{-3})$. Therefore, also in the auxiliary probability space obtained by Skorokhod's representation theorem,
$\xi^{\tau_n}=\frac{\psi^{\tau_n}-\bar{\psi}}{\sqrt{\tau_n}}$ where $\psi^{\tau_n}$ is the unique solution of \eqref{eq:rmse} in the auxiliary probability space given by \cref{prop:well_confinement} and $\bar{\psi}$ is the unique solution of \eqref{eq:averaging} given by \cref{prop:well_posed_averaging_limit}. 
In particular, for every $n\in\mathbb N$,
$\xi^{\tau_n}\in C([0,T];L^2)$ $\mathbb{P}-a.s.$, and the estimates of
\cref{prop:compactness_space_fluct} continue to hold.
\begin{lemma}\label[lemma]{lem_uniqueness_fluct_eq}
	   There exists a unique solution of \eqref{eq:fluctuations} in the
	   class of adapted processes with paths in
	   $C_w([0,T];H^{-3})\ \mathbb{P}$-a.s.
\end{lemma}
\begin{proof}
Since we are working at $n$ fixed, we drop the dependence on $n$ to save notation.
    By linearity it is enough to show that every adapted process with paths
    in $C_w([0,T];H^{-3})\ \mathbb{P}$-a.s. that solves, in the analytically
    weak sense, the PDE with random coefficients
    \begin{align*}
        \begin{cases}
            i\partial_t\xi&=-\Delta\xi+ 2iA\cdot\nabla \xi+|A|^2\xi\\
            \xi_0&=0
        \end{cases}
    \end{align*}
    is identically $0$. Also, setting
    $v=(I-\Delta)^{-3/2}\xi\in C_w([0,T];L^2)$ $\mathbb{P}-a.s.$, it is
    enough to show uniqueness in this class for the random PDE
    \begin{align}\label{eq:aux_fluct}
        \begin{cases}
            i\partial_t v=-\Delta v+2iA\cdot\nabla v
            +2i[A\cdot,(I-\Delta)^{-3/2}]
            \nabla(I-\Delta)^{3/2}v\\
            \hspace{27mm}
            +(I-\Delta)^{-3/2}
            \left(|A|^2(I-\Delta)^{3/2}v\right),\\
            v_0=0.
        \end{cases}
    \end{align}
     Let $\chi$ be a standard smooth mollifier and let, for each $\epsilon>0$, $\chi_{\epsilon}(x)=\frac{\chi(x/\epsilon)}{\epsilon^d}$. Then choosing $\chi_\epsilon(x-\cdot)$ as a test function in \eqref{eq:aux_fluct} we obtain
\begin{align*}
	     \begin{cases}
	            i\partial_t v^{\epsilon}=-\Delta v^{\epsilon}
	            +2iA\cdot\nabla v^{\epsilon}
	            +2i[A\cdot\nabla,\chi_{\epsilon}\ast]v\\
	            \hspace{20mm}
	            +2i\left([A\cdot,(I-\Delta)^{-3/2}]
	            \nabla(I-\Delta)^{3/2}v\right)\ast\chi_{\epsilon}\\
	            \hspace{20mm}
	            +\left((I-\Delta)^{-3/2}
	            \left(|A|^2(I-\Delta)^{3/2}v\right)\right)
	            \ast\chi_{\epsilon},\\
	            v^{\epsilon}_0=0.
	        \end{cases}
\end{align*}
    where $v^{\epsilon}=v\ast \chi_{\epsilon}$ is smooth, the notation $[A\cdot\nabla,\chi_{\epsilon}\ast]$ stands for the commutator and the equation is satisfied in classical sense. Therefore, calling 
    \begin{align*}
        g^{1,\epsilon}&=2\left([A\cdot,(I-\Delta)^{-3/2}]\nabla(I-\Delta)^{3/2}v\right)\ast \chi_{\epsilon},\\
        g^{2,\epsilon}&=\left((I-\Delta)^{-3/2}\left(|A|^2(I-\Delta)^{3/2}v\right)\right)\ast\chi_{\epsilon},
    \end{align*}
    for every $t\in[0,T]$, $\mathbb{P}-a.s.$,
    \begin{align}\label{eq:mollified_estimate}
        \|v^{\epsilon}_t\|^2&\leq
        2\int_0^t\left|
        \langle [A_s\cdot\nabla,\chi_{\epsilon}\ast]v_s,v^{\epsilon}_s\rangle
        +\langle v^{\epsilon}_s,[A_s\cdot\nabla,\chi_{\epsilon}\ast]v_s\rangle
        \right|ds\\
        &\quad+2\int_0^t \|v^{\epsilon}_s\|
        \left(\|g^{1,\epsilon}_s\|+\|g^{2,\epsilon}_s\|\right)ds.
    \end{align}
    By \cref{lem:commutator}, the regularity of $A$, and standard properties of convolutions
    \begin{align*}
        \int_0^t \|v^{\epsilon}_s\|\left(\|g^{1,\epsilon}_s\|+\|g^{2,\epsilon}_s\|\right)ds& \lesssim \int_0^t \|v_s\|^2\left(\|A_s\|_{H^{5+\frac{d+\theta}{2}}}+\|A_s\|_{W^{3,\infty}}^2\right)ds
    \end{align*}
    Secondly, by \cite[Lemma 4.2]{butori2026background}, for almost every
    $s\in[0,T]$,
   \begin{align*}
       \lim_{\epsilon\rightarrow 0}\|[A_s\cdot\nabla,\chi_{\epsilon}\ast]v_s\|&=0,\\
       \sup_{\epsilon\in (0,1)}\|[A_s\cdot\nabla,\chi_{\epsilon}\ast]v_s\|&\lesssim \|A_s\|_{H^{1+\frac{d+\theta}{2}}}\|v_s\|.
   \end{align*}
	   Therefore, by the dominated convergence theorem and standard properties
	   of convolutions, letting $\epsilon\rightarrow 0$ in
	   \eqref{eq:mollified_estimate}, there exists a null set
	   $N\subset\Omega$ such that, on $N^c$, for each $t\in[0,T]$,
   \begin{align*}
       \|v_t\|^2&\lesssim \int_0^t
       \left(\|A_s\|_{H^{5+\frac{d+\theta}{2}}}
       +\|A_s\|_{W^{3,\infty}}^2\right)\|v_s\|^2ds.
   \end{align*}
   The latter implies the claim by Gr\"onwall's lemma.
\end{proof}
In order to identify the limit $\xi$ as the unique solution of \eqref{eq:gaussianlimit}, the following It\^o formula plays a role.
\begin{lemma}\label[lemma]{perturbed_test_function_fluctuations}
For each $n\in \N$ the following holds $\mathbb{P}-a.s.$: for every $\phi\in \mathscr{S}(\R^d)$ and $ t\in [0,T]$ 
\begin{align*}
    i\langle \xi^{\tau_n}_t,\phi\rangle
    &+\frac{\tau_n}{2}\langle \xi^{\tau_n}_t,|A^{n}_t|^2\phi\rangle
    -2i\tau_n\langle \xi^{\tau_n}_t,A^{n}_t\cdot\nabla\phi\rangle\\
    &=\int_0^t\langle \xi^{\tau_n}_s,(-\Delta+V)\phi\rangle ds
    +2i\int_0^t\langle Z^{n}_s\cdot\nabla\bar{\psi}_s,\phi\rangle ds\\
    &\quad+\int_0^t
    \left\langle\frac{\wick{|A^n_s|^2}}{\sqrt{\tau_n}}\bar{\psi}_s,
    \phi\right\rangle ds+\frac{i\tau_n}{2}\int_0^t
    \langle\xi^{\tau_n}_s,\Delta(|A^{n}_s|^2\phi)\rangle ds\\
    &\quad+2\tau_n\int_0^t
    \langle\xi^{\tau_n}_s,\Delta(A^{n}_s\cdot\nabla\phi)\rangle ds-\tau_n\int_0^t
    \langle\xi^{\tau_n}_s,A^n_s\cdot\nabla(|A^{n}_s|^2\phi)\rangle ds\\
    &\quad-\frac{i\tau_n}{2}\int_0^t
    \langle\xi^{\tau_n}_s,|A^{n}_s|^4\phi\rangle ds+4i\tau_n\int_0^t
    \langle\xi^{\tau_n}_s,
    A^{n}_s\cdot\nabla(A^{n}_s\cdot\nabla\phi)\rangle ds\\
    &\quad-2\tau_n\int_0^t
    \langle\xi^{\tau_n}_s,|A^{n}_s|^2A^{n}_s\cdot\nabla\phi\rangle ds-\sqrt{\tau_n}\int_0^t
    \langle\bar{\psi}_s,A^{n}_s\cdot\nabla(|A^{n}_s|^2\phi)\rangle ds\\
    &\quad-\frac{i\sqrt{\tau_n}}{2}\int_0^t
    \langle |A^{n}_s|^2(|A^{n}_s|^2-V)\bar{\psi}_s,\phi\rangle ds\\ &\quad +\sqrt{2\tau_n}\sum_{k\in I}\int_0^t
    \langle\xi^{\tau_n}_s,A^{n}_s\cdot\sigma_k\phi\rangle dW^{k,n}_s-i\sqrt{8\tau_n}\sum_{k\in I}\int_0^t
    \langle\xi^{\tau_n}_s,\sigma_k\cdot\nabla\phi\rangle dW^{k,n}_s\\ &\quad +4i\sqrt{\tau_n}\int_0^t
    \langle\bar{\psi}_s,A^{n}_s\cdot\nabla
    (A^{n}_s\cdot\nabla\phi)\rangle ds-2\sqrt{\tau_n}\int_0^t
    \langle\wick{|A^n_s|^2}\bar{\psi}_s,
    A^{n}_s\cdot\nabla\phi\rangle ds.
\end{align*}
\end{lemma}
\begin{proof}
   Since we are working at $n$ fixed we drop the superscript $n$ to save notation, the $\xi$ appearing in this proof is never the one of \eqref{eq:gaussianlimit}. Let $\chi$ be a standard smooth mollifier and let, for each $\epsilon>0$, $\chi_{\epsilon}(x)=\frac{\chi(x/\epsilon)}{\epsilon^d}$. Then choosing $\chi_\epsilon(x-\cdot)$ as a test function in \eqref{eq:weak_formulation_changed_space_prelimit}, we obtain
   \begin{align*}
       i\partial_t\xi^{\epsilon}=-\Delta\xi^{\epsilon}+ 2i\left(A\cdot\nabla \xi\right)\ast \chi_{\epsilon}+ 2i\left(Z \cdot\nabla \bar{\psi}\right)\ast \chi_{\epsilon}+\left(|A|^2\xi\right)\ast \chi_{\epsilon}+\frac{\wick{|A|^2}\bar{\psi}}{\sqrt{\tau}}\ast \chi_{\epsilon},
   \end{align*}
   where $\xi^{\epsilon}=\xi\ast \chi_{\epsilon}$ is smooth and the equation above is now satisfied pointwise. For each $n\in \N,\ \xi^{\tau_n}\in C([0,T];L^2)\ \mathbb{P}-a.s.$, (cf. the discussion before \cref{lem_uniqueness_fluct_eq}). Therefore,
   as $\epsilon \to 0$ 
   \begin{align}\label{convergence_convolution}
       \xi^{\epsilon}\rightarrow \xi \in C([0,T];L^2)\quad\mathbb{P}-a.s.
   \end{align}
   Also the equation for $A$ is satisfied pointwise, since it belongs to $H^{d/2+}.$ Therefore, for each $x\in \R^d$ we can apply the classical It\^o formula, obtaining the pointwise relations
   \begin{align}\label{ito_step_1}
     |A_t|^2\xi^{\epsilon}_t&=\int_0^t|A_s|^2\left(i\Delta\xi^{\epsilon}_s+ 2\left(A_s\cdot\nabla \xi_s\right)\ast \chi_{\epsilon}+ 2\left(Z_s \cdot\nabla \bar{\psi}_s\right)\ast \chi_{\epsilon}\right.\notag\\ & \left.\quad \quad \quad\quad \quad -i\left(|A_s|^2\xi_s\right)\ast \chi_{\epsilon}-i\frac{\wick{|A_s|^2}\bar{\psi}_s}{\sqrt{\tau}}\ast \chi_{\epsilon}\right)ds\notag\\ & -\frac{2}{\tau}\int_0^t \wick{|A_s|^2}\xi^{\epsilon}_s ds+\frac{2\sqrt{2}}{\sqrt{\tau}}\sum_{k\in I }\int_0^t A_s\cdot \sigma_k \xi^{\epsilon}_s dW^k_s,
   \end{align}
   \begin{align}\label{ito_step_2}
     iA_t  \xi^{\epsilon}_t&=\int_0^t A_s\left(-\Delta\xi^{\epsilon}_s+ 2i\left(A_s\cdot\nabla \xi_s\right)\ast \chi_{\epsilon}+ 2i\left(Z_s \cdot\nabla \bar{\psi}_s\right)\ast \chi_{\epsilon}\right.\notag\\ & \left.\quad \quad \quad\quad \quad +\left(|A_s|^2\xi_s\right)\ast \chi_{\epsilon}+\frac{\wick{|A_s|^2}\bar{\psi}_s}{\sqrt{\tau}}\ast \chi_{\epsilon}\right)ds\notag\\ & -\frac{i}{\tau}\int_0^t A_s\xi^{\epsilon}_s ds+\frac{\sqrt{2}i}{\sqrt{\tau}}\sum_{k\in I }\int_0^t  \sigma_k \xi^{\epsilon}_s dW^k_s.
   \end{align}
   Now let $\phi \in \mathscr{S}(\R^d)$. Testing \eqref{ito_step_1} against $\phi$ and integrating by parts, we get
   \begin{align}\label{ito_step_3}
     \langle |A_t|^2\xi^{\epsilon}_t,\phi\rangle&=i\int_0^t\langle\xi^{\epsilon}_s,\Delta\left(|A_s|^2 \phi\right)\rangle ds  +2\int_0^t\left\langle A_s\cdot\nabla\xi_s+Z_s \cdot\nabla \bar{\psi}_s, (|A_s|^2\phi)^{\epsilon}\right\rangle ds \notag
     \\& -i\int_0^t\left\langle |A_s|^2\xi_s+\frac{\wick{|A_s|^2}\bar{\psi}_s}{\sqrt{\tau}}, (|A_s|^2\phi)^{\epsilon}\right\rangle ds\notag\\ & -\frac{2}{\tau}\int_0^t \langle\wick{|A_s|^2}\xi^{\epsilon}_s,\phi\rangle ds+\frac{2\sqrt{2}}{\sqrt{\tau}}\sum_{k\in I }\int_0^t  \langle A_s\cdot \sigma_k\xi^{\epsilon}_s, \phi\rangle dW^k_s,
   \end{align}
   where $(|A_t|^2\phi)^{\epsilon}=(|A_t|^2\phi)\ast \chi_{\epsilon}.$ Due to \eqref{convergence_convolution} and the regularity of the test function $\phi$, $A$ and $\bar{\psi}$, cf. \cref{lem:stochastic_conv} and \cref{prop:well_posed_averaging_limit}, it is standard to pass to the limit in the deterministic integrals as $\epsilon\rightarrow 0$ via the dominated convergence theorem.
   Moreover, similarly, it holds
   \begin{align*}
      \sum_{k\in I }\int_0^T\langle A_s\cdot \sigma_k(\xi^{\epsilon}_s-\xi_s), \phi\rangle^2ds \rightarrow 0\quad\mbox{in probability.}
   \end{align*}
   Therefore, \cite[Lemma 4.3]{bagnara2025no} implies 
   \begin{align*}
       \frac{2\sqrt{2}}{\sqrt{\tau}}\sum_{k\in I }\int_0^t  \langle A_s\cdot \sigma_k\xi^{\epsilon}_s, \phi\rangle dW^k_s\rightarrow \frac{2\sqrt{2}}{\sqrt{\tau}}\sum_{k\in I }\int_0^t  \langle A_s\cdot \sigma_k\xi_s, \phi\rangle dW^k_s \quad\mbox{in probability.}
   \end{align*}
   After passing to a further, non-relabeled subsequence, the stochastic
   integrals converge $\mathbb{P}-a.s.$ In conclusion, we have shown
    \begin{align}\label{ito_step_4}
     \langle |A_t|^2\xi_t,\phi\rangle&=i\int_0^t\langle\xi_s,\Delta\left(|A_s|^2 \phi\right)\rangle ds+2\int_0^t\left\langle |A_s|^2 \left(A_s\cdot\nabla\xi_s+Z_s \cdot\nabla \bar{\psi}_s\right), \phi\right\rangle ds \notag
     \\& -i\int_0^t\left\langle |A_s|^2\left( |A_s|^2\xi_s+\frac{\wick{|A_s|^2}\bar{\psi}_s}{\sqrt{\tau}}\right), \phi\right\rangle ds\notag\\ & -\frac{2}{\tau}\int_0^t \langle\wick{|A_s|^2}\xi_s,\phi\rangle ds+\frac{2\sqrt{2}}{\sqrt{\tau}}\sum_{k\in I }\int_0^t  \langle A_s\cdot \sigma_k\xi_s, \phi\rangle dW^k_s\quad \mathbb{P}-a.s.
   \end{align}
   Testing \eqref{ito_step_2} against $\nabla\phi$ and integrating by parts we get
	   \begin{align*}
	       i\langle\xi^{\epsilon}_t,A_t\cdot\nabla\phi\rangle
	       &=-\int_0^t
	       \langle\xi^{\epsilon}_s,\Delta(A_s\cdot\nabla\phi)\rangle ds\\
	       &\quad+2i\int_0^t
	       \left\langle A_s\cdot\nabla\xi_s+Z_s\cdot\nabla\bar{\psi}_s,
	       (A_s\cdot\nabla\phi)^{\epsilon}\right\rangle ds\\
	       &\quad+\int_0^t\left\langle |A_s|^2\xi_s
	       +\frac{\wick{|A_s|^2}\bar{\psi}_s}{\sqrt{\tau}},
	       (A_s\cdot\nabla\phi)^{\epsilon}\right\rangle ds\\
	       &\quad-\frac{i}{\tau}\int_0^t
	       \langle \xi^{\epsilon}_s,A_s\cdot\nabla\phi\rangle ds+\frac{\sqrt{2}i}{\sqrt{\tau}}\sum_{k\in I}\int_0^t
	       \langle \xi^{\epsilon}_s,\sigma_k\cdot\nabla\phi\rangle dW^k_s,
   \end{align*}
    where $(A_t\cdot\nabla\phi)^{\epsilon}=(A_t\cdot\nabla\phi)\ast \chi_{\epsilon}.$
Arguing similarly as for the proof of \eqref{ito_step_4}, we can let $\epsilon\rightarrow 0$ in the relation above obtaining
\begin{align}\label{ito_step_5}
    i\langle   \xi_t, A_t\cdot\nabla \phi\rangle&=-\int_0^t\langle\xi_s,\Delta\left(A_s\cdot\nabla \phi\right)\rangle ds  +2i\int_0^t\left\langle A_s\cdot\nabla\xi_s+Z_s \cdot\nabla \bar{\psi}_s, A_s\cdot\nabla\phi\right\rangle ds \notag
     \\& +\int_0^t\left\langle |A_s|^2\xi_s+\frac{\wick{|A_s|^2}\bar{\psi}_s}{\sqrt{\tau}}, A_s\cdot\nabla\phi\right\rangle ds\notag\\ & -\frac{i}{\tau}\int_0^t \langle \xi_s,A_s\cdot\nabla\phi\rangle ds+\frac{\sqrt{2}i}{\sqrt{\tau}}\sum_{k\in I }\int_0^t  \langle \xi_s,  \sigma_k\cdot\nabla\phi\rangle dW^k_s. 
\end{align}
Summing \eqref{ito_step_4}, \eqref{ito_step_5}, and the weak formulation
satisfied by $\xi$, we obtain, for every $\phi\in\mathscr{S}$ and
$t\in[0,T]$,
\begin{align*}
    i\langle \xi_t, \phi\rangle&+\frac{\tau}{2}\langle \xi_t,|A_t|^2\phi\rangle-2i\tau\langle \xi_t, A_t\cdot\nabla \phi\rangle\\ & =\int_0^t \langle \xi_s, (-\Delta+V)\phi\rangle ds+2i\int_0^t \langle Z_s\cdot\nabla\bar{\psi}_s,\phi\rangle ds +\int_0^t \left\langle \frac{\wick{|A_s|^2}}{\sqrt{\tau}}\bar{\psi}_s, \phi\right\rangle ds \\ & +\frac{\tau i}{2}\int_0^t\langle\xi_s,\Delta\left(|A_s|^2 \phi\right)\rangle ds+2\tau\int_0^t\langle\xi_s,\Delta\left(A_s\cdot\nabla \phi\right)\rangle ds\\& -\tau\int_0^t \langle \xi_s, A_s\cdot\nabla\left(|A_s|^2\phi\right)\rangle ds -\frac{i\tau}{2}\int_0^t\left\langle \xi_s,  |A_s|^4\phi\right\rangle ds\\ & +4i\tau\int_0^t\left\langle \xi_s, A_s\cdot\nabla\left(A_s\cdot\nabla\phi\right)\right\rangle ds -2\tau\int_0^t\left\langle \xi_s, |A_s|^2 A_s\cdot\nabla\phi\right\rangle ds\notag\\ 
     &  -\sqrt{\tau}\int_0^t\left\langle \bar{\psi}_s, A_s \cdot\nabla\left(|A_s|^2\phi\right)\right\rangle ds-\frac{i\sqrt{\tau}}{2}\int_0^t\left\langle |A_s|^2 \wick{|A_s|^2}\bar{\psi}_s, \phi\right\rangle ds \notag
     \\ &+\sqrt{2\tau}\sum_{k\in I }\int_0^t  \langle \xi_s, A_s\cdot \sigma_k \phi\rangle dW^k_s -\sqrt{8\tau}i\sum_{k\in I }\int_0^t  \langle \xi_s,  \sigma_k\cdot\nabla\phi\rangle dW^k_s\\ 
     &  +4i\sqrt{\tau}\int_0^t\left\langle  \bar{\psi}_s,A_s \cdot\nabla\left( A_s\cdot\nabla\phi\right)\right\rangle ds  -2\sqrt{\tau}\int_0^t\left\langle \wick{|A_s|^2}\bar{\psi}_s, A_s\cdot\nabla\phi\right\rangle ds. 
\end{align*}
  By continuity in time of the terms and integrals on the right-hand side,
  and by density of smooth functions, there is a null set $N$ such that the
  claim holds on $N^c$ for every $t\in[0,T]$ and
  $\phi\in\mathscr{S}$, completing the proof.
  \end{proof}
   To study the limit behaviour of $\Theta^n,\ \Gamma^n $, defined in \eqref{eq:TG_n}, the following lemma will be useful.
\begin{lemma}\label[lemma]{lemma_ito_limit_martingales}
Let $v,v'\in H^1(\R^d)$ and $0\leq s\leq t\leq T$. Set
\begin{align*}
 a_k^v(r)&:=\langle \sigma_k\bar\psi_r,\nabla v\rangle,
 &b_{k,j}^v(r)&:=\langle (\sigma_k\cdot\sigma_j)\bar\psi_r,v\rangle,
 &c_{k}^v(r,r')&:=\langle A_r^n\cdot\sigma_k\bar{\psi}_{r'},v\rangle.
\end{align*}
Then
\begin{align}
    \langle\mathbb{E}\left[\Theta^n_t-\Theta^n_s\mid\mathcal{F}^n_s\right],v\rangle&=- \frac{1}{\sqrt{\tau_n}}\int_s^t e^{-r/\tau_n}\langle A^n_0 \bar{\psi}_r,\nabla v\rangle dr\notag\\ &- \frac{\sqrt{2}}{{\tau_n}}\sum_{k\in I}\int_s^t \int_0^s e^{-(r-u)/\tau_n} a_k^v(r) dW^{k,n}_u dr,\label{identity_one}\\
        \langle\mathbb{E}\left[\Gamma^n_t-\Gamma^n_s\mid\mathcal{F}^n_s\right],v\rangle&=\int_s^t e^{-2r/\tau_{n}} \left\langle\frac{\wick{|A^n_0|^2}}{\sqrt{\tau_n}}\bar{\psi}_r,v\right\rangle dr\notag\\ & +\frac{2\sqrt{2}}{\tau_n}\sum_{k\in I }\int_s^t\int_0^s e^{-2(r-u)/\tau_n} c_k^v(u,r) dW^{k,n}_u   dr,\label{identity_two}
\end{align}
and
\begin{align}
 \mathbb E\!\left[
  \left|\left\langle
   \mathbb E[\Theta_t^n-\Theta_s^n\mid\mathcal F_s^n],v
  \right\rangle\right|\right]
 &\lesssim \sqrt{\tau_n}\,
  \sup_{r\in[0,T]}\|\bar\psi_r\|\,\|v\|_{H^1},
 \label{asymptotic_martingales_1}\\
 \mathbb E\!\left[
  \left|\left\langle
   \mathbb E[\Gamma_t^n-\Gamma_s^n\mid\mathcal F_s^n],v
  \right\rangle\right|\right]
 &\lesssim \sqrt{\tau_n}\,
  \sup_{r\in[0,T]}\|\bar\psi_r\|\,\|v\|_{H^1}.
 \label{asymptotic_martingales_2}
\end{align}
Moreover,
\begin{align}
 &\mathbb E\!\left[
  \langle\Theta_t^n-\Theta_s^n,v\rangle
  \langle\Theta_t^n-\Theta_s^n,v'\rangle
  \,\middle|\,\mathcal F_s^n\right]\notag\\
 &\qquad=\frac1{\tau_n}\sum_{k\in I}
  \int_s^t\!\!\int_s^t e^{-|r-r'|/\tau_n}
  a_k^v(r)a_k^{v'}(r')\,dr\,dr'
  +\rho_{s,t}^{1,n}(v,v'),\label{second_moment_1}\\
 &\mathbb E\!\left[
  \langle\Gamma_t^n-\Gamma_s^n,v\rangle
  \langle\Gamma_t^n-\Gamma_s^n,v'\rangle
  \,\middle|\,\mathcal F_s^n\right]\notag\\
 &\qquad=\frac2{\tau_n}\sum_{k,j\in I}
  \int_s^t\!\!\int_s^t e^{-2|r-r'|/\tau_n}
  b_{k,j}^v(r)b_{k,j}^{v'}(r')\,dr\,dr'
  +\rho_{s,t}^{2,n}(v,v'),\label{second_moment_2}\\
 &\mathbb E\!\left[
  \langle\Theta_t^n-\Theta_s^n,v\rangle
  \langle\Gamma_t^n-\Gamma_s^n,v'\rangle
  \,\middle|\,\mathcal F_s^n\right]
  =\rho_{s,t}^{3,n}(v,v'),\label{second_moment_3}
\end{align}
where
\begin{align*}
    \rho^{1,n}_{s,t}(v,v')&=\frac{1}{\tau_n}\int_s^t\int_s^t e^{-(r+r')/\tau_n} \langle A^n_0 \bar{\psi}_r,\nabla v\rangle \langle A^n_0 \bar{\psi}_{r'},\nabla v'\rangle dr dr'\\ & +\frac{\sqrt{2}}{\tau_n^{3/2}}\sum_{k\in I}\int_s^t\int_s^t\int_0^{s} e^{-(r+r'-u)/\tau_n}\langle A^n_0 \bar{\psi}_r,\nabla v\rangle a_k^{v'}(r') dW^{k,n}_u dr dr'\\ & +\frac{\sqrt{2}}{\tau_n^{3/2}}\sum_{k\in I}\int_s^t\int_s^t\int_0^{s} e^{-(r+r'-u)/\tau_n} a_k^v(r) \langle A^n_0 \bar{\psi}_{r'},\nabla v'\rangle  dW^{k,n}_u dr dr'\\ &+\frac{2}{\tau_n^2}\sum_{k,k'\in I}\int_s^t\int_s^t\int_0^s \int_0^s e^{-(r+r'-u-u')/\tau_n}a_k^v(r)  a_k^{v'}(r') dW^{k,n}_u dW^{k',n}_{u'}dr dr'\\ & -\frac{1}{\tau_n}\sum_{k\in I}\int_s^t \int_s^t  e^{-(r+r'-2s)/\tau_n} a_k^v(r)a_k^{v'}(r') dr dr',
    \\ \rho^{2,n}_{s,t}(v,v')&=\frac{1}{\tau_n}\int_s^t\int_s^t e^{-2(r+r')/\tau_n} \langle \wick{|A^n_0|^2} \bar{\psi}_r, v\rangle \langle \wick{|A^n_0|^2} \bar{\psi}_{r'}, v'\rangle dr dr'\\ & +\frac{\sqrt{8}}{\tau_n^{3/2}}\sum_{k\in I}\int_s^t\int_s^t\int_0^{s} e^{-2(r+r'-u)/\tau_n}\langle \wick{|A^n_0|^2} \bar{\psi}_r, v\rangle c_k^{v'}(u,r') dW^{k,n}_u dr dr'\\ & +\frac{\sqrt{8}}{\tau_n^{3/2}}\sum_{k\in I}\int_s^t\int_s^t\int_0^{s} e^{-2(r+r'-u)/\tau_n} c_k^{v}(u,r) \langle \wick{|A^n_0|^2} \bar{\psi}_{r'}, v'\rangle  dW^{k,n}_u dr dr'\\ &+\frac{8}{\tau_n^2}\sum_{k,k'\in I}\int_s^t\int_s^t\int_0^s \int_0^s e^{-2(r+r'-u-u')/\tau_n} c_k^v(u,r) c_{k'}^{v'}(u',r') dW^{k,n}_u dW^{k',n}_{u'}dr dr'\\ & -\frac{2}{\tau_n}\sum_{k,k'\in I}\int_s^t \int_s^t  e^{-2(r+r'-2s)/\tau_n} b_{k,k'}^v(r) b_{k,k'}^{v'}(r') dr dr'   \\
    \rho^{3,n}_{s,t}(v,v')&=-\frac{1}{\tau_n}\int_s^t\int_s^t e^{-(r+2r')/\tau_n} \langle A_0 \bar{\psi}_r, v\rangle \langle \wick{|A^n_0|^2} \bar{\psi}_{r'}, v'\rangle dr dr'\\ & -\frac{\sqrt{8}}{\tau_n^{3/2}}\sum_{k\in I}\int_s^t\int_s^t\int_0^{s} e^{-(r+2r'-2u)/\tau_n}\langle \wick{A_0} \bar{\psi}_r, \nabla v\rangle c_k^{v'}(u,r') dW^{k,n}_u dr dr'\\ & -\frac{\sqrt{2}}{\tau_n^{3/2}}\sum_{k\in I}\int_s^t\int_s^t\int_0^{s} e^{-(r+2r'-u)/\tau_n} a_{k}^v(r) \langle \wick{|A^n_0|^2} \bar{\psi}_{r'}, v'\rangle  dW^{k,n}_u dr dr'\\ &-\frac{4}{\tau_n^2}\sum_{k,k'\in I}\int_s^t\int_s^t\int_0^s \int_0^s e^{-(r+2r'-u-2u')/\tau_n} a_k^v(r) c_{k'}^{v'}(u',r') dW^{k,n}_u dW^{k',n}_{u'}dr dr',
\end{align*}
and
\begin{align}\label{bound_reminders}
 \mathbb E\!\left[
  |\rho_{s,t}^{1,n}(v,v')|
  +|\rho_{s,t}^{2,n}(v,v')|
  +|\rho_{s,t}^{3,n}(v,v')|
 \right]
 \lesssim \tau_n\,
 \sup_{r\in[0,T]}\|\bar\psi_r\|^2
 \|v\|_{H^1}\|v'\|_{H^1}.
\end{align}
The implicit constants are independent of $n,s,t$.
\end{lemma}
\begin{proof}
By definition it holds
    \begin{align}\label{eq:time_difference_theta}
        \Theta_t-\Theta_s&=\frac{1}{\sqrt{\tau}}\int_s^t  \left(e^{-r/\tau}A_0+\frac{\sqrt{2}}{\sqrt{\tau}}\sum_{k\in I}\int_0^r e^{-(r-u)/\tau}\sigma_k dW^k_u \right)\cdot\nabla\bar{\psi}_r dr.
    \end{align}
Therefore
    \begin{align*}
     \mathbb{E}\left[\Theta_t-\Theta_s\mid\mathcal{F}_s\right]&=  \frac{1}{\sqrt{\tau}}\int_s^t  \left(e^{-r/\tau}A_0+\frac{\sqrt{2}}{\sqrt{\tau}}\sum_{k\in I}\int_0^s e^{-(r-u)/\tau}\sigma_k dW^k_u \right)\cdot\nabla\bar{\psi}_r dr
    \end{align*}
    and this implies \eqref{identity_one} testing against $v$ and integrating by parts. 
    The validity of \eqref{second_moment_1} similarly follows from \eqref{eq:time_difference_theta} and independence of the $W^k$'s, we omit the easy details.
    Then, by Stochastic Fubini theorem, H\"older's inequality and It\^o isometry 
    \begin{multline*}
        \expt{\left\lvert \langle\mathbb{E}\left[\Theta_t-\Theta_s\mid\mathcal{F}_s\right],v\rangle\right\rvert}\\ \lesssim \sqrt{\tau}\sup_{t\in [0,T]}\|\bar{\psi}_t\|\|v\|_{H^1}\left(\mathbb{E}[\|A_0\|_{L^\infty}]+\left(\sum_{k\in I}\|\sigma_k\|_{L^{\infty}}^2\right)^{1/2}\right)\\  \lesssim \sqrt{\tau}\sup_{t\in [0,T]}\|\bar{\psi}_t\|\|v\|_{H^1}
    \end{multline*}
    by \cref{hp:noise} and \cref{lem:stochastic_conv}. This completes the proof of \eqref{asymptotic_martingales_1}.
Arguing similarly we also get
\begin{align*}
    \mathbb{E}\left[|\rho^{1,n}_{s,t}(v,v')|\right]& \lesssim \tau\sup_{t\in [0,T]}\|\bar{\psi}_t\|^2\|v\|_{H^1}\|v'\|_{H^1},
\end{align*}
    we omit the easy details. For the second claim, we first observe that by It\^o formula it holds
    \begin{align*}
       d\wick{|A|^2}=-\frac{2}{\tau}\wick{|A|^2}dt+\frac{2\sqrt{2}}{\sqrt{\tau}}\sum_{k\in I}\sigma_k\cdot AdW^k_t.
    \end{align*}
    Therefore 
    \begin{align*}
        \frac{\wick{|A_t|^2}}{\sqrt{\tau}}&=e^{-2t/\tau} \frac{\wick{|A_0|^2}}{\sqrt{\tau}}+\frac{2\sqrt{2}}{\tau}\sum_{k\in I}\int_0^t e^{-2(t-r)/\tau} \sigma_k\cdot A_r dW^k_r.
    \end{align*}
    In particular
    \begin{align}\label{eq:time_difference_gamma}
      \Gamma_t-\Gamma_s&=\int_s^t e^{-2r/\tau} \frac{\wick{|A_0|^2}}{\sqrt{\tau}}\bar{\psi}_r dr+\frac{2\sqrt{2}}{\tau}\sum_{k\in I }\int_s^t\int_0^r e^{-2(r-u)/\tau}A_u\cdot\sigma_k dW^k_u \bar{\psi}_r  dr   
    \end{align}
    and
    \begin{align*}
        \mathbb{E}\left[\Gamma_t-\Gamma_s\mid\mathcal{F}_s\right]&=\int_s^t e^{-2r/\tau} \frac{\wick{|A_0|^2}}{\sqrt{\tau}}\bar{\psi}_r dr\\ &+\frac{2\sqrt{2}}{\tau}\sum_{k\in I }\int_s^t\int_0^s e^{-2(r-u)/\tau}A_u\cdot\sigma_k dW^k_u \bar{\psi}_r  dr.
    \end{align*}
    The latter implies \eqref{identity_two} testing against $v$. Relation \eqref{second_moment_2} similarly follows from \eqref{eq:time_difference_gamma}, the independence of the $W^k$'s and \cref{lem:stochastic_conv}, we omit the easy details. Then, by Stochastic Fubini theorem, H\"older's inequality and It\^o isometry 
    \begin{align*}
        &\expt{\left\lvert \langle\mathbb{E}\left[\Gamma_t-\Gamma_s\mid\mathcal{F}_s\right],v\rangle\right\rvert}\\ & \lesssim \sqrt{\tau}\sup_{t\in [0,T]}\|\bar{\psi}_t\|\|v\|\left(\mathbb{E}[\|A_0\|^2_{L^\infty}]+\sum_{k\in I}\|\sigma_k\|_{L^{\infty}}^2\right)\\ & +\frac{1}{\tau}\int_0^s \left(\mathbb{E}\left[\sum_{k\in I}\left(\int_s^t e^{-2(r-u)/\tau} \langle \sigma_k A_u,v\rangle dr\right)^2\right]\right)^{1/2} du\\ & \lesssim \sqrt{\tau}\sup_{t\in [0,T]}\|\bar{\psi}_t\|\|v\|\left(\mathbb{E}[\|A_0\|^2_{L^\infty}]+\sum_{k\in I}\|\sigma_k\|_{L^{\infty}}^2\right)\\ & +\int_0^s \left(\mathbb{E}\left[\sum_{k\in I}\|\sigma_k\|_{L^{\infty}}^2\|v\|^2\|A_u\|^2 \left(e^{-2(t-u)/\tau}+e^{-2(s-u)/\tau}\right)\right]\right)^{1/2} du \\ & \lesssim \sqrt{\tau}\sup_{t\in [0,T]}\|\bar{\psi}_t\|\|v\|
    \end{align*}
    by \cref{hp:noise} and \cref{lem:stochastic_conv}. This completes the proof of \eqref{asymptotic_martingales_2}. Arguing similarly we also get
\begin{align*}
    \mathbb{E}\left[|\rho^{2,n}_{s,t}(v,v')|\right]& \lesssim \tau\sup_{t\in [0,T]}\|\bar{\psi}_t\|^2\|v\|\|v'\|.
\end{align*}
    Lastly, combining \eqref{eq:time_difference_theta} and \eqref{eq:time_difference_gamma}, we obtain
    \begin{align*}
        &\langle\Theta_t-\Theta_s,v\rangle\langle\Gamma_t-\Gamma_s,v'\rangle\\ &=-\frac{1}{\tau}\int_s^t\int_s^t e^{-(r+2r')/\tau} \langle A_0 \bar{\psi}_r, v\rangle \langle \wick{|A_0|^2} \bar{\psi}_{r'}, v'\rangle dr dr'\\ & -\frac{\sqrt{8}}{\tau^{3/2}}\sum_{k\in I}\int_s^t\int_s^t\int_0^{r'} e^{-(r+2r'-2u)/\tau}\langle \wick{A_0} \bar{\psi}_r, \nabla v\rangle c^{v'}_k(u,r') dW^{k}_u dr dr'\\ & -\frac{\sqrt{2}}{\tau^{3/2}}\sum_{k\in I}\int_s^t\int_s^t\int_0^{r} e^{-(r+2r'-u)/\tau} a_{k}^v(r) \langle \wick{|A_0|^2} \bar{\psi}_{r'}, v'\rangle  dW^{k}_u dr dr'\\ &-\frac{4}{\tau^2}\sum_{k,k'\in I}\int_s^t\int_s^t\int_0^r \int_0^{r'} e^{-(r+2r'-u-2u')/\tau} a_{k}^v(r) c_{k'}^{v'}(u',r') dW^{k}_u dW^{k'}_{u'}dr dr'.
    \end{align*}
    Since $A_t$ is zero mean, cf. \cref{lem:stochastic_conv}, taking the conditional expectation with respect to $\mathcal{F}_s$ of the expression above we obtain \eqref{second_moment_3}. By similar reasoning as for the previous terms,
    \begin{align*}
    \mathbb{E}\left[|\rho^{3,n}_{s,t}(v,v')|\right]& \lesssim \tau\sup_{t\in [0,T]}\|\bar{\psi}_t\|^2\|\nabla v\|\|v'\|.
\end{align*}
\end{proof}

We are now ready to complete the proof of \cref{thm_fluctuations}.
\begin{proof}[Proof of \cref{thm_fluctuations}]
We split the proof into two steps. In the first one, by \cref{perturbed_test_function_fluctuations}, we pass to the limit in the equation satisfied by $\xi^{\tau_n}.$ Secondly we identify the limit processes $\Gamma_t,\ \Theta_t$ as those appearing in \eqref{eq:fluctuations} by a martingale representation theorem, cf. \cite[Theorem 8.2]{da2014stochastic}.\\
\emph{Step 1.} Let $\phi $ as in the statement of \cref{perturbed_test_function_fluctuations} and $t\in [0,T]$. Thanks to \eqref{convergence_xi}, \eqref{convergence_theta} and \eqref{convergence_gamma} we easily have
\begin{align*}
    i\langle \xi^{\tau_n}_t, \phi\rangle&\rightarrow i\langle \xi_t, \phi\rangle\quad\mathbb{P}-a.s.,
    \\
    \int_0^t \langle \xi^{\tau_n}_s, (-\Delta+V)\phi\rangle ds&\rightarrow \int_0^t \langle \xi_s, (-\Delta+V)\phi\rangle ds\quad\mathbb{P}-a.s.,\\
    2i\int_0^t \langle Z^{n}_s\cdot\nabla\bar{\psi}_s,\phi\rangle ds &\rightarrow 2i\langle \Theta_t,\phi\rangle\quad\mathbb{P}-a.s.,\\
    \int_0^t \left\langle
    \frac{\wick{|A^n_s|^2}}{\sqrt{\tau_n}}\bar{\psi}_s,\phi
    \right\rangle ds
    &\rightarrow \langle \Gamma_t,\phi\rangle\quad\mathbb{P}-a.s.
\end{align*}
We are left to show that all the other terms go to zero as $n\rightarrow +\infty$.
For each $\delta>0$, Young's inequality gives
\begin{align*}
    &\mathbb{P}\left(
    \tau_n|\langle \xi^{\tau_n}_t,|A^{n}_t|^2\phi\rangle|>\delta
    \right)\\
    &\quad\leq
    \mathbb{P}\left(\tau_n\|\xi^{\tau_n}_t\|_{H^{-3}}^2>\delta\right)
    +\mathbb{P}\left(
    \tau_n\|\phi\|_{H^3}^2
    \|A^{n}_t\|_{H^{5+\frac{d+\theta}{2}}}^4>\delta\right).
\end{align*}
The first probability tends to zero because, by \eqref{convergence_xi},
\begin{align*}
\tau_n\sup_{t\in[0,T]}\|\xi^{\tau_n}_t\|_{H^{-3}}^2
\longrightarrow0\quad\mathbb{P}-a.s.
\end{align*}
The second tends to zero by Markov's inequality and
\cref{lem:stochastic_conv}. Therefore,
\begin{align*}
    \frac{\tau_n}{2}\langle \xi^{\tau_n}_t,|A^{n}_t|^2\phi\rangle\rightarrow 0\quad\mbox{in probability.}
\end{align*}
An analogous argument proves convergence in probability of
$-2i\tau_n\langle \xi^{\tau_n}_t,A^{n}_t\cdot\nabla\phi\rangle$.
The deterministic integrals can be treated similarly; we give one example:
\begin{align*}
    &\mathbb{P}\left(
    \left\lvert\tau_n\int_0^t
    \langle\xi^{\tau_n}_s,\Delta(A^{n}_s\cdot\nabla\phi)\rangle ds
    \right\rvert>\delta\right)\\
    &\quad\leq\mathbb{P}\left(
    \tau_n\sup_{t\in[0,T]}\|\xi^{\tau_n}_t\|_{H^{-3}}^2>\delta
    \right)+\mathbb{P}\left(
    \tau_n\|\phi\|^2_{H^6}
    \left(\int_0^T
    \|A^{n}_t\|_{H^{5+\frac{d+\theta}{2}}}dt\right)^2>\delta
    \right).
\end{align*}
Both terms tend to zero by the preceding argument. We are left to prove
convergence to zero in probability of the stochastic integrals. By
\cite[Lemma 4.3]{bagnara2025no}, it is enough to show that
\begin{align*}
    \tau_n\sum_{k\in I }\int_0^T
    \left|\langle\xi^{\tau_n}_s,
    A^n_s\cdot \sigma_k\phi\rangle\right|^2 ds
    &\rightarrow 0\quad\mbox{in probability,} \\
    \tau_n\sum_{k\in I }\int_0^T
    \left|\langle\xi^{\tau_n}_s,
    \sigma_k\cdot\nabla\phi\rangle\right|^2 ds
    &\rightarrow 0\quad\mbox{in probability.} 
\end{align*}
We just show the first one, the second being analogous and simpler.
Let $\delta>0$, by Young's inequality it holds
\begin{align*}
    &\mathbb{P}\left(\tau_n\sum_{k\in I }\int_0^T \langle\xi^{\tau_n}_s, A^n_s\cdot \sigma_k\phi\rangle^2 ds>\delta\right)\\ &\leq \mathbb{P}\left(\tau_n\sum_{k\in I }\|\sigma_k\|_{H^{5+\frac{d+\theta}{2}}}^2\|\phi\|_{H^3}^2\sup_{t\in [0,T]}\|\xi^{\tau_n}_t\|_{H^{-3}}^2\int_0^T \|A^n_s\|_{H^{5+\frac{d+\theta}{2}}}^2ds>\delta\right)\\ & \leq \mathbb{P}\left(\tau_n\left(\sum_{k\in I }\|\sigma_k\|_{H^{5+\frac{d+\theta}{2}}}^2\right)^2\|\phi\|_{H^3}^4\sup_{t\in [0,T]}\|\xi^{\tau_n}_t\|_{H^{-3}}^4>\delta\right)\\ &  +
    \mathbb{P}\left(\tau_n\left(\int_0^T \|A^n_s\|_{H^{5+\frac{d+\theta}{2}}}^2ds\right)^2>\delta\right).
\end{align*}
The first term converges to zero by \eqref{convergence_xi}, whereas the second one by \cref{lem:stochastic_conv} and Markov's inequality.  This proves the desired convergence. Computations above imply that for each $\phi \in \mathscr{S}$ there exists
$N\subset \Omega$ with $\mathbb{P}(N)=0$ such that, on $N^c$, the
following holds for each $t\in \Q\cap [0,T]$:
\begin{align}\label{eq_limit}
    i\langle \xi_t, \phi\rangle
    &=\int_0^t \langle \xi_s, (-\Delta+V)\phi\rangle ds
      +2i\langle \Theta_t,\phi\rangle
      +\langle \Gamma_t,\phi\rangle.
\end{align}
By continuity in time of all the terms above and the existence of a countable
dense subset of $H^{5}$ consisting of functions in $\mathscr{S}$, there exists
$N_0\subset \Omega:\ \mathbb{P}(N_0)=0$ and on $N_0^c$ equation \eqref{eq_limit} holds for each $t\in [0,T],\ \phi\in H^5.$
\\ \emph{Step 2.} We are left to identify the limit process $(\Theta_t,\Gamma_t)$ as a martingale with covariance structure consistent with \eqref{eq:fluctuations}. To ease the notation, let us introduce for $t\in [0,T]$ the spaces
\begin{align*}
    \mathcal{Z}_t:=\bar{Z}_t\times  C([0,t];H^{-1})\times C([0,t];L^{2})\times H^{5+\frac{d+\theta}{2}}\times  C([0,t];\R^I)
\end{align*}
In view of the results of \cref{subsec:compactness_fluct} and Vitali theorem, the convergence of \eqref{convergence_gamma}, \eqref{convergence_theta} holds also in $L^p(\Omega)$ for every $p\in [1,+\infty).$
Thus, for every\footnote{Note that, even if $\bar{Z}_t$ is not Polish, the function $h(\xi,\Theta,\Gamma, A_0,W):(\Omega,\mathcal{F})\rightarrow(\R,\mathcal{B}(\R))$ is nevertheless a well-defined real-valued random variable by continuity of $h$; see \cite[Lemma 4.51]{aliprantis2006infinite}. The same reasoning applies to the other expressions appearing in this argument.} 
\begin{align*}
    v,v'\in H^5,\quad h\in C_b(\mathcal{Z}_s)
\end{align*}
it holds
\begin{align}\label{convergence_martin_quadr_var} \hspace{-5cm}\mathbb{E}\left[\langle \Theta^n_t-\Theta^n_s-\Gamma^n_s,v\rangle h(\xi^{\tau_n},\Theta^n,\Gamma^n, A^n_0,W^n)\right] &\rightarrow \mathbb{E}\left[\langle \Theta_t-\Theta_s,v\rangle h(\xi,\Theta,\Gamma, A_0,\bar{W})\right],\\
\hspace{-5cm}\mathbb{E}\left[\langle \Gamma^n_t-\Gamma^n_s,v\rangle h(\xi^{\tau_n},\Theta^n,\Gamma^n, A^n_0,W^n)\right] &\rightarrow \mathbb{E}\left[\langle \Gamma_t-\Gamma_s,v\rangle h(\xi,\Theta,\Gamma, A_0,\bar{W})\right],\\
    \notag\hspace{1cm}\mathbb{E}\left[\langle \Theta^n_t-\Theta^n_s,v\rangle\langle \Theta^n_t-\Theta^n_s,v'\rangle  h(\xi^{\tau_n},\Theta^n,\Gamma^n, A^n_0,W^n)\right]\\ &\hspace{-3cm}\rightarrow \mathbb{E}\left[\langle \Theta_t-\Theta_s,v\rangle\langle \Theta_t-\Theta_s,v'\rangle h(\xi,\Theta,\Gamma, A_0,\bar{W})\right],
    \\ \mathbb{E}\left[\langle \Gamma^n_t-\Gamma^n_s,v\rangle\langle \Gamma^n_t-\Gamma^n_s,v'\rangle  h(\xi^{\tau_n},\Theta^n,\Gamma^n, A^n_0,W^n)\right]\notag\\ &\hspace{-3cm}\rightarrow \mathbb{E}\left[\langle \Gamma_t-\Gamma_s,v\rangle\langle \Gamma_t-\Gamma_s,v'\rangle h(\xi,\Theta,\Gamma, A_0,\bar{W})\right],\\ \mathbb{E}\left[\langle \Theta^n_t-\Theta^n_s,v\rangle\langle \Gamma^n_t-\Gamma^n_s,v'\rangle  h(\xi^{\tau_n},\Theta^n,\Gamma^n, A^n_0,W^n)\right]\notag\\ &\hspace{-3cm}\rightarrow \mathbb{E}\left[\langle \Theta_t-\Theta_s,v\rangle\langle \Gamma_t-\Gamma_s,v'\rangle h(\xi,\Theta,\Gamma, A_0,\bar{W})\right].
\end{align}
We are therefore left to show
\begin{align}\label{property_martingale}
    \mathbb{E}\left[\langle \Theta^n_t-\Theta^n_s,v\rangle h(\xi^{\tau_n},\Theta^n,\Gamma^n, A^n_0,W^n)\right]&\rightarrow 0,\\  \label{property_martingale_2}\mathbb{E}\left[\langle \Gamma^n_t-\Gamma^n_s,v\rangle h(\xi^{\tau_n},\Theta^n,\Gamma^n, A^n_0,W^n)\right]&\rightarrow 0,
\end{align}
and the corresponding property for the quadratic variations. In order to save notation, we define
\begin{align*}
    H_n:&=h(\xi^{\tau_n},\Theta^n,\Gamma^n, A^n_0,W^n),\\
    H:&=h(\xi,\Theta,\Gamma, A_0,\bar{W}).
\end{align*}
Claim \eqref{property_martingale}, immediately follows from \eqref{asymptotic_martingales_1}, \eqref{asymptotic_martingales_2} and the boundedness of $h$.
Indeed, by the tower property, 
\begin{align*}
    \left\lvert\mathbb{E}\left[\langle \Theta^n_t-\Theta^n_s,u\rangle H_n\right]\right\rvert  & =\left\lvert\mathbb{E}\left[\mathbb{E}\left[\langle \Theta^n_t-\Theta^n_s,u\rangle\mid \mathcal{F}^n_s\right] H_n\right]\right\rvert\\& \lesssim \mathbb{E}\left[\left\lvert\mathbb{E}\left[\langle \Theta^n_t-\Theta^n_s,u\rangle\mid \mathcal{F}^n_s\right] \right\rvert\right] \lesssim \sqrt{\tau_n}\sup_{t\in [0,T]}\|\bar{\psi}_t\|\|v\|_{H^1}\rightarrow 0.
\end{align*}
Analogously one can show \eqref{property_martingale_2}.
The study of the quadratic variations also follows by \cref{lemma_ito_limit_martingales}. Indeed, by the tower property we get
\begin{align*}
  &\mathbb{E}\left[\langle \Theta^n_t-\Theta^n_s,v\rangle\langle \Theta^n_t-\Theta^n_s,v'\rangle  H_n\right]\\ &=\mathbb{E}\left[\rho_{s,t}^{1,n}(v,v')  H_n\right] +\frac{\expt{H_n}}{\tau_n}\sum_{k\in I}\int_s^t \int_s^t  e^{-|r-r'|/\tau_n} \langle \sigma_k\bar{\psi}_r,\nabla v\rangle \langle \sigma_k\bar{\psi}_{r'},\nabla v'\rangle dr dr',\\ &\mathbb{E}\left[\langle \Gamma^n_t-\Gamma^n_s,v\rangle\langle \Gamma^n_t-\Gamma^n_s,v'\rangle  H_n\right]\\ &=\mathbb{E}\left[\rho_{s,t}^{2,n}(v,v')  H_n\right] 
  +\frac{2\expt{H_n}}{\tau_n}\sum_{k,j\in I}\int_s^t \int_s^t  e^{-2|r-r'|/\tau_n} \langle \sigma_k\cdot\sigma_j\bar{\psi}_r, v\rangle \langle \sigma_k\cdot\sigma_j\bar{\psi}_{r'}, v'\rangle dr dr',\\
&\mathbb{E}\left[\langle \Theta^n_t-\Theta^n_s,v\rangle\langle \Gamma^n_t-\Gamma^n_s,v'\rangle  H_n\right]=\mathbb{E}\left[\rho_{s,t}^{3,n}(v,v')  H_n\right] .
\end{align*}
From \eqref{bound_reminders} and the boundedness of $h$ it follows that
\begin{align*}
    &\left|\mathbb{E}\left[\sum_{i=1}^3\rho_{s,t}^{i,n}(v,v')  H_n\right]\right|\lesssim \tau_n\|v\|_{H^1}\|v'\|_{H^1}\rightarrow 0,\qquad \expt{H_n}\rightarrow \expt{H}
\end{align*}
as $n\rightarrow+\infty.$ While
\begin{multline}\label{convergence_1_covar}
    \frac{1}{\tau_n}\sum_{k\in I}\int_s^t \int_s^t  e^{-|r-r'|/\tau_n} \langle \sigma_k\bar{\psi}_r,\nabla v\rangle \langle \sigma_k\bar{\psi}_{r'},\nabla v'\rangle dr dr'\\ \rightarrow 2\sum_{k\in I}\int_s^t \langle \sigma_k\bar{\psi}_r,\nabla v\rangle \langle \sigma_k\bar{\psi}_{r},\nabla v'\rangle dr,\end{multline}
    \begin{multline}\label{convergence_2_covar}
    \frac{2}{\tau_n}\sum_{k,j\in I}\int_s^t \int_s^t  e^{-2|r-r'|/\tau_n} \langle \sigma_k\cdot\sigma_j\bar{\psi}_r, v\rangle \langle \sigma_k\cdot\sigma_j\bar{\psi}_{r'}, v'\rangle dr dr'\\ \rightarrow  2\sum_{k,j\in I}\int_s^t \langle \sigma_k\cdot\sigma_j\bar{\psi}_r, v\rangle \langle \sigma_k\cdot\sigma_j\bar{\psi}_{r}, v'\rangle dr. 
\end{multline}
Combining \eqref{convergence_1_covar} and \eqref{convergence_2_covar} with \eqref{convergence_martin_quadr_var}, we obtain
\begin{align*}
\mathbb{E}\left[\langle \Theta_t-\Theta_s,v\rangle h(\xi,\Theta,\Gamma, A_0,\bar{W})\right]=0,\quad \mathbb{E}\left[\langle \Gamma_t-\Gamma_s,v\rangle h(\xi,\Theta,\Gamma, A_0,\bar{W})\right]=0,
    \end{align*}
    \begin{align}
    \label{quadratic_cov_2}
\mathbb{E}\left[\langle \Theta_t-\Theta_s,v\rangle\langle \Theta_t-\Theta_s,v'\rangle h(\xi,\Theta,\Gamma, A_0,\bar{W})\right]  &=2\sum_{k\in I}\int_s^t \langle \sigma_k\bar{\psi}_r,\nabla v\rangle \langle \sigma_k\bar{\psi}_{r},\nabla v'\rangle dr,\\ \label{quadratic_cov_3}
\mathbb{E}\left[\langle \Gamma_t-\Gamma_s,v\rangle\langle \Gamma_t-\Gamma_s,v'\rangle h(\xi,\Theta,\Gamma, A_0,\bar{W})\right]  &=2\sum_{k,j\in I}\int_s^t \langle \sigma_k\cdot\sigma_j\bar{\psi}_r, v\rangle \langle \sigma_k\cdot\sigma_j\bar{\psi}_{r}, v'\rangle dr,\\ \mathbb{E}\left[\langle \Theta_t-\Theta_s,v\rangle\langle \Gamma_t-\Gamma_s,v'\rangle h(\xi,\Theta,\Gamma, A_0,\bar{W})\right]  &=0
\label{eq:quadratic_variation_1}
    \end{align}
for every 
  $v,v'\in H^5(\R^d),\  h\in C_b(\mathcal{Z}_t)$. Namely $(\Theta,\Gamma)$ is a continuous martingale with quadratic variation given by \eqref{quadratic_cov_2}, \eqref{quadratic_cov_3}, \eqref{eq:quadratic_variation_1}.
Therefore, up to enlarging the probability space, cf. \cite[Theorem 8.2]{da2014stochastic}, there exist two families of real mutually independent Brownian motions, $((W^k_t)_{k\in  I},(B^{k,j}_t)_{k,j\in I})$, such that $(B^{k,j}_t)_{k,j\in I}\mathrel{\perp\!\!\!\perp} (W^{k}_t)_{k\in I}$ and 
\begin{align*}
   ( \Theta_t,\Gamma_t)= \left(\sqrt{2}\sum_{k\in I}\int_0^t \sigma_k\cdot\nabla\bar{\psi}_r dW^k_r,\sqrt{2}\sum_{k,j\in I}\int_0^t \sigma_k\cdot\sigma_j\bar{\psi}_r dB^{k,j}_r\right). 
\end{align*}
This completes the proof by uniqueness of the limit.
\end{proof}
  
\section{Averaging and Fluctuations of the Effective Model}\label{sec:proof_effective}
We split this section into two subsections, proving \eqref{eq_rate_averaging} and \eqref{eq_rate_fluctuctions} respectively.
\subsection{\texorpdfstring{Proof of \eqref{eq_rate_averaging}}{Proof of the averaging rate}}
We first observe that \cref{prop:well_effective} and
\cref{prop:well_posed_averaging_limit} imply, for each $q\geq 1$,
\begin{align}\label{eq:apriori_diff_h1}
    \expt{\sup_{t\in [0,T]}\|\Psi^{\tau}_t-\bar{\psi}_t\|^q}&\lesssim_{q,\psi_0} 1.
\end{align}
Next, by standard arguments, see for example
\cite[Section 5.2]{da2014stochastic}, denoting by
$h^{\tau}_t=\Psi^{\tau}_t-\bar{\psi}_t$ and by $S(t)=e^{it\Delta}$ the Schr\"odinger
group, we have
\begin{align*}
    h^{\tau}_t
    &=-i\int_0^t S(t-s)[Vh^{\tau}_s]ds\\
    &\quad+\tau\int_0^t S(t-s)
    [4\operatorname{div}(Q\nabla\Psi^{\tau}_s)-R\Psi^{\tau}_s]ds
    +\sqrt{\tau}Z^1_t+\sqrt{\tau}Z^2_t,
\end{align*}
    where
    \begin{align*}
    Z^1_t&=2\sqrt{2}\sum_{k\in I}\int_0^t S(t-s)[\sigma_k\cdot\nabla \Psi^{\tau}_s] dW^k_s,\\
    Z^2_t&=-\sqrt{2}i\sum_{k,j\in I}\int_0^t S(t-s)[\sigma_k\cdot\sigma_j\Psi^{\tau}_s] dB^{k,j}_s
\end{align*}
seen as equalities in $H^{-2}(\R^d)$. In particular, evaluating the $H^{-2}$ norm of the equation above we obtain by simple computations
\begin{align*}
   \|h^{\tau}_t \|_{H^{-2}}
   &\lesssim \int_0^t \|h^{\tau}_s \|_{H^{-2}} ds
   +\tau \int_0^T \|\Psi^{\tau}_s\|ds\\
   &\quad+\sqrt{\tau}\sup_{t\in [0,T]}\|Z^1_t\|_{H^{-2}}
   +\sqrt{\tau}\sup_{t\in [0,T]}\|Z^2_t\|_{H^{-2}}.
\end{align*}
Taking the supremum for $t\in [0,r]$, the latter implies by \cref{prop:well_effective} 
\begin{align*}
    \sup_{t\in [0,r]}\|h^{\tau}_t \|_{H^{-2}}
    &\lesssim \int_0^r \|h^{\tau}_s \|_{H^{-2}} ds+\tau+\sqrt{\tau}\sup_{t\in [0,T]}\|Z^1_t\|_{H^{-2}}
    +\sqrt{\tau}\sup_{t\in [0,T]}\|Z^2_t\|_{H^{-2}}\\
    &\lesssim \int_0^r \sup_{t\in [0,s]}
    \|h^{\tau}_t \|_{H^{-2}} ds+\tau+\sqrt{\tau}\sup_{t\in [0,T]}\|Z^1_t\|_{H^{-2}}
    +\sqrt{\tau}\sup_{t\in [0,T]}\|Z^2_t\|_{H^{-2}}.
\end{align*}
Therefore, by Gr\"onwall's inequality, for each $q\geq1$,
\begin{align*}
    \expt{\sup_{t\in [0,r]}\|\Psi^{\tau}_t-\bar{\psi}_t \|_{H^{-2}}^q}
    &\lesssim_{\psi_0,q}\tau^q
    +\tau^{q/2}\expt{\sup_{t\in [0,T]}\|Z^1_t\|_{H^{-2}}^q+\sup_{t\in [0,T]}\|Z^2_t\|_{H^{-2}}^q}.
\end{align*}
The two stochastic integrals can be controlled by \cite[Theorem 1, Remark 2]{tubaro1984estimate} obtaining from \cref{prop:well_effective} for $q\geq 2$
\begin{align*}
    \expt{\sup_{t\in [0,T]}\|Z^1_t\|_{H^{-1}}^q}&\lesssim_{q,T} \expt{\left(\sum_{k\in I}\int_0^T\|\sigma_k\cdot\nabla\Psi^{\tau}_s\|_{H^{-1}}^2 ds\right)^{q/2}}\\ & \lesssim_{q,T}  \|\psi_0\|^q\left(\sum_{k\in I}\|\sigma_k\|_{L^{\infty}}^2 \right)^{q/2},\\
     \expt{\sup_{t\in [0,T]}\|Z^2_t\|_{H^{-1}}^q}&\lesssim_{q,T} \expt{\left(\sum_{k,j\in I}\int_0^T\|\sigma_k\cdot\sigma_j \Psi^{\tau}_s\|_{H^{-1}}^2 ds\right)^{q/2}}\\
     &\lesssim_{q,T} \|\psi_0\|^q\left(\sum_{k,j\in I}\|\sigma_k\|_{L^{\infty}}^2\|\sigma_j\|_{L^{\infty}}^2 \right)^{q/2}.
\end{align*}
The estimates for $1\leq q<2$ follow from the case $q=2$ by Jensen's
inequality.
Therefore we proved
\begin{align}\label{eq:apriori_diff_h-1}
	     \expt{\sup_{t\in [0,r]}
	     \|\Psi^{\tau}_t-\bar{\psi}_t \|_{H^{-2}}^q}
	     \lesssim \tau^{q/2}.
\end{align}
Interpolating between \eqref{eq:apriori_diff_h1} and \eqref{eq:apriori_diff_h-1} we get for each $\theta_1\in [-2,0)$ and $q\geq 1$
\begin{align*}
   \expt{\|\Psi^{\tau}-\bar{\psi} \|_{L^\infty(0,T;H^{\theta_1})}^q}\lesssim_{\theta_1,T,q,\psi_0} \tau^{-q\theta_1 /4}.
\end{align*}
Since $\Psi^{\tau}-\bar{\psi}\in C([0,T];L^2)\hookrightarrow C([0,T];H^{\theta_1})\ \mathbb{P}-a.s.$ the essential supremum in the inequality above can be upgraded to a supremum and the claim follows.
\subsection{\texorpdfstring{Proof of \eqref{eq_rate_fluctuctions}}{Proof of the fluctuation rate}}
First observe that, by definition, $\Xi^{\tau}$ satisfies for each $\phi\in \mathscr{S}(\R^d),\ t\in [0,T]$
\begin{align*}
  i \langle\Xi^{\tau}_t,\phi\rangle
  &=\int_0^t\langle \Xi^{\tau}_s,(-\Delta+V)\phi\rangle ds\\
  &\quad+4i\sqrt{\tau}\int_0^t
  \langle \Psi^{\tau}_s,\operatorname{div}(Q\nabla\phi)\rangle ds
  -i\sqrt{\tau}\int_0^t\langle R\Psi^{\tau}_s,\phi\rangle ds\\
  &\quad-2\sqrt{2}i\sum_{k\in I}\int_0^t
  \langle\sigma_k\Psi^{\tau}_s,\nabla\phi\rangle dW^k_s+\sqrt{2}\sum_{k,j\in I}\int_0^t
  \langle(\sigma_k\cdot\sigma_j)\Psi^{\tau}_s,\phi\rangle dB^{k,j}_s
  \quad\mathbb{P}-a.s.
\end{align*}
By standard arguments, see for example \cite[Section 5.2]{da2014stochastic}, it holds
\begin{align*}
    \Xi^{\tau}_t&=-i\int_0^tS(t-s)[V\Xi^{\tau}_s]ds+\sqrt{\tau}\int_0^t S(t-s)[4\operatorname{div}(Q\nabla\Psi^{\tau}_s)-R\Psi^{\tau}_s]ds\\&+2\sqrt{2}\sum_{k\in I}\int_0^t S(t-s)[\sigma_k\cdot\nabla\Psi^{\tau}_s]dW^k_s-\sqrt{2}i\sum_{k,j\in I}\int_0^t S(t-s)[\sigma_k\cdot\sigma_j \Psi^{\tau}_s]dB^{k,j}_s,\\
    \xi_t&=-i\int_0^tS(t-s)[V\xi_s]ds\\&+2\sqrt{2}\sum_{k\in I}\int_0^t S(t-s)[\sigma_k\cdot\nabla\bar\psi_s]dW^k_s-\sqrt{2}i\sum_{k,j\in I}\int_0^t S(t-s)[\sigma_k\cdot\sigma_j \bar\psi_s]dB^{k,j}_s
\end{align*}
$\mathbb{P}-a.s.$ seen as equalities in $H^{-2}.$ Therefore, calling $\gamma^{\tau}_t=\Xi^{\tau}_t-\xi_t $ and recalling the definition of $h^{\tau}_t$ from the previous section, we have
\begin{align*}
    \gamma^{\tau}_t
    &=-i\int_0^tS(t-s)[V\gamma^{\tau}_s]ds\\
    &\quad+\sqrt{\tau}\int_0^t S(t-s)
    [4\operatorname{div}(Q\nabla\Psi^{\tau}_s)-R\Psi^{\tau}_s]ds
    +\tilde{Z}^{1,\tau}_t+\tilde{Z}^{2,\tau}_t,
\end{align*}
    where
    \begin{align*}
    \tilde{Z}^{1,\tau}_t&=2\sqrt{2}\sum_{k\in I}\int_0^t S(t-s)[\sigma_k\cdot\nabla h^{\tau}_s]dW^k_s,\\ \tilde{Z}^{2,\tau}_t&=-\sqrt{2}i\sum_{k,j\in I}\int_0^t S(t-s)[\sigma_k\cdot\sigma_j h^{\tau}_s]dB^{k,j}_s.
\end{align*}
Evaluating the $H^{-2}$ norm of the equation above we obtain by simple computations
\begin{align*}
   \|\gamma^{\tau}_t \|_{H^{-2}}
   &\lesssim \int_0^t \|\gamma^{\tau}_s \|_{H^{-2}} ds
   +\sqrt{\tau} \int_0^T \|\Psi^{\tau}_s\| ds+\sup_{t\in [0,T]}\|\tilde{Z}^{1,\tau}_t\|_{H^{-2}}
   +\sup_{t\in [0,T]}\|\tilde{Z}^{2,\tau}_t\|_{H^{-2}}.
\end{align*}
By computations analogous to those of the previous subsection, the latter implies that for each $q\geq 1$
\begin{align*}
    \expt{\sup_{t\in [0,T]}\|\Xi^{\tau}_t-\xi_t \|_{H^{-2}}^q}
    &\lesssim_{q,T,\psi_0}\tau^{q/2}
    +\expt{\sup_{t\in [0,T]}
    \|\tilde{Z}^{1,\tau}_t\|_{H^{-2}}^q+\sup_{t\in [0,T]}
    \|\tilde{Z}^{2,\tau}_t\|_{H^{-2}}^q}.
\end{align*}
Repeating the same computations in $H^{-3}$ also gives
\begin{align*}
    \expt{\sup_{t\in [0,T]}\|\Xi^{\tau}_t-\xi_t \|_{H^{-3}}^q}
    &\lesssim_{q,T,\psi_0}\tau^{q/2}
    +\expt{\sup_{t\in [0,T]}
    \|\tilde{Z}^{1,\tau}_t\|_{H^{-3}}^q+\sup_{t\in [0,T]}
    \|\tilde{Z}^{2,\tau}_t\|_{H^{-3}}^q}.
\end{align*}
By \cite[Theorem 1, Remark 2]{tubaro1984estimate} and \eqref{eq_rate_averaging} we get for $q\geq 2$ 
\begin{align*}
    \expt{\sup_{t\in [0,T]}\|\tilde{Z}^{1,\tau}_t\|_{H^{-2}}^q}&\lesssim_{q,T} \expt{\left(\sum_{k\in I}\int_0^T\|\sigma_k\cdot\nabla h^{\tau}_s\|_{H^{-2}}^2 ds\right)^{q/2}}\\ & \lesssim_{q,T}  \left(\sum_{k\in I}\|\sigma_k\|_{W^{1,\infty}}^2 \right)^{q/2}\expt{\sup_{t\in [0,T]}\|\Psi^{\tau}_t-\bar\psi_t\|_{H^{-1}}^q}\\ & \lesssim_{q,T,\psi_0} \tau^{q/4},\\
    \expt{\sup_{t\in [0,T]}\|\tilde{Z}^{1,\tau}_t\|_{H^{-3}}^q}&\lesssim_{q,T} \expt{\left(\sum_{k\in I}\int_0^T\|\sigma_k\cdot\nabla h^{\tau}_s\|_{H^{-3}}^2 ds\right)^{q/2}}\\ &\lesssim \expt{\left(\sum_{k\in I}\int_0^T\|\sigma_k h^{\tau}_s\|_{H^{-2}}^2 ds\right)^{q/2}}\\ & \lesssim_{q,T}  \left(\sum_{k\in I}\|\sigma_k\|_{W^{2,\infty}}^2 \right)^{q/2}\expt{\sup_{t\in [0,T]}\|\Psi^{\tau}_t-\bar\psi_t\|_{H^{-2}}^q}\\ & \lesssim_{q,T,\psi_0} \tau^{q/2},
    \\
     \expt{\sup_{t\in [0,T]}\|\tilde{Z}^{2,\tau}_t\|_{H^{-2}}^q}&\lesssim_{q,T} \expt{\left(\sum_{k,j\in I}\int_0^T\|\sigma_k\cdot\sigma_j h^{\tau}_s\|_{H^{-2}}^2 ds\right)^{q/2}}\\
	     &\lesssim_{q,T} \left(\sum_{k,j\in I}\|\sigma_k\cdot\sigma_j\|_{W^{2,\infty}}^2 \right)^{q/2}\expt{\sup_{t\in [0,T]}\|\Psi^{\tau}_t-\bar\psi_t\|_{H^{-2}}^q}\\ & \lesssim_{q,T,\psi_0} \tau^{q/2}.
\end{align*}
Again, the case $1\leq q<2$ follows by Jensen's inequality.
The relations above imply
\begin{align}\label{estimate_h-1clt}
    \expt{\sup_{t\in [0,T]}\|\Xi^{\tau}_t-\xi_t \|_{H^{-2}}^q}&\lesssim_{q,T,\psi_0} \tau^{q/4},\\ 
    \expt{\sup_{t\in [0,T]}\|\Xi^{\tau}_t-\xi_t \|_{H^{-3}}^q}&\lesssim_{q,T,\psi_0} \tau^{q/2}. \label{estimate_h-2clt}
\end{align}
Interpolating between \eqref{estimate_h-1clt} and \eqref{estimate_h-2clt} the claim follows in the regime $\theta_2\in [-3,-2]$ arguing as in the previous step. In order to reach the regime $\theta_2\in [-2,-\frac{5}{3})$, simply interpolate between \eqref{estimate_h-1clt} and the bound
\begin{align*}
     \expt{\sup_{t\in [0,T]}\|\Xi^{\tau}_t-\xi_t \|_{H^{-1}}^q}&\lesssim_{q,T,\psi_0} \tau^{-q/2}
\end{align*}
Indeed, $\|\Xi^\tau_t\|\leq 2\|\psi_0\|/\sqrt{\tau}$ by
\eqref{bounds_effective} and conservation of the norm for $\bar\psi$.
Moreover, since $\bar{\psi}\in C([0,T];L^2)$, the standard stochastic-convolution
estimate applied to \eqref{eq:gaussianlimit} gives
\(\mathbb E[\sup_{t\in[0,T]}\|\xi_t\|_{H^{-1}}^q]<\infty\).
\section{Large Deviation Principle of the Effective Model}\label{sec_ldp}
Contrary to \eqref{eq:rmse}, it is quite easy to study the probability of rare events for the effective model \eqref{eq:simplified_model} by large-deviation techniques.
This is the content of the main result of this section. In order to rigorously state the result we need to set some notation. Let us recall that
\begin{align*}
    W_t&=\sqrt{8}\sum_{k\in I}\sigma_k W^k,\qquad
    B_t=\sqrt{2}\sum_{k,j\in I}\sigma_k\cdot\sigma_j B^{k,j}
\end{align*}
and, by \cref{hp:noise}, $W_t$ (resp. $B_t$) is a Brownian motion on 
\begin{align*}
    H^{5+\frac{d+\theta}{2}}_{div}:=\left\{u\in H^{5+\frac{d+\theta}{2}}(\R^d;\R^d):\ \operatorname{div}u=0 \right\},
\end{align*}
(resp. $H^{5+\frac{d+\theta}{2}}(\R^d;\R),$) with covariance operator $\mathcal{Q}_1$ (resp. $\mathcal{Q}_2$) which is trace class on $H^{5+\frac{d+\theta}{2}}_{div}$ (resp. $H^{5+\frac{d+\theta}{2}}(\R^d;\R)$).
We denote by
\begin{align*}
\mathcal{H}_{0,1}
=\mathcal{Q}_1^{1/2}H^{5+\frac{d+\theta}{2}}_{div},\quad
\mathcal{H}_{0,2}
=\mathcal{Q}_2^{1/2}H^{5+\frac{d+\theta}{2}},
\end{align*}
the respective Cameron--Martin spaces of the noises, with inner products
\begin{align*}
    \langle u,v\rangle_{\mathcal{H}_{0,1}}
    &=\langle \mathcal{Q}_1^{-1/2}u,\mathcal{Q}_1^{-1/2}v
    \rangle_{H^{5+\frac{d+\theta}{2}}_{div}},\\
    \langle u,v\rangle_{\mathcal{H}_{0,2}}
    &=\langle \mathcal{Q}_2^{-1/2}u,\mathcal{Q}_2^{-1/2}v
    \rangle_{H^{5+\frac{d+\theta}{2}}}.
\end{align*}
where the inverse square roots are understood on the corresponding ranges.
As is well known, cf. \cite[Chapter 2]{da2014stochastic}, we have
\begin{align}\label{embedding_cameron_martin}
  \mathcal{H}_{0,1}\hookrightarrow  H^{5+\frac{d+\theta}{2}}_{div},\quad \mathcal{H}_{0,2}\hookrightarrow H^{5+\frac{d+\theta}{2}}. 
\end{align}
To accommodate the weak convergence approach of \cite{LDPweak_convergence}, we also define the spaces \begin{align*}
  S^N_1&:=\Big\{v\in L^2(0,T;\mathcal{H}_{0,1}):\quad \int_0^T \lVert v_s\rVert_{\mathcal{H}_{0,1}}^2 ds \leq N\Big\},\\
  S^N_2&:=\Big\{v\in L^2(0,T;\mathcal{H}_{0,2}):\quad \int_0^T \lVert v_s\rVert_{\mathcal{H}_{0,2}}^2 ds \leq N\Big\}
\end{align*}
which are Polish spaces when endowed with the weak topology. 
We denote by $\mathcal{P}_{1}$ (resp. $\mathcal{P}_{2}$) the space of $\mathcal{H}_{0,1}$-valued (resp. $\mathcal{H}_{0,2}$-valued), $\mathcal{F}_t$-predictable and $\mathbb{P}-a.s.$ square integrable processes. Next we define 
\begin{align*}
    \mathcal{P}_1^N
    &:=\{\phi\in \mathcal{P}_1:
    \ \phi(\omega)\in S_1^N\quad \mathbb{P}-a.s.\},\\
    \mathcal{P}_2^N
    &:=\{\phi\in \mathcal{P}_2:
    \ \phi(\omega)\in S_2^N\quad \mathbb{P}-a.s.\}.
\end{align*}
Fix $R>0$ and set $\mathcal{E}_0:=B^{L^2}_R$. We aim to show a large deviation principle for \eqref{eq:simplified_model} in the Polish space
\begin{align*}
    \mathcal{E}:=C([0,T];\tilde{H}^{-})\cap C([0,T];B^{L^2}_{R,w}).
\end{align*}
In the following to shorten the notation we will denote by $\Psi^{\tau,\psi_0,f,g}$ (resp. $\psi^{\psi_0,f,g}$) the unique weak solution in $\mathcal{E}$
of 
\begin{equation}\label{eq:simplified_model_ldp}
     \begin{cases}
         i d\Psi&=(-\Delta+V+if\cdot\nabla+g)\Psi dt+i\sqrt{8\tau}\sum_{k\in I}\sigma_k\cdot\nabla \Psi\circ dW^k  +\sqrt{\tau} \Psi\circ dB_t\\ \Psi(0)&=\psi_0
     \end{cases}
\end{equation}
\begin{equation}\label{eq:limit_model_ldp}\mbox{(resp. }
     \begin{cases}
         i \partial_t\psi&=(-\Delta+V+if\cdot\nabla+g)\psi \\ \psi(0)&=\psi_0
     \end{cases}\mbox{)}
\end{equation}
for $(f,g)\in \mathcal{P}_1^N\times \mathcal{P}_2^N$ given by \cref{well_simplified_ldp} (resp. \cref{well_limit_ldp}) below which is a minor adaptation of \cref{prop:well_effective} (resp. \cref{prop:well_confinement}). In case of $f=g=0$ we will drop them from the superscripts.
\begin{proposition}\label[proposition]{prop:LDP}
   Assume \cref{hp:noise} and fix $R>0$. The family of processes $\Psi^{\tau,\psi_0},\ \tau\in (0,1),\ \psi_0\in \mathcal{E}_0$, solving \eqref{eq:simplified_model} with initial condition $\psi_0$, satisfies a Laplace principle on $\mathcal{E}$ uniformly on compact subsets of $\mathcal E_0$, at scale $\tau$ (equivalently, with speed $1/\tau$), with rate function
   \begin{align}\label{rate_function}
       I_{\psi_0}(\psi):=\inf\bigg\{\frac{1}{2}\|f\|_{L^2(0,T;\mathcal{H}_{0,1})}^2+&\frac{1}{2}\|g\|^2_{L^2(0,T;\mathcal{H}_{0,2})}:\notag\\ &\hspace{-1cm} f\in L^2(0,T;\mathcal{H}_{0,1}),\ g\in L^2(0,T;\mathcal{H}_{0,2}),\ \psi=\psi^{\psi_0,f,g} \bigg\}.
   \end{align}
In particular, $\Psi^{\tau,\psi_0}$ also satisfies a large-deviation principle on $\mathcal{E}$ uniformly on compact subsets of $\mathcal E_0$, with speed $1/\tau$ and the same rate function.
\end{proposition}
Due to the weak convergence approach developed in \cite{LDPweak_convergence}, the validity of \cref{prop:LDP} immediately follows from the verification of the following properties:
\begin{enumerate}
    \item Well-posedness in $\mathcal{E}$ of both \eqref{eq:simplified_model_ldp} and \eqref{eq:limit_model_ldp}.
    \item Compactness for \eqref{eq:limit_model_ldp}, i.e. if $K\subset \mathcal{E}_0$ is compact, then \begin{align*}
        \Gamma_{K,N}:=\{\psi^{\psi_0,f,g}\mbox{ for }\psi_0\in K,\ f\in S_1^N,\ g\in S_2^N\}
    \end{align*} 
    is a compact subset of $\mathcal{E}$.
    \item Stability of \eqref{eq:simplified_model_ldp} and \eqref{eq:limit_model_ldp}, i.e. if $\psi_0^\tau\rightarrow \psi_0$ in $\mathcal{E}_0$ and $(f^\tau,g^\tau)$ converges in law to  $(f,g)$ as $(S_1^N,\ S_2^N)$ random variables, then $\Psi^{\tau,\psi^{\tau}_0,f^\tau,g^\tau}$ converges in law to $\psi^{\psi_0,f,g}$ as $\mathcal{E}$ random variables. 
    \item Lower semi-continuity of the map $\psi_0\rightarrow I_{\psi_0}(\psi)$ from $\mathcal{E}_0$ to $[0,+\infty]$ for $\psi\in \mathcal{E}$ fixed.
\end{enumerate}
We prove each of the items above in the following four subsections.
\subsection{Well-posedness}
We start showing well-posedness for \eqref{eq:simplified_model_ldp}. This is the content of the next proposition.
\begin{proposition}\label[proposition]{well_simplified_ldp}
    Assume \cref{hp:noise}, $\psi_0\in B^{L^2}_R$, and
    $(f,g)\in \mathcal{P}_1^N\times \mathcal{P}_2^N$. Then there exists a
    unique probabilistically strong, analytically weak solution of
    \eqref{eq:simplified_model_ldp} with trajectories in $\mathcal{E}$
    $\mathbb{P}-a.s.$ Moreover, it satisfies
    \begin{align*}
        \|\Psi^{\tau,\psi_0,f,g}_t\|^2=\|\psi_0\|^2 \quad \text{for all }t\geq 0\quad  \mathbb{P}-a.s.
    \end{align*}
   In particular, $\Psi^{\tau,\psi_0,f,g}$ belongs to $C([0,T];L^2)\ \mathbb{P}-a.s.$
\end{proposition}
\begin{proof}
One can prove this proposition following closely the arguments developed in \cref{prop:well_effective}. Here we prefer to provide a different proof exploiting Girsanov theorem, which is more in the spirit of the weak convergence approach for large deviations, cf. \cite[Theorem 10]{LDPweak_convergence}, \cite[Proposition 24]{fehrman2023non}.\\
Everything holds at $\tau,\psi_0$ fixed, so we drop their dependence to simplify the notation.
The case $f=g=0$ follows by \cref{prop:well_effective}. In the general
case we use Girsanov's theorem, cf. \cite[Chapter 10]{da2014stochastic}.
Let us consider a filtered probability space $(\Omega,\mathcal{F},\{\mathcal{F}_t\}_{t\geq 0},\mathbb{P})$ carrying $W, B$ independent Wiener processes adapted to $\{\mathcal{F}_t\}_{t \geq 0}$ as described before \cref{prop:LDP} and $f,g$ as in the statement. Also denote \begin{align*}
    \gamma_{f,g}:=\exp\Bigg(-\frac{1}{\sqrt{\tau}}\int_0^T \langle f_s,dW_s\rangle_{\mathcal{H}_{0,1}}&-\frac{1}{\sqrt{\tau}}\int_0^T \langle g_s,dB_s\rangle_{\mathcal{H}_{0,2}}\\ & -\frac{1}{2\tau}\int_0^T \norm{f_s}_{\mathcal{H}_{0,1}}^2 ds-\frac{1}{2\tau}\int_0^T \norm{g_s}_{\mathcal{H}_{0,2}}^2 ds\Bigg).
\end{align*}
The energy bound in the definition of $\mathcal P_1^N\times\mathcal
P_2^N$ ensures Novikov's condition.
By the Girsanov theorem, there exists a probability measure $\mathbb{P}_{f,g}$ on $(\Omega,\mathcal{F})$ such that 
\begin{align*}
    \frac{d \mathbb{P}_{f,g}}{d\mathbb{P}}=\gamma_{f,g}
\end{align*}
and
\begin{equation*}
     \hat{W}_t:=W_t+\frac{1}{\sqrt{\tau}}\int_0^t f_s\,ds,
 \qquad
 \hat{B}_t:=B_t+\frac{1}{\sqrt{\tau}}\int_0^t g_s\,ds
\end{equation*}
are Wiener processes with covariance operators $\mathcal{Q}_1$ and
$\mathcal{Q}_2$, respectively, on the auxiliary filtered probability
space $(\Omega,\mathcal{F},\{\mathcal{F}_t\}_{t\geq 0},\mathbb{P}_{f,g})$.
Moreover, the two probability measures $\mathbb{P}$ and $\mathbb{P}_{f,g}$ are equivalent. By \cref{prop:well_effective}, on the auxiliary probability space $(\Omega,\mathcal{F},\{\mathcal{F}_t\}_{t\geq 0},\mathbb{P}_{f,g})$ there exists a unique $\Psi$ adapted satisfying \begin{align*}
        \|\Psi_t\|^2=\|\psi_0\|^2 \quad \text{for all }t\geq 0\quad  \mathbb{P}_{f,g}-a.s.,
    \end{align*}
   with trajectories in $\mathcal{E}$ $\PP_{f,g}$-almost surely, solving
   \eqref{eq:simplified_model} with initial condition $\psi_0$ and noise
   $\hat{W},\hat{B}$. Therefore, on our original probability space, $\Psi$
   solves \eqref{eq:simplified_model_ldp} with initial condition $\psi_0$,
   Wiener processes $W,B$, and the additional terms $f,g$. Moreover, since
   $\mathbb{P}_{f,g}$ and $\mathbb{P}$ are equivalent,
  it also holds
  \begin{align*}
        \|\Psi_t\|^2=\|\psi_0\|^2 \quad \text{for all }t\geq 0\quad  \mathbb{P}-a.s.
    \end{align*}
  and $\Psi$ has trajectories in $\mathcal{E}$ $\PP$-almost surely. This completes the existence part. Concerning uniqueness, we argue similarly. By linearity it is enough to show uniqueness assuming $\Psi_0=0$ without enforcing \eqref{bounds_effective}. Let $\Psi $ with trajectories in $\mathcal{E}$ $\PP$-almost surely solving \eqref{eq:simplified_model_ldp} on $(\Omega,\mathcal{F},\{\mathcal{F}_t\}_{t\geq 0},\mathbb{P})$ with null initial condition and satisfying \eqref{boundedness_spde}. Therefore $\Psi$ solves \eqref{eq:simplified_model} on the auxiliary probability space $(\Omega,\mathcal{F},\{\mathcal{F}_t\}_{t\geq 0},\mathbb{P}_{f,g})$, has trajectories in $\mathcal{E}$ and satisfies \eqref{boundedness_spde} $\PP_{f,g}$-almost surely. By \cref{prop:well_effective} it follows that $\Psi\equiv 0$, $\mathbb{P}_{f,g}$-almost surely. Since $\PP_{f,g}$ and $\PP$ are equivalent, $\Psi\equiv 0$ almost surely with respect to $\mathbb{P}$ and the result follows.
\end{proof}
Next we show well-posedness of \eqref{eq:limit_model_ldp}.
\begin{proposition}\label[proposition]{well_limit_ldp}
    Assume \cref{hp:noise}, $\psi_0\in B^{L^2}_R$, and
    $(f,g)\in \mathcal{P}_1^N\times \mathcal{P}_2^N$. Then there exists a
    unique probabilistically strong, analytically weak solution of
    \eqref{eq:limit_model_ldp} with trajectories in $\mathcal{E}$
    $\mathbb{P}-a.s.$ Moreover, it satisfies
    \begin{align*}
        \|\psi^{\psi_0,f,g}_t\|^2=\|\psi_0\|^2 \quad \text{for all }t\geq 0\quad  \mathbb{P}-a.s.
    \end{align*}
   In particular, $\psi^{\psi_0,f,g}$ has paths in $C([0,T];L^2)\ \mathbb{P}-a.s.$
\end{proposition}
\begin{proof}
   We only sketch the argument, which is analogous to that for
   \cref{prop:well_confinement}. Existence of solutions can be obtained by a
   vanishing-viscosity scheme, considering
   \begin{align}\label{viscous_LDP}
       \begin{cases}
         i \partial_t\psi^{\nu}&=(-i\nu\Delta^2-\Delta+V+if\cdot\nabla+g)\psi^{\nu} \\ \psi^{\nu}(0)&=\psi_0
     \end{cases}
   \end{align}
   for which existence of solutions in $C([0,T];L^2)\cap L^2(0,T;H^2)$ is classical. In particular, by \eqref{embedding_cameron_martin}, the solution satisfies for each $s,t\in [0,T]$
   \begin{align*}
       \|\psi^{\nu}_t\|&\leq \|\psi_0\|\quad \mathbb{P}-a.s.\\
	   \|\psi^\nu_t-\psi^\nu_s\|_{H^{-4}}
	   &\lesssim_N |t-s|^{1/2}\|\psi_0\|
	   \quad \mathbb{P}-a.s.
	   \end{align*}
	   uniformly in $\nu$. Then, arguing as in the proof of
	   \cref{prop:well_confinement}, it is possible to pass to the limit as
	   $\nu\rightarrow0$ and show the existence of a weak solution of
	   \eqref{eq:limit_model_ldp} in $\mathcal{E}$. Uniqueness and the
	   remaining claims follow from commutator arguments in the spirit of
	   \cite{diperna1989ordinary}; see \cref{subsec_well_posed_fast_slow}
	   for details in our setting. In particular, every weak solution of
	   \eqref{eq:limit_model_ldp} preserves its $L^2$ norm.
\end{proof}

\subsection{Compactness}\label{compactness_ldp}
Since $\mathcal{E}$ is a Polish space it is enough to show sequential compactness. Let $\psi^n\in \Gamma_{K,N}$ and $\psi_0^n,\ f^n,g^n$ the corresponding initial condition and additional linear terms. By compactness of $K,\ S_1^N$ and $S_2^N$, up to passing to non-relabeled subsequences, we can assume that \begin{align}\label{convergence_data_compactness_lemma}
        &\psi_0^n\rightarrow \psi_0\in K,
        \quad\\ \label{convergence_data_compactness_lemma_1}
        & f^n\rightharpoonup f\in S_1^N,\qquad 
        g^n\rightharpoonup g\in S_2^N.  
    \end{align}
    Let $\psi^{\psi_0,f,g}\in \mathcal{E}$ be the unique weak solution of \eqref{eq:limit_model_ldp} given by \cref{well_limit_ldp}. We are left to show that
    \begin{align*}
\lim_{n\rightarrow+\infty}d_{\mathcal{E}}\left(\psi^n,\psi^{\psi_0,f,g}\right)=0.
    \end{align*}
    By \cref{well_limit_ldp} we already know that 
    \begin{align}
       \sup_{n\geq 1} \sup_{t\in [0,T]}\|\psi^n_t\|\leq R.
    \end{align}
    Moreover, using also \eqref{embedding_cameron_martin}, it is easy to show, arguing as in \cref{lem_time_increments_viscous},
	    \begin{align*}
	        \|\psi^n_t-\psi^n_s\|_{H^{-2}}
	        &\leq |t-s|^{1/2}\bar{R},
	        \qquad n\geq1,\quad s,t\in[0,T],
	    \end{align*}
    for some $\bar{R}$ depending on $N,\ T$ and $R.$  Let $\mathbb{B}_{\bar{R}}$ denote the closure in $\mathcal{E}$ of the intersection between $\mathcal{E}$ and the ball of radius ${\bar{R}}$ centered at $0$ in $C^{1/3}([0,T];\tilde{H}^{-2})$.
    By \cite[Lemma 2.2]{crippa2025zero}, the set
    $\mathbb{B}_{\bar{R}}$ is compact in $\mathcal{E}$. Therefore, up
    to passing to a non-relabeled subsequence, there exists
    $\psi\in \mathcal{E}$ such that
    \begin{align}\label{convergence_stability}
\lim_{n\rightarrow+\infty}d_{\mathcal{E}}\left(\psi^n,\psi\right)=0.
    \end{align}
    It remains to show that $\psi$ is a weak solution of
    \eqref{eq:limit_model_ldp}. Uniqueness will then imply
    $\psi=\psi^{\psi_0,f,g}$. Let $\phi\in\mathscr{S}$. For every
    $t\in[0,T]$ and $n\geq1$,
	    \begin{align*}
	        &i\langle \psi^n_t,\phi\rangle-i\langle \psi^n_0,\phi\rangle
	        -\int_0^t\langle \psi^n_s,(-\Delta +V)\phi\rangle ds=-\int_0^t
	        \left(i\langle f^n_s\psi^n_s,\nabla \phi\rangle
	        -\langle g^n_s\psi^n_s,\phi\rangle\right)ds.
	    \end{align*}
    From \eqref{convergence_data_compactness_lemma} and
    \eqref{convergence_stability},
	    \begin{align*}
	     &\lim_{n\rightarrow +\infty}\left(
	     i\langle \psi^n_t,\phi\rangle-i\langle \psi^n_0,\phi\rangle
	     -\int_0^t\langle \psi^n_s,(-\Delta +V)\phi\rangle ds\right)\\
	     &\quad=i\langle \psi_t,\phi\rangle-i\langle \psi_0,\phi\rangle
	     -\int_0^t\langle \psi_s,(-\Delta +V)\phi\rangle ds.
	    \end{align*}
    The triangle inequality, the spatial decay of $\phi$ and its derivatives,
    \eqref{convergence_data_compactness_lemma_1}, and
    \eqref{convergence_stability} give
    \begin{align*}
      &\lim_{n\rightarrow+\infty}\left\lvert\int_0^t i\langle f^n_s\psi^n_s-f_s\psi_s,\nabla \phi\rangle-\langle g^n_s\psi^n_s-g_s\psi_s,\phi\rangle ds \right\rvert\\ & \leq \limsup_{n\rightarrow+\infty}\left\lvert\int_0^t \langle \psi^n_s-\psi_s,f^n_s\nabla \phi\rangle ds \right\rvert+\limsup_{n\rightarrow+\infty}\left\lvert\int_0^t \langle \psi^n_s-\psi_s,g^n_s  \phi\rangle ds \right\rvert\\ & +\limsup_{n\rightarrow+\infty}\left\lvert\int_0^t \langle f^n_s-f_s ,\psi_s\nabla \phi\rangle ds \right\rvert+\limsup_{n\rightarrow+\infty}\left\lvert\int_0^t \langle g^n_s-g_s, \psi_s  \phi\rangle ds \right\rvert =0,
    \end{align*}
completing the proof.
\subsection{Stability}  From the uniqueness of the limit in law, it is enough to show that for each sequence $\tau_n\rightarrow 0$, there is a subsequence $\tau_{n_l}\rightarrow 0$ such that the claim holds. To save the notation let us call $\Psi^n=\Psi^{\tau_n,\psi^{\tau_n}_0,f^{\tau_n},g^{\tau_n}}$ and $\psi^0=\psi^{\psi_0,f,g}.$
By \cref{well_simplified_ldp} we know that 
\begin{align}\label{uniform_bound_ldp}
    \|\Psi^n_t\|\leq R \quad \mathbb{P}-a.s.
\end{align}
From the latter and \eqref{embedding_cameron_martin}, arguing as in
\cref{lem_time_increments_viscous}, it is easy to show that for each
$\gamma\geq 4$ there is some $\bar R<+\infty$ such that
	    \begin{align}\label{time_compactness_ldp}
	        \sup_{n\geq 1}\mathbb{E}
	        \left[\|\Psi^n_t-\Psi^n_s\|^{\gamma}_{H^{-2}}\right]
	        &\leq |t-s|^{\gamma/2}\bar{R},
	        \qquad s,t\in[0,T].
    \end{align}
    From the convergence in law of $(f^n,g^n)$ to $(f,g)$ and the uniform bounds \eqref{uniform_bound_ldp}, \eqref{time_compactness_ldp} arguing as in \cref{lem_compactness_viscous}, \cref{subsec:well_spde} it follows that 
    the family of laws
    \begin{align*}
        (\Psi^{n},W^n=(W^{k,n})_{k\in I}, B^n=(B^{k,j,n})_{k,j\in I},f^n,g^n)
    \end{align*}
    is tight in $\mathcal{E}\times C([0,T];\R^I)\times
    C([0,T];\R^{I\times I})\times S_1^N\times S_2^N$. Therefore, by the
    Jakubowski version of the Skorokhod representation theorem, we find a
    subsequence $n_l\rightarrow+\infty$. On an auxiliary probability space,
    which for simplicity we continue to call
    $(\Omega,\mathcal{F},\mathbb{P})$, we construct processes
\begin{align*}
&(\Psi^{n_l},W^{n_l}=(W^{k,n_l})_{k\in I}, B^{n_l}=(B^{k,j,n_l})_{k,j\in I},f^{n_l},g^{n_l}),\\
&(\Psi,W=(W^{k})_{k\in I},B=(B^{k,j})_{k,j\in I},f,g)
\end{align*}
such that
\begin{align*}
 \Psi^{n_l}&\rightarrow \Psi \quad\text{ in }\mathcal{E}\quad \mathbb{P}-a.s.,\\
 W^{n_l}&\rightarrow W \text{ in }C([0,T];\R^I)\quad \mathbb{P}-a.s.\\
  B^{n_l}&\rightarrow B \text{ in }C([0,T];\R^{I\times I})\quad \mathbb{P}-a.s.\\
  f^{n_l}&\rightarrow f \text{ in }S_1^N\quad \mathbb{P}-a.s.\\
  g^{n_l}&\rightarrow g \text{ in }S^N_2\quad \mathbb{P}-a.s.\\
\end{align*}
 Moreover, $\Psi^{n_l}$ is the unique solution of \eqref{eq:simplified_model_ldp} for $\tau=\tau_{n_l}$ and additional terms $f^{n_l},\ g^{n_l}$ given by \cref{well_simplified_ldp}. By linearity of the equations and the above convergence, arguing as in the previous subsection it is easy to check that $\Psi$ satisfies
 \begin{align*}
     \sup_{t\in [0,T]}\|\Psi_t\|\leq \|\psi_0\|\quad \mathbb{P}-a.s.
 \end{align*}
   and that, for every $\phi\in\mathscr{S}(\R^d)$ and $t\in[0,T]$,
   \begin{align*}
      i\langle \Psi_t,\phi\rangle&=i\langle \psi_0,\phi\rangle+\int_0^t \langle\Psi_s, (-\Delta+V)\phi\rangle ds\\
      &\quad+\int_0^t\left(-i\langle f_s\Psi_s,\nabla \phi\rangle+\langle g_s\Psi_s,\phi\rangle\right)ds\quad \mathbb{P}-a.s.
   \end{align*}
   Therefore $\mathcal{L}(\Psi)=\mathcal{L}(\psi^0)$ and the proof is complete.
\subsection{Lower semi-continuity}
First, let us observe that the infimum appearing in \eqref{rate_function} is attained, if finite. 
    Indeed, let $I_{\psi_0}(\psi)=M<+\infty$ and $(f^n,g^n)$ be a minimizing sequence. 
    Without loss of generality, we can assume
    \begin{align*}
        2M\leq\lVert f^n \rVert_{L^2(0,T;\mathcal{H}_{0,1})}^2
        +\lVert g^n \rVert_{L^2(0,T;\mathcal{H}_{0,2})}^2
        \leq 2M+1
    \end{align*}
	    and 
	    \begin{align*}
	        \lVert f^n \rVert_{L^2(0,T;\mathcal{H}_{0,1})}^2
	        +\lVert g^n \rVert_{L^2(0,T;\mathcal{H}_{0,2})}^2
	        \rightarrow 2M.
	    \end{align*}
	    Then, up to passing to a non-relabeled subsequence, there exists
	    \begin{align*}
	    (f,g)\in L^2(0,T;\mathcal{H}_{0,1})
	    \times L^2(0,T;\mathcal{H}_{0,2})
	    \end{align*}
	    such that
	    \begin{align*}
	       f^n&\rightharpoonup f
	       &&\text{in }L^2(0,T;\mathcal{H}_{0,1}),\\
	       g^n&\rightharpoonup g
	       &&\text{in }L^2(0,T;\mathcal{H}_{0,2}).
	    \end{align*}
	    Moreover,
	    \begin{align*}
	       \|f\|_{L^2(0,T;\mathcal H_{0,1})}^2
	       +\|g\|_{L^2(0,T;\mathcal H_{0,2})}^2
	       &\leq\liminf_{n\to\infty}
	       \left(
	       \|f^n\|_{L^2(0,T;\mathcal H_{0,1})}^2
	       +\|g^n\|_{L^2(0,T;\mathcal H_{0,2})}^2
	       \right)\\
	       &=2M.
    \end{align*}
    Due to the convergence of $(f^n,g^n)$ to $(f,g)$,
    \cref{compactness_ldp} shows that $\psi$ is the unique weak solution of
    \eqref{eq:limit_model_ldp} with initial condition $\psi_0$ and
    additional terms $(f,g)$.
	    Therefore $f$ and $g$ realize the infimum in \eqref{rate_function}.
    
    Now we are ready to prove the lower-semicontinuity of the map $\psi_0\rightarrow I_{\psi_0}(\psi)$. 
    Fix $\psi_0\in \mathcal{E}_0$ and a family $\{\psi_0^n\}_{n\in \mathbb{N}}\subset \mathcal{E}_0 $ converging to $\psi_0$. Without loss of generality we assume $\liminf_{n\rightarrow +\infty}I_{\psi_0^n}(\psi)=M<+\infty$, otherwise we have nothing to prove. Therefore, since the infimum in \eqref{rate_function} is attained, there exists a subsequence $n_k$ and family $\{f^{n_k}, g^{n_k}\}_{n_k\in\mathbb{N}}\subset S_1^{2M+1}\times S_2^{2M+1}$ such that, for each $k$, $\psi$ is the unique weak solution of \eqref{eq:limit_model_ldp} with initial condition $\psi_0^{n_k}$ and additional terms $(f^{n_k},g^{n_k})$.
Up to passing to a further subsequence, which we continue to denote by $(f^{n_k},g^{n_k})$ for simplicity of notation, there exists $(f,g)\in S^{2M+1}_1\times S^{2M+1}_2$ such that $f^{n_k}\rightharpoonup f$ in $L^2(0,T; \mathcal{H}_{0,1})$, $g^{n_k}\rightharpoonup g$ in $L^2(0,T; \mathcal{H}_{0,2})$. 

Due to the convergence of $(f^{n_k}, g^{n_k})$ to $(f,g)$ and
$\psi_0^{n_k}\to\psi_0$, the results of \cref{compactness_ldp} show that
$\psi$ is the unique weak solution of \eqref{eq:limit_model_ldp} with
initial condition $\psi_0$ and additional terms $(f,g)$.
Hence, from the lower-semicontinuity of the norm with respect to the weak convergence: 
\begin{align*}
    I_{\psi_0}(\psi)&\leq \frac{1}{2} \lVert f\rVert^2_{L^2(0,T;\mathcal{H}_{0,1})}  +\frac{1}{2} \lVert g\rVert^2_{L^2(0,T;\mathcal{H}_{0,2})}  \\ &
    \leq 
    \liminf_{k\rightarrow +\infty}
    \frac{1}{2}\int_0^T
    \left(\lVert f^{n_k}(t)\rVert^2_{\mathcal{H}_{0,1}}
    +\lVert g^{n_k}(t)\rVert^2_{\mathcal{H}_{0,2}}\right)dt
    \leq M=\liminf_{n\rightarrow +\infty} I_{\psi_0^n}(\psi).
\end{align*}
The proof is complete.

\appendix

\section{Remarks on Related Topics in Statistical and Quantum Mechanics}\label{sec:appendix_formal}

We report here some heuristics supporting our motivation for the present study. We expect that a rigorous mathematical treatment would not be easy, and we are not proposing the latter as an interesting further development of our work, but rather emphasizing the convenience and meaningfulness of the effective SPDE approach we discussed.

\Cref{ssec:ldpsse} outlines a formal derivation of a large-deviation
principle for the starting equation \eqref{eq:rmse} and shows that
\eqref{eq:simplified_model} can be seen as the small-noise dynamics
associated with the quadratic truncation of the large-deviation action.
Even though the most common large-deviation setting is the small-noise
limit, averaging limits of fast processes are also part of the classical
theory \cite[Chapter 7]{freidlinwentzell3}. The computations of rate
functionals in the two settings are analogous (and combinations have also
been considered \cite{veretennikov}), so we follow a paradigm that may be
traced back to the asymptotic expansions of classical statistical mechanics
\cite{touchette2009}.

The last subsection then points out that the effective equation
\eqref{eq:simplified_model} has the same form as the \emph{unraveling
stochastic dynamics} associated with open quantum systems. Our starting
model \eqref{eq:rmse} is not an open quantum system combining an observed
system and a heat bath; instead, it may be considered as a model of the
effect of a random classical magnetic field on a quantum particle (whose
statistics are irrelevant). Here we simply wish to point out that
\eqref{eq:simplified_model} arises as a natural approximation in a
well-established applied context.

\subsection{Large Deviations of Stochastic Schr\"odinger Equation}
\label{ssec:ldpsse}
In this paragraph brackets denote the inner product of $L^2(\R^d;\C)$  as a real Hilbert space with inner product $\brak{\cdot,\cdot}_\R$. The forthcoming computations can be made rigorous for Galerkin truncations of the dynamics under consideration, and we shall speak of ``matrices'' and ``vectors'' as if we were in a finite dimensional context. 
Even basic steps such as manipulating the cumulant generating function for the actual infinite-dimensional model would additionally require exponential tightness and control of the unbounded Schr\"odinger operators, so we emphasize once again that this paragraph must be treated as a heuristic.

Let $a^\tau$ denote the coefficient of $A^\tau$ in the
expansion \(A^\tau=\sum_k a_k^\tau\sigma_k\), and set
\begin{equation*}
    \mathcal F(q,a)
 :=\left(-i\nabla-\sum_k a_k\sigma_k\right)^2q.
\end{equation*}
For slowly varying test paths $p$ and $q$, the relevant scaled cumulant
generating function is the one of the time-additive functional,
\begin{align}\label{eq:lambdacumulant}
 \Lambda(p,q)
 &:=
 \lim_{\tau\to0}\tau\log\mathbb E\!\left[
 \exp\left\{\frac1\tau\int_0^T
 \langle p(t),\mathcal F(q(t),a_t^\tau)\rangle_{\mathbb R}\,dt
 \right\}\right]\notag\\
 &=\int_0^T H(p(t),q(t))\,dt.
\end{align}
Formally, \(H(p,q)\) is the principal eigenvalue of the tilted
Ornstein--Uhlenbeck generator
\begin{equation*}
    \mathcal L_{\mathrm{OU}}
 +\langle p,\mathcal F(q,\cdot)\rangle_{\mathbb R}.
\end{equation*}
In a Galerkin truncation this eigenvalue is obtained with a Gaussian
exponential ansatz.
Explicitly, introducing for brevity the notation
\begin{equation*}
    Q(p,q)_{hk}=\brak{p,q \sigma_h\cdot \sigma_k},
    \quad
    b(p,q)_k=\brak{p,2i\sigma_k\cdot\nabla q},
\end{equation*}
for $p\in L^2$ and $q\in H^1$, and letting
\begin{equation*}
    H(p,q)=\brak{p,-\Delta q}+\frac12 \trace\pa{I-\sqrt{I-4Q(p,q)}}+b(p,q)^t(I-4Q(p,q))^{-1}b(p,q),
\end{equation*}
we obtain \eqref{eq:lambdacumulant} if the largest eigenvalue of the real symmetric matrix $Q(p,q)$ is less than $1/4$; otherwise $H(p,q)=+\infty$.
The Large Deviations rate corresponds to the Legendre transform (integrated in time) of such action, 
\begin{equation*}
    L(y,q)=\sup_p (\brak{p,y}-H(p,q)),
\end{equation*}
which we evaluate formally exploiting the fact that $H(p,q)$ depends on $p$ only through the composition of 
\begin{equation*}
    h(Q,b)=\frac12 \trace\pa{I-\sqrt{I-4Q}}+b^t(I-4Q)^{-1}b
\end{equation*}
and linear functions, and the Legendre transform of $h$ can be explicitly evaluated (optimize first in $b$ and then in $Q$). The result is finite only on a specific constraint: if
\begin{equation*}
    y=-\Delta q
      +2i\sum_{k\in I}u_k\,\sigma_k\cdot\nabla q
      +q\sum_{h,k\in I}(u_hu_k+C_{hk})\,\sigma_h\cdot\sigma_k,
\end{equation*}
for some real vector $u$ and symmetric positive-definite matrix $C$,
then
\begin{equation*}
    L(y,q)
    =\inf_{u,C\text{ as above}}
      \frac14\left(
        |u|^2+\trace(C+C^{-1}-2I)
      \right).
\end{equation*}
Otherwise $L(y,q)=\infty$. Moving to the control form makes the structure of the original dynamics emerge again: for paths with $\psi(0)=\psi_0$, the formal rate functional of \eqref{eq:rmse} is
\begin{align}\label{eq:ldpsse_exact_action}
    I^{\mathrm{RMSE}}_{\psi_0}(\psi)
    =\inf\Bigg\{&\frac14\int_0^T
       \left(|u_t|^2+\trace(C_t+C_t^{-1}-2I)\right)dt:\notag\\[-2mm]
       &i\partial_t\psi
       =-\Delta\psi
        +2i\sum_k u_k\,\sigma_k\cdot\nabla\psi
       +\psi\sum_{h,k}(u_hu_k+C_{hk})\,\sigma_h\cdot\sigma_k
    \Bigg\}.
\end{align}
Here the infimum is over measurable controls $u_t$ and symmetric positive-definite matrices $C_t$ for which the displayed cost is finite. For every positive-definite $C$,
\begin{equation*}
    \trace(C+C^{-1}-2I)
    =\sum_j\left(\sqrt{\lambda_j(C)}-\lambda_j(C)^{-1/2}\right)^2\geq0,
\end{equation*}
thus the unique zero-cost controls are $u=0$ and $C=I$. For fixed initial datum the zero-cost path is therefore the global minimizer, and the constraint in \eqref{eq:ldpsse_exact_action} reduces exactly to the averaged equation \eqref{eq:averaging}.

We now expand the action around this minimum. Since $Q(p,q)$ and $b(p,q)$ are linear in $p$, the Taylor expansion of the Hamiltonian at $p=0$ gives
\begin{align}\label{eq:ldpsse_quadratic_hamiltonian}
    H(p,q)
    ={}&\langle p,(-\Delta+V)q\rangle_{\R}
      +| b(p,q)|^2
      +\trace\bigl( Q(p,q)^2\bigr)
      +O(\|p\|^3),
\end{align}
where we used $\trace Q(p,q)=\langle p,Vq\rangle_{\R}$. There is no mixed quadratic term between $ b$ and $Q$: this is the counterpart of the independence of the two Gaussian noises appearing in \eqref{eq:simplified_model}.
The same conclusion is visible directly in the control representation. Put $C=I+D$. Near $(u,D)=(0,0)$,
\begin{equation*}
    \frac14\left(|u|^2+\trace(C+C^{-1}-2I)\right)
    =\frac14\left(|u|^2+\|D\|_{\mathrm{HS}}^2\right)
      +O(\|D\|_{\mathrm{HS}}^3).
\end{equation*}
Moreover, the term $u\otimes u$ in the constraint of
\eqref{eq:ldpsse_exact_action} is of second order in the controls and
therefore does not enter the constraint for the first variation of the
path. Retaining the quadratic cost and this linearized constraint, and
setting $u=\sqrt2\,v$ and $D=\sqrt2\,Z$, we obtain the quadratic action
\begin{align}\label{eq:ldpsse_quadratic_action}
    I^{(2)}_{\psi_0}(\psi)
    =\inf\Bigg\{&\frac12\int_0^T
       \left(|v_t|^2+\|Z_t\|_{\mathrm{HS}}^2\right)dt:\notag\\[-2mm]
       &i\partial_t\psi
       =(-\Delta+V)\psi
        +i\sqrt8\sum_k v_k\,\sigma_k\cdot\nabla\psi
       +\sqrt2\,\psi\sum_{h,k}Z_{hk}\,\sigma_h\cdot\sigma_k
    \Bigg\}.
\end{align}
The matrix control $Z$ may be taken symmetric. Allowing arbitrary matrices, as in the representation by independent Brownian motions $B^{h,k}$ below, leaves the infimum unchanged because only the symmetric part of $Z$ enters the constraint, and symmetrization cannot increase the Hilbert--Schmidt norm. Using the Karhunen--Lo\`eve representation of $B$ following \eqref{second_covariance}, \eqref{eq:simplified_model} can be written as
\begin{align*}
    i\,d\Psi^\tau
    &=(-\Delta+V)\Psi^\tau\,dt\\
    &\quad+i\sqrt{8\tau}\sum_k
    \sigma_k\cdot\nabla\Psi^\tau\circ dW_t^k+\sqrt{2\tau}\,\Psi^\tau
    \sum_{h,k}(\sigma_h\cdot\sigma_k)\circ dB_t^{h,k}.
\end{align*}
Its Freidlin--Wentzell skeleton is exactly the constraint in \eqref{eq:ldpsse_quadratic_action}, and its control cost is the one displayed there. Equivalently, its local Hamiltonian is
\begin{align*}
    H_{\mathrm{SSE}}(p,q)
    ={}&\langle p,(-\Delta+V)q\rangle_{\R}
       +\frac12\sum_k
          \langle p,i\sqrt8\,\sigma_k\cdot\nabla q\rangle_{\R}^{2}
      +\frac12\sum_{h,k}
          \langle p,\sqrt2\,(\sigma_h\cdot\sigma_k)q\rangle_{\R}^{2}\\
    ={}&\langle p,(-\Delta+V)q\rangle_{\R}
       +| b(p,q)|^2
       +\trace\bigl( Q(p,q)^2\bigr),
\end{align*}
which is precisely the second-order truncation \eqref{eq:ldpsse_quadratic_hamiltonian}. The It\^o--Stratonovich correction is of order $\tau$ and therefore does not change the rate functional at scale $\tau$ (speed $1/\tau$). Under the identifications
\begin{equation*}
    f=2\sqrt2\sum_k v_k\sigma_k,
    \qquad
    g=\sqrt2\sum_{h,k}Z_{hk}(\sigma_h\cdot\sigma_k),
\end{equation*}
the infimum over coordinate representations gives precisely the Cameron--Martin norms of $W$ and $B$. Thus \eqref{eq:ldpsse_quadratic_action} is the rate function \eqref{rate_function} proved above for the effective model, and  \eqref{eq:simplified_model} is naturally interpreted as the Gaussian, or quadratic-action, approximation of the original fast-field dynamics around its averaged minimizer. Linearizing it once more at $\bar\psi$ recovers the fluctuation equation \eqref{eq:gaussianlimit}.

\subsection{Stochastic Unraveling of Lindbladian Dynamics}
\label{ssec:lindbladian}

Let $\HH_S$ be the Hilbert space of a quantum system and $\HH_B$ that of an environment. If the composite state evolves unitarily on $\HH_S\otimes\HH_B$, the reduced state $\rho_S(t)=\trace_B\rho_{SB}(t)$ generally follows a non-unitary dynamics. In regimes in which the reservoir memory can be neglected---for instance, in suitable weak-coupling limits \cite{davies1974markovian}---the reduced evolution is modeled by a quantum dynamical semigroup of completely positive, trace-preserving maps. In finite dimension, and more generally for uniformly continuous semigroups with bounded coefficients, the generator takes the form
\begin{equation}\label{eq:gksl_general}
    \mathcal{L}(\rho)
    =-i[H,\rho]
    +\sum_{\alpha}\left(
        V_{\alpha}\rho V_{\alpha}^*
        -\frac12\{V_{\alpha}^*V_{\alpha},\rho\}
    \right),
\end{equation}
where $H$ is self-adjoint and $\{A,B\}=AB+BA$. We refer to  \cite{gorini1976completely,lindblad1976generators} for this fact and to \cite{breuer2002theory} for a general introduction. A stochastic \emph{unraveling} of \eqref{eq:gksl_general} is a pure-state stochastic evolution $\Phi_t$ whose ensemble density operator $\expt{\ket{\Phi_t}\bra{\Phi_t}}$ solves the corresponding master equation; diffusive unravelings are treated systematically in \cite{barchielli2009quantum}. Such a representation is not unique, and here the term is used only in this ensemble sense, without imposing a measurement interpretation.

The stochastic dynamics \eqref{eq:simplified_model} considered in this paper can in fact be identified with an unraveling of this type. Using the Karhunen--Lo\`eve representation of $B$ displayed after \eqref{second_covariance}, let $M_{k,j}$ denote multiplication by the real-valued function $\sigma_k\cdot\sigma_j$ and set
\begin{equation*}
    H_0=-\Delta+V,\qquad
    L_k=\sqrt{8\tau}\,\sigma_k\cdot\nabla,\qquad
    K_{k,j}=-i\sqrt{2\tau}\,M_{k,j}.
\end{equation*}
Under the real-valued and divergence-free condition on the vector fields $\sigma_k$, $H_0$ is self-adjoint, whereas $L_k$ and $K_{k,j}$ are skew-adjoint on their natural domains after closure. Since $H_0$ and the transport operators $L_k$ are unbounded, the following generator computation is understood formally on a common invariant core. Dividing \eqref{eq:simplified_model} by $i$ gives
\begin{equation*}
    d\Psi^\tau
    =-iH_0\Psi^\tau\,dt
    +\sum_{k\in I}L_k\Psi^\tau\circ dW_t^k
    +\sum_{k,j\in I}K_{k,j}\Psi^\tau\circ dB_t^{k,j}.
\end{equation*}
Equivalently, in It\^o form,
\begin{multline*}
    d\Psi^\tau
    =\left(
        -iH_0
        -\frac12\sum_{k\in I}L_k^*L_k
        -\frac12\sum_{k,j\in I}K_{k,j}^*K_{k,j}
    \right)\Psi^\tau\,dt\\
    +\sum_{k\in I}L_k\Psi^\tau\,dW_t^k
    +\sum_{k,j\in I}K_{k,j}\Psi^\tau\,dB_t^{k,j}.
\end{multline*}
For normalized initial data, define the ensemble density operator $\rho_t=\expt{\ket{\Psi_t^\tau}\bra{\Psi_t^\tau}}$. Applying It\^o's formula to the rank-one operator $\ket{\Psi_t^\tau}\bra{\Psi_t^\tau}$ and taking expectations,
\begin{multline*}
    \partial_t\rho_t
    =-i[H_0,\rho_t]
    +\sum_{k\in I}\left(
        L_k\rho_tL_k^*
        -\frac12\{L_k^*L_k,\rho_t\}
    \right)\\
    +\sum_{k,j\in I}\left(
        K_{k,j}\rho_tK_{k,j}^*
        -\frac12\{K_{k,j}^*K_{k,j},\rho_t\}
    \right),
\end{multline*}
that is precisely in the form \eqref{eq:gksl_general}. 
Thus \eqref{eq:simplified_model} is a diffusive, random unitary unraveling of a Lindbladian dynamics. This is a formal identification of the ensemble evolution, not a microscopic derivation of the coefficients from a particular system--bath Hamiltonian.

\section{Well-Posedness of the Dynamics}\label{app:well-posed}
We begin by recording some standard properties of the magnetic potential $A^{\tau}_t$.
\begin{lemma}\label[lemma]{lem:stochastic_conv}
  Assume \cref{hp:noise}. For each $\tau\in (0,1)$ and $T\geq 0$, the process
  $(A^{\tau}_t)_{t\in [0,T]}$ has paths in
  $C([0,T];H^{5+\frac{d+\theta}{2}})$ $\mathbb{P}-a.s.$
  Moreover, it is stationary and
\begin{align}\label{exponential_estimate}
      A^{\tau}_t\sim \mu \quad \text{for each }t\in [0,T],\ \tau\in (0,1),
  \end{align}
  where $\mu$ is the Gaussian measure with covariance kernel given by $Q(x,y).$
  In particular, there exists $\lambda>0$ such that, for every
  $t\in[0,T]$ and $\tau\in(0,1)$,
  \begin{align*}
\expt{e^{\lambda\|A^{\tau}_t\|_{H^{5+\frac{d+\theta}{2}}}^2}}<+\infty.
  \end{align*}
Analogous statements hold replacing $H^{5+\frac{d+\theta}{2}}(\R^d)$ with $W^{5,\infty}(\R^d)$.
\end{lemma}
\begin{proof}
The first two claims are a direct consequence of \cref{hp:noise}, cf. \cite[Chapter 5]{da2014stochastic}.
    The exponential estimate, as well as the independence of $\lambda$ from
    $\tau$ and $t$, is a consequence of Fernique's theorem
    \cite[Theorem 2.7]{da2014stochastic}.
    The last claim follows by Sobolev embedding.
\end{proof}
The well-posedness of \eqref{eq:rmse}, \eqref{eq:simplified_model} is proved introducing, for each $\tau>0$, the approximate problems
\begin{align}
    \begin{cases}\label{eq:rmse_viscous}
         i\partial_t \psi^{\tau,\nu}&=-i\nu \Delta^2 \psi^{\tau,\nu}+(-i\nabla- A^{\tau} )^2\psi^{\tau,\nu}\\
         \psi^{\tau,\nu}(0)&=\psi_0,
    \end{cases}
\end{align}
\begin{align}
\begin{cases}\label{eq:effectivese_viscous}
        i d\Psi^{\tau,\nu}
        &=(-i\nu \Delta^2-\Delta+V)\Psi^{\tau,\nu}dt\\
        &\quad+i\sqrt{8\tau}\sum_{k\in I}
        \sigma_k\cdot\nabla \Psi^{\tau,\nu}\circ dW^k
        +\sqrt{\tau} \Psi^{\tau,\nu}\circ dB_t,\\
         \Psi^{\tau,\nu}(0)&=\psi_0,
    \end{cases}
\end{align}
where $\nu>0$.
The hyperviscous operator $\nu\Delta^2$ makes the equations
above parabolic, so existence and uniqueness are classical; see, for example,
\cite[Chapters 5.1--5.2]{flandoli2024regularity}.
Namely the following holds.
\begin{lemma}\label[lemma]{lem:rmse_viscous}
Assume \cref{hp:noise} and $\psi_0\in L^2(\R^d)$. There exists a unique
progressively measurable process $\psi^{\tau,\nu}$ with paths in
$C([0,T];L^2)\cap L^2(0,T;H^2)$ $\mathbb{P}-a.s.$ satisfying, for each
$\phi\in \mathscr{S}(\R^d)$ and $t\in [0,T]$,
\begin{align*}
i\langle \psi^{\tau,\nu}_t,\phi\rangle-i\langle \psi_0,\phi\rangle
&=-i\nu \int_0^t
\langle\psi^{\tau,\nu}_s,\Delta^2\phi\rangle ds
-\int_0^t \langle\psi^{\tau,\nu}_s,\Delta\phi\rangle ds\\
&\quad-2i\int_0^t
\langle \psi^{\tau,\nu}_s,A^{\tau}_s\cdot\nabla\phi\rangle ds+\int_0^t
\langle|A^{\tau}_s|^2\psi^{\tau,\nu}_s,\phi\rangle ds
\quad \mathbb{P}-a.s.,
   \end{align*}
   where $A^{\tau}_t$ is the stationary Ornstein–Uhlenbeck process \eqref{eq:OU_prel}. 
\end{lemma}
\begin{lemma}\label[lemma]{lem:effective_viscous}
Assume \cref{hp:noise} and $\psi_0\in L^2(\R^d)$. There exists a unique process \footnote{If $H$ is a separable Hilbert space, here we denote by $C_{\mathcal{F}}\left(  \left[  0,T\right]  ;H\right) $ the space of continuous adapted processes $\left(  X_{t}\right)  _{t\in\left[
0,T\right]  }$ with values in $H$ such that
\begin{align*}
    \mathbb{E} \bigg[ \sup_{t\in\left[  0,T\right]  }\left\Vert X_{t}\right\Vert
_{H}^{2}\bigg]  <\infty
\end{align*}
and by $L_{\mathcal{F}}^{2}\left(  0,T;Z\right),$ the space of progressively
measurable processes $\left(  X_{t}\right)  _{t\in\left[  0,T\right]  }$ with
values in $H$ such that
\begin{align*}
     \mathbb{E} \bigg[ \int_{0}^{T}\left\Vert X_{t}\right\Vert _{H}^{2}dt \bigg]
<\infty.
\end{align*}} $\Psi^{\tau,\nu}$ in
$C_{\mathcal{F}}([0,T];L^2)\cap L^2_{\mathcal{F}}(0,T;H^2)$ satisfying,
for each $\phi\in \mathscr{S}(\R^d)$ and $t\in [0,T]$,
	   \begin{align*}
	       i\langle \Psi^{\tau,\nu}_t,\phi\rangle
	       &=i\langle \psi_0,\phi\rangle
	       +\int_0^t \langle\Psi^{\tau,\nu}_s,
	       (-i\nu\Delta^2-\Delta+V)\phi\rangle ds\\
	       &\quad+4i\tau \int_0^t\langle \Psi^{\tau,\nu}_s,
	       \operatorname{div}(Q\nabla\phi)\rangle ds
	       -i\tau \int_0^t\langle R\Psi^{\tau,\nu}_s,\phi\rangle ds\\
	       &\quad-i\sqrt{8\tau}\sum_{k\in I}\int_0^t
	       \langle \Psi^{\tau,\nu}_s,\sigma_k\cdot\nabla\phi\rangle dW^k_s\\
	       &\quad+\sqrt{2\tau}\sum_{k,j\in I}\int_0^t
	       \langle \Psi^{\tau,\nu}_s\sigma_k\cdot\sigma_j,\phi\rangle
	       dB^{k,j}_s
	       \quad \mathbb{P}-a.s.
   \end{align*}    
\end{lemma}
\subsection{A priori estimates for the viscous problems}
We first show the following uniform bounds for $\psi^{\tau,\nu},\ \Psi^{\tau,\nu}$.
\begin{lemma}\label[lemma]{lem:uniform_viscousrmse}
    For each $\nu,\tau\in(0,1)$,
    \begin{align}
         \sup_{t \in [0,T]}  \|\psi^{\tau,\nu}_t\|^2&\leq  \|\psi_0\|^2\quad \mathbb{P}-a.s.\label{first_claim_viscous}\\
         \sup_{t \in [0,T]}  \|\Psi^{\tau,\nu}_t\|^2&\leq  \|\psi_0\|^2\quad \mathbb{P}-a.s.\label{first_claim_viscous_sse}
    \end{align}
\end{lemma}
\begin{proof}
	    By \cref{lem:rmse_viscous},
	    \begin{align*}
	     \psi^{\tau,\nu}&\in L^2(0,T;H^2(\R^d)),\\
	     -i\nu \Delta^2 \psi^{\tau,\nu}
	     +(-i\nabla- A^{\tau} )^2\psi^{\tau,\nu}
	     &\in L^2(0,T;H^{-2}(\R^d))
    \end{align*}
    $\mathbb{P}-a.s.$, we can apply the Lions--Magenes lemma to obtain
    \begin{align*}
\|\psi^{\tau,\nu}_t\|^2+2\nu\int_0^t \|\Delta \psi^{\tau,\nu}_s\|^2 ds &=\|\psi_0\|^2\quad\mathbb{P}-a.s.
    \end{align*}
    for each $t\in [0,T]$, and the first claim trivially follows.\\
    By \cref{hp:noise} and \cref{lem:effective_viscous}, it holds
\begin{align*}
    &(-i\nu\Delta^2-\Delta+V)\Psi^{\tau,\nu}
    +4i\tau\operatorname{div}(Q\nabla\Psi^{\tau,\nu})
    -i\tau R\Psi^{\tau,\nu}\in L^2(0,T;H^{-2})
    \quad\mathbb{P}-a.s 
\end{align*}
and 
\begin{align*}
    &i\sqrt{8\tau}\sum_{k\in I}\int_0^\cdot
    \sigma_k\cdot\nabla\Psi^{\tau,\nu}_s\,dW^k_s+\sqrt{2\tau}\sum_{k,j\in I}\int_0^\cdot
    (\sigma_k\cdot\sigma_j)\Psi^{\tau,\nu}_s\,dB^{k,j}_s
\end{align*}
	 is an $H^1(\R^d)$-valued local martingale. Therefore we can apply It\^o's formula in $L^2(\R^d)$, cf. \cite[Theorem 2.13]{rozovskii2012stochastic}, obtaining
 \begin{align*}
     \|\Psi^{\tau,\nu}_t\|^2+2\nu\int_0^t \|\Delta \Psi^{\tau,\nu}_s\|^2 ds &=\|\psi_0\|^2\quad\mathbb{P}-a.s.
 \end{align*}
 for each $t\in [0,T]$, which implies the second claim.
\end{proof}
Concerning time increments, the following holds true.
\begin{lemma}\label[lemma]{lem_time_increments_viscous}
	    For each $\nu,\tau\in(0,1)$, $\vartheta,\gamma\geq4$, and
	    $0\leq s\leq t\leq T$,
	    \begin{align*}
	        \|\psi^{\tau,\nu}_t-\psi^{\tau,\nu}_s\|_{H^{-\vartheta}}
	        &\lesssim (t-s)^{1/2}\|\psi_0\|
	        \left(\int_0^T(1+\|A^{\tau}_r\|_{L^{\infty}}^4)dr
	        \right)^{1/2}
	        \quad \mathbb{P}-a.s.\\
	        \expt{\|\Psi^{\tau,\nu}_t-\Psi^{\tau,\nu}_s\|_{H^{-\vartheta}}^\gamma}
	        &\lesssim_{\gamma,T,\psi_0} (t-s)^{\gamma/2}.
    \end{align*}
\end{lemma}
\begin{proof}
    From the weak formulation satisfied by $\psi^{\tau,\nu}$ we easily get the first claim from \eqref{first_claim_viscous}. Indeed, since $\vartheta\geq 4,$
    \begin{align*}
	        \|\psi^{\tau,\nu}_t-\psi^{\tau,\nu}_s\|_{H^{-\vartheta}}
	        &\leq (1+\nu) \int_s^t \|\psi^{\tau,\nu}_r\|dr\\
	        &\quad+2\int_s^t
	        \|\operatorname{div}(A^{\tau}_r\psi^{\tau,\nu}_r)\|_{H^{-\vartheta}}dr+\int_s^t
	        \||A^{\tau}_r|^2\psi^{\tau,\nu}_r\|_{H^{-\vartheta}}dr\\
	        &\lesssim \int_s^t
	        (1+\|A^{\tau}_r\|_{L^{\infty}}
	        +\|A^{\tau}_r\|_{L^{\infty}}^2)
	        \|\psi^{\tau,\nu}_r\|dr\\
	        &\lesssim (t-s)^{1/2}\|\psi_0\|
	        \left(\int_0^T(1+\|A^{\tau}_r\|_{L^{\infty}}^4)dr
	        \right)^{1/2}
	        \quad \mathbb{P}-a.s.
    \end{align*}
    For the second claim we argue similarly, just applying Burkholder-Davis-Gundy inequality. Therefore
    \begin{align*}
	        \expt{\|\Psi^{\tau,\nu}_t-\Psi^{\tau,\nu}_s
	        \|_{H^{-\vartheta}}^\gamma}
	        &\lesssim_{\gamma}
	        (1+\nu+\|V\|_{L^{\infty}})^\gamma
	        \expt{\left(\int_s^t \|\Psi^{\tau,\nu}_r\| dr\right)^{\gamma}}\\
	        &\qquad+\expt{\left(\sum_{k\in I}\int_s^t
	        \|\operatorname{div}(\sigma_k\Psi^{\tau,\nu}_r)
	        \|_{H^{-\vartheta}}^2dr\right)^{\gamma/2}}\\
	        &\qquad+\expt{\left(\sum_{k,j\in I}\int_s^t
	        \|(\sigma_k\cdot\sigma_j)\Psi^{\tau,\nu}_r
	        \|_{H^{-\vartheta}}^2dr\right)^{\gamma/2}}\\
	        &\quad\lesssim_{\gamma,T,\psi_0}(t-s)^{\gamma/2}.
    \end{align*}
\end{proof}

\subsection{Compactness and Passage to the Limit}
Let $\vartheta\geq 4$ be as in \cref{lem_time_increments_viscous} and $R=\|\psi_0\|$, define the Polish space
\begin{align*}
    Z_T:=
    C([0,T];B^{L^2}_{R,w})
    \cap C([0,T];\tilde{H}^-).
\end{align*}
Let us also introduce the spaces
\begin{align*}
    \mathcal{X}_1
    &=Z_T\times H^{5+\frac{d+\theta}{2}}
    \times C([0,T];H^{5+\frac{d+\theta}{2}})
    \times C([0,T];\R^I),\\
    \mathcal{X}_2
    &=Z_T\times C([0,T];\R^I)
    \times C([0,T];\R^{I\times I}).
\end{align*}
Next we show the tightness of the viscous approximations constructed so far.
\begin{lemma}\label[lemma]{lem_compactness_viscous}
    For each $\tau>0$, the family of laws
    $\{\mathscr{L}(\psi^{\tau,\nu},A^{\tau}_0,A^{\tau},
    (W^k)_{k\in I})\}_{\nu\in(0,1)}$ is tight in $\mathcal{X}_1$.
    The family
    $\{\mathscr{L}(\Psi^{\tau,\nu},(W^k)_{k\in I},
    (B^{k,j})_{k,j\in I})\}_{\nu\in(0,1)}$ is tight in $\mathcal{X}_2$.
\end{lemma}
\begin{proof}
    Since the joint laws of
    $(A^{\tau}_0,A^{\tau},(W^k)_{k\in I})$ and
    $((W^k)_{k\in I},(B^{k,j})_{k,j\in I})$ do not depend on $\nu$, it is
    enough to check tightness of $\mathscr{L}(\psi^{\tau,\nu})$ (resp.
    $\mathscr{L}(\Psi^{\tau,\nu})$) on $Z_T$.
    Let $\mathbb{B}_{\bar{R}}$ denote the closure in $Z_T$ of the intersection between $Z_T$ and the ball of radius ${\bar{R}}$ centered at $0$ in $C^{1/3}([0,T];\tilde{H}^{-\vartheta})$.
    By \cite[Lemma 2.2]{crippa2025zero},  $\mathbb B_{\bar{R}}$ is compact in
    $Z_T$.  The first estimate of
    \cref{lem_time_increments_viscous} and
    \cref{lem:stochastic_conv} control the corresponding random
    H\"older seminorm for $\psi^{\tau,\nu}$.  For
    $\Psi^{\tau,\nu}$, the second estimate of that lemma and the
    Kolmogorov--Chentsov theorem (equivalently, the
    Garsia--Rodemich--Rumsey inequality) give the same control, uniformly
    in $\nu$.  Thus, for some $p\geq1$,
    \begin{align*}
      \sup_{\nu\in(0,1)}
      \mathbb P(\psi^{\tau,\nu}\notin\mathbb B_{\bar{R}})
      +\sup_{\nu\in(0,1)}
      \mathbb P(\Psi^{\tau,\nu}\notin\mathbb B_{\bar{R}})
      \lesssim_{\tau} \bar{R}^{-p}.
    \end{align*} 
    
    Letting $\bar{R}\to\infty$ proves tightness in $Z_T$.
\end{proof}
\subsubsection{Proof of \cref{prop:well_confinement}}\label{subsec_well_posed_fast_slow}
By the Yamada--Watanabe theorem, cf. \cite{kurtzYW}, it is enough to show
weak existence and pathwise uniqueness for solutions of \eqref{eq:rmse} in
$C([0,T];B^{L^2}_{R,w})$.\\
\emph{Weak Existence: }
Thanks to \cref{lem_compactness_viscous}, the
Skorokhod representation theorem, and standard arguments (see, for example,
\cite[Chapter 2]{flandoli2023stochastic} or
\cite[Section 5.3]{brzezniak2013existence}), we find a sequence
$\nu_n\rightarrow0$,
an auxiliary probability space, which for simplicity we continue to call
$(\Omega,\mathcal{F},\mathbb{P})$\footnote{From now on, in case of ambiguity,
we use $(\tilde{\Omega},\tilde{\mathcal{F}},\tilde{\mathbb{P}})$ for the
original probability space.}, and processes
\begin{align*}
(\psi^{\tau,n},A^{\tau,n}_0,A^{\tau,n},
W^n=(W^{k,n})_{k\in I}),\quad (\psi^{\tau},A^{\tau}_0,A^{\tau},W=(W^{k})_{k\in I})
\end{align*}
such that
\begin{align*}
 \psi^{\tau,n}&\rightarrow \psi^{\tau}\quad\text{ in }Z_T\quad \mathbb{P}-a.s.,\\
 A^{\tau,n}_0&\rightarrow A^{\tau}_0 \text{ in }H^{5+\frac{d+\theta}{2}}\quad \mathbb{P}-a.s.\\
 A^{\tau,n}&\rightarrow A^{\tau} \text{ in }C([0,T];H^{5+\frac{d+\theta}{2}})\quad \mathbb{P}-a.s.\\
 W^n&\rightarrow W \text{ in }C([0,T];\R^I)\quad \mathbb{P}-a.s.
\end{align*}
Of course the convergence above between $W^n$ and $W$ can be seen as the uniform convergence of
cylindrical Wiener processes
on a suitable Hilbert space $U_0$. Moreover, $\psi^{\tau,n}$ is the unique solution of \eqref{eq:rmse_viscous} with viscosity $\nu_n$. By linearity of the equations and the above convergence, it is easy to check that $(\psi^\tau,A^{\tau})$ satisfies for each $\phi\in \mathscr{S}(\R^d),\ t\in [0,T]$ 
	   \begin{align*}
	       i\langle \psi^{\tau}_t,\phi\rangle-i\langle \psi_0,\phi\rangle
	       &=-\int_0^t \langle\psi^{\tau}_s,\Delta\phi\rangle ds\\
	       &\quad-2i\int_0^t
	       \langle \psi^{\tau}_s,A^{\tau}_s\cdot\nabla\phi\rangle ds
	       +\int_0^t \langle|A^{\tau}_s|^2\psi^\tau_s,\phi\rangle ds
	       \quad \mathbb{P}-a.s.,\\
	       A^{\tau}_t
	       &=e^{-\frac{t}{\tau}}A^{\tau}_0
	       +\sqrt{\frac{2}{{\tau}}}\sum_{k\in I}\int_{0}^t
	       e^{-\frac{(t-s)}{\tau}}\sigma_k\,dW^k_s
	       \quad \mathbb{P}-a.s.
   \end{align*}
\noindent\emph{Pathwise uniqueness:} By linearity of \eqref{eq:rmse} and standard uniqueness results for \eqref{eq:OU} given by \cite[Chapter 5]{da2014stochastic}, it is enough to show pathwise uniqueness for $\psi_0=0$ and $A^{\tau}_t$ given. In particular, we will show the validity of the conservation of the $L^2$ norm. The latter readily implies the claim. In order to save notation, we drop the dependence on $\tau$ on $A^\tau,\ \psi^\tau$. Let $\chi$ be a standard smooth mollifier and let, for each $\epsilon>0$, $\chi_{\epsilon}(x)=\frac{\chi(x/\epsilon)}{\epsilon^d}$. Then choosing $\chi_\epsilon(x-\cdot)$ as a test function in \eqref{eq:weak:original} we obtain
\begin{align*}
     \begin{cases}
            i\partial_t \psi^{\epsilon}&=-\Delta \psi^{\epsilon}+ 2iA\cdot\nabla \psi^{\epsilon}+2i[A\cdot\nabla,\chi_{\epsilon}\ast]\psi+\left(|A|^2\psi\right)\ast\chi_{\epsilon}\\
            \psi^{\epsilon}_0&=\psi_0\ast \chi_{\epsilon}
        \end{cases}
\end{align*}
    where $\psi^{\epsilon}=\psi\ast \chi_{\epsilon}$ is smooth, the notation $[A\cdot\nabla,\chi_{\epsilon}\ast]$ stands for the commutator and the equation is satisfied in classical sense. 
    Therefore
    for every $t\in[0,T]$, $\mathbb{P}-a.s.$,
\begin{align}\label{eq:mollified_estimate_uniqueness}
        \|\psi^{\epsilon}_t\|^2-\|\psi^{\epsilon}_0\|^2&= 2\int_0^t \langle [A_s\cdot\nabla,\chi_{\epsilon}\ast]\psi_s,\psi^{\epsilon}_s\rangle +\langle\psi^{\epsilon}_s,[A_s\cdot\nabla,\chi_{\epsilon}\ast]\psi_s\rangle ds\notag\\ & -i\int_0^t 
        \langle \left(|A_s|^2\psi_s\right)\ast\chi_{\epsilon},\psi^{\epsilon}_s\rangle-\langle \psi^{\epsilon}_s,\left(|A_s|^2\psi_s\right)\ast\chi_{\epsilon}\rangle
        ds.
    \end{align}
    From the regularity of $\psi$ and $A$ and basic properties of convolutions, it follows that $\mathbb{P}-a.s.$
    \begin{align*}
\lim_{\epsilon\rightarrow 0}\|\psi^\epsilon_t-\psi_t\|=0 \quad &\mbox{ for each } t\in [0,T],\\
 \sup_{\epsilon>0} \|\psi^\epsilon_t\|\leq \|\psi_t\|
\leq\sup_{r\in[0,T]}\|\psi_r\|\quad &\mbox{ for each } t\in [0,T],\\
\lim_{\epsilon\rightarrow 0}\left\|\left(|A_t|^2\psi_t\right)\ast\chi_{\epsilon}-|A_t|^2\psi_t\right\|=0\quad &\mbox{ for each } t\in [0,T],\\
 \sup_{\epsilon>0} \|\left(|A_t|^2\psi_t\right)\ast\chi_{\epsilon}\|
&\leq\sup_{r\in[0,T]}\|\psi_r\|\,
\|A_t\|_{L^\infty}^2.
    \end{align*}
    Therefore, by dominated convergence, it holds
    \begin{align*}
        \sup_{t\in [0,T]}\left\lvert\int_0^t 
        \langle \left(|A_s|^2\psi_s\right)\ast\chi_{\epsilon},\psi^{\epsilon}_s\rangle-\langle \psi^{\epsilon}_s,\left(|A_s|^2\psi_s\right)\ast\chi_{\epsilon}\rangle
        ds\right\rvert \rightarrow 0\quad \mathbb{P}-a.s.
    \end{align*}
    Secondly, by \cite[Lemma 4.2]{butori2026background}, for almost every
    $s\in [0,T]$,
   \begin{align*}
       \lim_{\epsilon\rightarrow 0}\|[A_s\cdot\nabla,\chi_{\epsilon}\ast]\psi_s\|&=0,\\
       \sup_{\epsilon\in (0,1)}\|[A_s\cdot\nabla,\chi_{\epsilon}\ast]\psi_s\|&\lesssim \|A_s\|_{H^{1+\frac{d+\theta}{2}}}\|\psi_s\|.
   \end{align*}
   Therefore, again by the dominated convergence theorem and standard
   properties of convolutions, letting $\epsilon\rightarrow 0$ in
   \eqref{eq:mollified_estimate_uniqueness}, there exists a null set
   $N\subset \Omega$ such that, on $N^c$, for every $t\in [0,T]$,
   \begin{align*}
       \|\psi_t\|^2&=\|\psi_0\|^2
   \end{align*}
   completing the proof.
\subsubsection{Proof of \cref{prop:well_effective}}\label{subsec:well_spde}
By the Yamada--Watanabe theorem, cf. \cite{kurtzYW}, it is enough to show
weak existence and pathwise uniqueness for solutions of
\eqref{eq:simplified_model}. The proofs below closely resemble those of the
previous subsection.\\
\emph{Weak Existence: }
Thanks to \cref{lem_compactness_viscous}, the Skorokhod representation theorem, and standard arguments, we find a sequence
$\nu_n\rightarrow0$,
an auxiliary probability space, which for simplicity we continue to call
$(\Omega,\mathcal{F},\mathbb{P})$\footnote{From now on, in case of ambiguity,
we use $(\tilde{\Omega},\tilde{\mathcal{F}},\tilde{\mathbb{P}})$ for the
original probability space.}, and processes
\begin{align*}
&(\Psi^{\tau,n},W^n=(W^{k,n})_{k\in I},
B^n=(B^{k,j,n})_{k,j\in I}),\\
&(\Psi^{\tau},W=(W^{k})_{k\in I},
B=(B^{k,j})_{k,j\in I})
\end{align*}
such that
\begin{align*}
 \Psi^{\tau,n}&\rightarrow \Psi^{\tau}\quad\text{ in }Z_T\quad \mathbb{P}-a.s.,\\
 W^n&\rightarrow W \text{ in }C([0,T];\R^I)\quad \mathbb{P}-a.s.\\
  B^n&\rightarrow B \text{ in }C([0,T];\R^{I\times I})\quad \mathbb{P}-a.s.
\end{align*}
 Moreover, $\Psi^{\tau,n}$ is the unique solution of \eqref{eq:effectivese_viscous} with viscosity $\nu_n$. By linearity of the equations and the above convergence, it is easy to check that $\Psi^\tau$ satisfies
 \begin{align*}
     \sup_{t\in [0,T]}\|\Psi^\tau_t\|\leq \|\psi_0\|\quad \mathbb{P}-a.s.
 \end{align*}
   The passage to the limit in the stochastic integrals can be justified,
   for example, by \cite[Lemma 4.3]{bagnara2025no}. Thus the following
   identity holds for each $\phi\in\mathscr{S}(\R^d)$ and $t\in[0,T]$:
	   \begin{align*}
	      i\langle \Psi^{\tau}_t,\phi\rangle
	      &=i\langle \psi_0,\phi\rangle
	      +\int_0^t \langle\Psi^{\tau}_s,(-\Delta+V)\phi\rangle ds\\
	      &\quad+4i\tau\int_0^t
	      \langle \Psi^{\tau}_s,\operatorname{div}(Q\nabla\phi)\rangle ds
	      -i\tau\int_0^t\langle R\Psi^{\tau}_s,\phi\rangle ds\\
	      &\quad-i\sqrt{8\tau}\sum_{k\in I}\int_0^t
	      \langle \Psi^{\tau}_s,\sigma_k\cdot\nabla\phi\rangle dW^k_s\\
	      &\quad+\sqrt{2\tau}\sum_{k,j\in I}\int_0^t
	      \langle \Psi^{\tau}_s\sigma_k\cdot\sigma_j,\phi\rangle dB^{k,j}_s
	      \quad \mathbb{P}-a.s.
   \end{align*}
\emph{Pathwise uniqueness:} By linearity of \eqref{eq:simplified_model} it is enough to show pathwise uniqueness for $\psi_0=0$. We will show the validity of the conservation of the $L^2$ norm. The latter readily implies the claim. In order to save notation, we drop the dependence on $\tau$ in $\Psi^\tau$. Let $\chi$ be a standard, radially symmetric, smooth mollifier with compact support in the ball of radius $1$ and let, for each $\epsilon>0$, $\chi_{\epsilon}(x)=\frac{\chi(x/\epsilon)}{\epsilon^d}$ and let
 $\Psi^{\epsilon}_t=\Psi_t\ast \chi_{\epsilon}$. From the regularity of $\Psi$ we know that on a full probability set, for all $t\in [0,T]$
 \begin{align}\label{eq:convolution_properties}
     \lim_{\epsilon\rightarrow0}\|\Psi^{\epsilon}_t-\Psi_t\|=0,\quad \|\Psi^{\epsilon}_t\|\leq \bar{R}.
 \end{align}
 We can apply It\^o's formula, cf. \cite[Theorem 4.32]{da2014stochastic}, to compute the evolution of $\langle\Psi_t,\Psi^{\epsilon}_t\rangle $:
\begin{align}\label{ito_formula_diss_measure}
\langle\Psi_t,\Psi^{\epsilon}_t\rangle
&-\langle\psi_0,\Psi^{\epsilon}_0\rangle\\
&=i\int_0^t\left(
\langle \Psi_s,(\Delta-V)\Psi^{\epsilon}_s\rangle
-\langle(\Delta-V)\Psi^{\epsilon}_s,\Psi_s\rangle\right)ds\notag\\
&\quad+4\tau\int_0^t\Bigg(
\langle \Psi_s,\operatorname{div}(Q\nabla\Psi^{\epsilon}_s)\rangle
+\langle \operatorname{div}(Q\nabla\Psi^{\epsilon}_s),\Psi_s\rangle\notag\\
&\hspace{40mm}
-2\sum_{k\in I}\langle \sigma_k\Psi_s,
(\sigma_k\Psi_s)\ast\nabla^2\chi_{\epsilon}\rangle
\Bigg)ds\notag\\
&\quad-\tau\int_0^t\Bigg(
\langle R\Psi_s,\Psi^{\epsilon}_s\rangle
+\langle \Psi^{\epsilon}_s,R\Psi_s\rangle\notag\\
&\hspace{31mm}
-2\sum_{k,j\in I}\langle(\sigma_k\cdot\sigma_j)\Psi_s,
((\sigma_k\cdot\sigma_j)\Psi_s)\ast\chi_{\epsilon}\rangle
\Bigg)ds\notag\\
&\quad-\sqrt{8\tau}\sum_{k\in I}\int_0^t\left(
\langle \Psi_s,\sigma_k\cdot\nabla\Psi^{\epsilon}_s\rangle
+\langle\sigma_k\cdot\nabla\Psi^{\epsilon}_s,\Psi_s\rangle
\right)dW^k_s\notag\\
&\quad-i\sqrt{2\tau}\sum_{k,j\in I}\int_0^t\left(
\langle(\sigma_k\cdot\sigma_j)\Psi_s,\Psi^{\epsilon}_s\rangle
-\langle\Psi^{\epsilon}_s,(\sigma_k\cdot\sigma_j)\Psi_s\rangle
\right)dB^{k,j}_s.
\end{align}
From \eqref{eq:convolution_properties} it trivially follows 
\begin{align*}
    \lim_{\epsilon\rightarrow 0}\left(
    \langle\Psi_t,\Psi^{\epsilon}_t\rangle
    -\langle\psi_0,\Psi^{\epsilon}_0\rangle\right)
    =\|\Psi_t\|^2-\|\psi_0\|^2\quad \mathbb{P}-a.s.
\end{align*}
We are therefore left to study the convergence to $0$ of all the terms
appearing on the right-hand side of \eqref{ito_formula_diss_measure}. The
terms of order $0$ are fairly simple. Indeed, from
\eqref{eq:convolution_properties}, \cref{hp:noise}, and dominated
convergence, on a full-probability set,

\begin{align}\label{convergence_det_order_0_1}
    &\int_0^T \left|\langle \Psi_s,V\Psi^{\epsilon}_s\rangle
    -\langle V\Psi^{\epsilon}_s,\Psi_s\rangle\right|ds\rightarrow 0,\\
    &\int_0^T \Bigg|\langle R\Psi_s,\Psi^{\epsilon}_s\rangle
    +\langle \Psi^{\epsilon}_s,R\Psi_s\rangle\notag\\
    &\hspace{27mm}
    -2\sum_{k,j\in I}\langle(\sigma_k\cdot\sigma_j)\Psi_s,
    ((\sigma_k\cdot\sigma_j)\Psi_s)\ast\chi_{\epsilon}\rangle
    \Bigg|ds\rightarrow0.
    \label{convergence_det_order_0_2}
\end{align}
as $\epsilon\rightarrow 0$.
Set
\begin{align*}
    D^{\epsilon}_{k,j}(s)
    :=\langle(\sigma_k\cdot\sigma_j)\Psi_s,\Psi^{\epsilon}_s\rangle
    -\langle \Psi^{\epsilon}_s,(\sigma_k\cdot\sigma_j)\Psi_s\rangle.
\end{align*}
By the Burkholder--Davis--Gundy inequality and similar reasoning,
\begin{align*}
    \mathbb{E}\left[\sup_{t\in[0,T]}
    \left|\sum_{k,j\in I}\int_0^t
    D^{\epsilon}_{k,j}(s)\,dB^{k,j}_s\right|^2\right]
    &\lesssim
    \mathbb{E}\left[\sum_{k,j\in I}\int_0^T
    |D^{\epsilon}_{k,j}(s)|^2ds\right]\rightarrow0.
\end{align*}
Therefore, after passing to a non-relabeled subsequence,
\begin{align}\label{convergence_stoch_order_0}
    \sup_{t\in[0,T]}
    \left|\sum_{k,j\in I}\int_0^t
    D^{\epsilon}_{k,j}(s)\,dB^{k,j}_s\right|
    \rightarrow0\quad\mathbb{P}-a.s.
\end{align}
as $\epsilon\rightarrow 0$.
Also, by standard properties of convolution it holds
\begin{align*}
    \langle \Psi_s,\Delta\Psi^{\epsilon}_s\rangle
    =\langle \Psi_s,\Delta\chi_{\epsilon}\ast\Psi_s\rangle
    =\langle \Delta\chi_{\epsilon}\ast\Psi_s,\Psi_s\rangle
    =\langle\Delta\Psi^{\epsilon}_s,\Psi_s\rangle.
\end{align*}
Hence
\begin{align}\label{convergence_det_order_2_1}
    \int_0^T |\langle \Psi_s, \Delta\Psi^{\epsilon}_s\rangle- \langle  \Delta\Psi^{\epsilon}_s,\Psi_s\rangle| ds = 0\quad\mathbb{P}-a.s.
\end{align}
Also, set
\begin{align*}
    C_k^\epsilon(s)
    :=\langle \Psi_s,\sigma_k\cdot\nabla\Psi^{\epsilon}_s\rangle
    +\langle\sigma_k\cdot\nabla\Psi^{\epsilon}_s,\Psi_s\rangle
    =-\langle\Psi_s,
    [\sigma_k\cdot\nabla,\chi_{\epsilon}\ast]\Psi_s\rangle,
\end{align*}
where the last identity uses our commutator convention. By the
Burkholder--Davis--Gundy inequality,
\cite[Lemma 4.2]{butori2026background}, \cref{hp:noise}, and dominated
convergence, we get
\begin{align*}
    \mathbb{E}\left[\sup_{t\in[0,T]}
    \left|\sum_{k\in I}\int_0^t C_k^\epsilon(s)\,dW^k_s\right|^2\right]
    &\lesssim \mathbb{E}\left[\sum_{k\in I}\int_0^T
    |C_k^\epsilon(s)|^2 ds\right] \\
    &\rightarrow 0\quad \mbox{as }\epsilon\rightarrow 0.
\end{align*}
Thus, after passing to a further non-relabeled subsequence,
\begin{align}\label{convergence_stoch_order_1}
    \sup_{t\in[0,T]}
    \left|\sum_{k\in I}\int_0^t C_k^\epsilon(s)\,dW^k_s\right|
    \rightarrow0\quad\mathbb{P}-a.s.
    \quad\mbox{as }\epsilon\rightarrow0.
\end{align}
The last term is the more involved one. Set
\begin{align*}
I^{\epsilon}_t &:=
\langle \Psi_t,\operatorname{div}(Q\nabla\Psi^{\epsilon}_t) \rangle
+\langle \operatorname{div}(Q\nabla\Psi^{\epsilon}_t),\Psi_t \rangle-2\sum_{k\in I}\langle \sigma_k\Psi_t,
\left(\sigma_k\Psi_t\right)\ast \nabla^2\chi_{\epsilon}\rangle.
\end{align*}
Expanding $\operatorname{div}(Q\nabla\Psi^{\epsilon}_t)$ and using the
symmetry of the integration variables and of $\chi$, we obtain
\begin{align*}
    I_t^\epsilon
    &=\sum_{k\in I }\int_{\R^d\times \R^d }
    \left(\sigma_k(x)-\sigma_k(y)\right)\otimes
    \left(\sigma_k(x)-\sigma_k(y)\right)
    :\nabla^2\chi_{\epsilon}(x-y)\\
    &\hspace{22mm}\times
    \Re\left(\Psi_t(x)\overline{\Psi_t(y)}\right)\,dx\,dy\\
    &\qquad+\sum_{k\in I }\int_{\R^d\times \R^d }
    \left(\sigma_k(x)\cdot\nabla\sigma_k(x)
    -\sigma_k(y)\cdot\nabla\sigma_k(y)\right)
    \cdot \nabla\chi_{\epsilon}(x-y)\\
    &\hspace{22mm}\times
    \Re\left(\Psi_t(x)\overline{\Psi_t(y)}\right)\,dx\,dy\\
    &\quad=\sum_{k\in I}\left(J^{1,\epsilon,k}_t+J^{2,\epsilon,k}_t\right).
\end{align*}
Changing variables in the integrals we have, for each $k\in I$, 
\begin{align*}
    J^{1,\epsilon,k}_t
    &=\frac{1}{\epsilon^{2}}\int_{\R^d}\int_{|z|\leq 1}
    \left(\sigma_k(x)-\sigma_k(x-\epsilon z)\right)^{\otimes2}
    :\nabla^2\chi(z)\\
    &\hspace{35mm}\times
    \Re\left(\Psi_t(x)\overline{\Psi_t(x-\epsilon z)}\right) dz\,dx,\\
    J^{2,\epsilon,k}_t
    &=\frac{1}{\epsilon}\int_{\R^d}\int_{|z|\leq 1}
    \left(\sigma_k(x)\cdot\nabla\sigma_k(x)
    -\sigma_k(x-\epsilon z)\cdot\nabla\sigma_k(x-\epsilon z)\right)\\
    &\hspace{25mm}\cdot\nabla\chi(z)
    \Re\left(\Psi_t(x)\overline{\Psi_t(x-\epsilon z)}\right)dz\,dx.
\end{align*}
The behavior of the integrands above as $\epsilon\rightarrow 0$ can be
understood by Taylor's formula:
\begin{align*}
    &\left\lvert
    \left(\sigma_k(x)-\sigma_k(x-\epsilon z)\right)^{\otimes2}
    -\epsilon^2\left(z\cdot\nabla \sigma_k(x)\right)^{\otimes2}
    \right\rvert
    \lesssim \epsilon^3 \|\sigma_k\|_{W^{2,\infty}}^2,\\
    &\left\lvert
    \sigma_k(x)\cdot\nabla\sigma_k(x)
    -\sigma_k(x-\epsilon z)\cdot\nabla\sigma_k(x-\epsilon z)
    -\epsilon z\cdot\nabla\left(\sigma_k(x)\cdot\nabla\sigma_k(x)\right)
    \right\rvert\\
    &\hspace{55mm}\lesssim
    \epsilon^2 \|\sigma_k\|_{W^{3,\infty}}^2.
\end{align*}
Therefore, for each $k$,
\begin{align*}
    J^{1,\epsilon,k}_t
    &=\int_{\R^d}\int_{|z|\leq 1}
    \left(z\cdot\nabla \sigma_k(x)\right)^{\otimes2}
    :\nabla^2\chi(z)|\Psi_t(x)|^2\,dz\,dx\\
    &\quad+\int_{\R^d}\int_{|z|\leq 1}
    \left(z\cdot\nabla \sigma_k(x)\right)^{\otimes2}
    :\nabla^2\chi(z)
    \Re\left(\Psi_t(x)
    \overline{\Psi_t(x-\epsilon z)-\Psi_t(x)}\right)dz\,dx
    +R^{1,\epsilon,k}_t,\\
    J^{2,\epsilon,k}_t
    &=\int_{\R^d}\int_{|z|\leq 1}
    z\cdot\nabla\left(\sigma_k(x)\cdot\nabla\sigma_k(x)\right)
    \cdot\nabla\chi(z)|\Psi_t(x)|^2\,dz\,dx\\
    &\quad+\int_{\R^d}\int_{|z|\leq 1}
    z\cdot\nabla\left(\sigma_k(x)\cdot\nabla\sigma_k(x)\right)
    \cdot\nabla\chi(z)
    \Re\left(\Psi_t(x)
    \overline{\Psi_t(x-\epsilon z)-\Psi_t(x)}\right)dz\,dx
    +R^{2,\epsilon,k}_t,
\end{align*}
where
\begin{align*}
    |R^{1,\epsilon,k}_t|+|R^{2,\epsilon,k}_t|
    &\lesssim \epsilon \|\sigma_k\|_{W^{3,\infty}}^2
    \|\Psi_t\|^2 \quad \mathbb{P}-a.s.
\end{align*}
Integrating by parts, for each $i,j,k,l\in \{1,\dots, d\}$, we also have 
\begin{align*}
    \int_{|z|\leq 1} z_k  z_l\partial_{i,j}\chi(z) dz =\delta_{i,k}\delta_{j,l}+\delta_{i,l}\delta_{j,k},\quad
    \int_{|z|\leq 1} z_k\partial_j \chi(z)dz =-\delta_{j,k}.
\end{align*}
Therefore,
\begin{align*}
     &\int_{\R^d}\int_{|z|\leq 1}
     \left(z\cdot\nabla \sigma_k(x)\right)\otimes
     \left(z\cdot\nabla \sigma_k(x) \right):\nabla^2\chi(z)
     |\Psi_t(x)|^2\,dz\,dx\\
     &\quad+\int_{\R^d}\int_{|z|\leq 1}
     z\cdot\nabla\left(\sigma_k(x)\cdot\nabla\sigma_k(x)\right)
     \cdot\nabla\chi(z)|\Psi_t(x)|^2\,dz\,dx\\
     &=\int_{\R^d}\left(
     \left(\operatorname{div}\sigma_k(x)\right)^2
     +\trace(\nabla\sigma_k(x)\nabla\sigma_k(x))
     -\operatorname{div}\left(\sigma_k(x)\cdot\nabla\sigma_k(x)\right)
     \right)|\Psi_t(x)|^2\,dx=0
\end{align*}
since the $\sigma_k$ are divergence-free.
In conclusion, we have shown
\begin{align*}
|I^\epsilon_t|
&\lesssim
     \sum_{k\in I}\left\lvert \int_{\R^d}\int_{|z|\leq 1}
     \left(z\cdot\nabla \sigma_k(x)\right)^{\otimes2}:\nabla^2\chi(z)\right.\\
&\hspace{36mm}\left.
     {}\times\Re\left(\Psi_t(x)
     \overline{\Psi_t(x-\epsilon z)-\Psi_t(x)}\right)dz\,dx
     \right\rvert\\
&\quad+\sum_{k\in I}\left\lvert \int_{\R^d}\int_{|z|\leq 1}
     z\cdot\nabla\left(\sigma_k(x)\cdot\nabla\sigma_k(x)\right)
     \cdot\nabla\chi(z)\right.\\
&\hspace{36mm}\left.
     {}\times\Re\left(\Psi_t(x)
     \overline{\Psi_t(x-\epsilon z)-\Psi_t(x)}\right)dz\,dx
     \right\rvert\\
&\quad+\epsilon\|\Psi_t\|^2
\left(\sum_{k\in I}\|\sigma_k\|_{W^{3,\infty}}^2\right).
\end{align*}
Since $\Psi\in C_w([0,T];L^2)$ $\mathbb{P}-a.s.$, there exists a
full-probability set on which, for every $t\in[0,T]$, continuity of
translations in $L^2$ shows that each term in the sums above tends to $0$.
Moreover, on the same set,
\begin{align*}
    |I^{\epsilon}_t|&\lesssim \sum_{k\in I}\|\sigma_k\|_{W^{3,\infty}}^2\|\Psi_t\|^2\in L^1(0,T).
\end{align*}
Therefore also
\begin{align}\label{convergence_det_order_2}
\int_0^T |I_t^\epsilon|\,dt\rightarrow0
\end{align}
as $\epsilon\rightarrow 0$. Combining \cref{convergence_det_order_0_1,convergence_det_order_0_2,convergence_stoch_order_0,convergence_det_order_2_1,convergence_stoch_order_1,convergence_det_order_2}
we conclude
\begin{align*}
    \|\Psi_t\|^2&=\|\psi_0\|^2 \quad \mathbb{P}-a.s.
\end{align*}
for each $t\in [0,T]$ and this completes the proof.


\bibliography{biblio.bib}{}

@book{pazy2012semigroups,
  title={Semigroups of linear operators and applications to partial differential equations},
  author={Pazy, Amnon},
  year={2012},
  publisher={Springer Science \& Business Media}
}

@book{da2014stochastic,
  title={Stochastic equations in infinite dimensions},
  author={Da Prato, Giuseppe and Zabczyk, Jerzy},
  year={2014},
  publisher={Cambridge university press}
}

@book{janson1997gaussian,
  title={Gaussian {H}ilbert {S}paces},
  author={Janson, Svante},
  year={1997},
  publisher={Cambridge university press}
}

@article{nelson1973free,
  title={The free {M}arkoff field},
  author={Nelson, Edward},
  journal={Journal of Functional Analysis},
  volume={12},
  number={2},
  pages={211--227},
  year={1973},
  publisher={Academic Press}
}

@incollection{Arn,
  title={Hasselmann’s program revisited: {T}he analysis of stochasticity in deterministic climate models},
  author={Arnold, Ludwig},
  booktitle={Stochastic climate models},
  pages={141--157},
  year={2001},
  publisher={Springer}
}

@article{BerDes,
  title = {Macroscopic fluctuation theory},
  author = {Bertini, Lorenzo and De Sole, Alberto and Gabrielli, Davide and Jona-Lasinio, Giovanni and Landim, Claudio},
  journal = {Rev. Mod. Phys.},
  volume = {87},
  issue = {2},
  pages = {593--636},
  numpages = {44},
  year = {2015},
  month = {Jun},
  publisher = {American Physical Society},
  doi = {10.1103/RevModPhys.87.593},
  url = {https://link.aps.org/doi/10.1103/RevModPhys.87.593}
}

@article{BesPen,
  title={Dynamical confinement and magnetic traps for charges and spins},
  author={Besharat, Afshin and Penin, Alexander A},
  journal={Physical Review A},
  volume={113},
  number={2},
  pages={023107},
  year={2026},
  publisher={APS}
}

@article{CorFis,
  title={The {D}ean--{K}awasaki equation and the structure of density fluctuations in systems of diffusing particles},
  author={Cornalba, Federico and Fischer, Julian},
  journal={Archive for Rational Mechanics and Analysis},
  volume={247},
  number={5},
  pages={76},
  year={2023},
  publisher={Springer}
}

@article{DirFeh,
  title={C{onservative Stochastic PDE and Fluctuations of the Symmetric Simple Exclusion Process}},
  author={Dirr, Nicolas and Fehrman, Benjamin and Gess, Benjamin},
  journal={Communications in Mathematical Physics},
  volume={407},
  number={4},
  pages={74},
  year={2026},
  publisher={Springer}
}

@article{DjuKre,
  title={Weak error analysis for a nonlinear {SPDE} approximation of the {D}ean--{K}awasaki equation},
  author={Djurdjevac, Ana and Kremp, Helena and Perkowski, Nicolas},
  journal={Stochastics and Partial Differential Equations: Analysis and Computations},
  volume={12},
  number={4},
  pages={2330--2355},
  year={2024},
  publisher={Springer}
}

@book{kallenberg2002,
 author = {Kallenberg, Olav},
 title = {Foundations of modern probability.},
 edition = {2nd ed.},
 fseries = {Probability and its Applications},
 series = {Probab. Appl.},
 issn = {2297-0371},
 isbn = {0-387-95313-2},
 year = {2002},
 publisher = {New York, NY: Springer},
 language = {English},
 zbMATH = {1713116},
 Zbl = {0996.60001}
}

@article{touchette2009,
  title={The large deviation approach to statistical mechanics},
  author={Touchette, Hugo},
  journal={Physics Reports},
  volume={478},
  number={1-3},
  pages={1--69},
  year={2009},
  publisher={Elsevier}
}

@book{freidlinwentzell3,
 author = {Freidlin, Mark I and Wentzell, Alexander D},
 title = {Random perturbations of dynamical systems. {Translated} from the {Russian} by {J} {Sz{\"u}cs}},
 edition = {3rd ed.},
 fseries = {Grundlehren der Mathematischen Wissenschaften},
 series = {Grundlehren Math. Wiss.},
 issn = {0072-7830},
 volume = {260},
 isbn = {978-3-642-25846-6; 978-3-642-25847-3},
 year = {2012},
 publisher = {Berlin: Springer},
 language = {English},
 doi = {10.1007/978-3-642-25847-3},
 zbMATH = {6002853},
 Zbl = {1267.60004}
}

@article{veretennikov,
 author = {Veretennikov, Alexander Y},
 title = {On large deviations for {SDEs} with small diffusion and averaging.},
 journal = {Stochastic Processes and their Applications},
 issn = {0304-4149},
 volume = {89},
 number = {1},
 pages = {69--79},
 year = {2000},
 language = {English},
 doi = {10.1016/S0304-4149(00)00013-2},
 zbMATH = {2098341},
 Zbl = {1045.60065}
}

@article {Hairer_1,
    AUTHOR = {Hairer, Martin and Li, Xue-Mei},
     TITLE = {Averaging dynamics driven by fractional {B}rownian motion},
  JOURNAL = {The Annals of Probability},
    VOLUME = {48},
      YEAR = {2020},
    NUMBER = {4},
     PAGES = {1826--1860},
      ISSN = {0091-1798,2168-894X},
   MRCLASS = {60G22 (60H05 60H10 60L20)},
  MRNUMBER = {4124526},
MRREVIEWER = {Mireia\ Besal\'u},
       DOI = {10.1214/19-AOP1408},
       URL = {https://doi.org/10.1214/19-AOP1408},
}

@article{hairer2022generating,
  title={Generating diffusions with fractional {B}rownian motion},
  author={Hairer, Martin and Li, Xue-Mei},
  journal={Communications in Mathematical Physics},
  volume={396},
  number={1},
  pages={91--141},
  year={2022},
  publisher={Springer}
}

@article{li2022slow,
  title={Slow-fast systems with fractional environment and dynamics},
  author={Li, Xue-Mei and Sieber, Julian},
  journal={The Annals of Applied Probability},
  volume={32},
  number={5},
  pages={3964--4003},
  year={2022},
  publisher={JSTOR}
}

@article {cerrai3,
    AUTHOR = {Cerrai, Sandra and Freidlin, Mark},
     TITLE = {Fast flow asymptotics for stochastic incompressible viscous
              fluids in {$\mathbb{R}^2$} and {SPDE}s on graphs},
  JOURNAL = {Probability Theory and Related Fields},
    VOLUME = {173},
      YEAR = {2019},
    NUMBER = {1-2},
     PAGES = {491--535},
      ISSN = {0178-8051,1432-2064},
   MRCLASS = {60H15 (60J25 70K65)},
  MRNUMBER = {3916113},
MRREVIEWER = {Vassili\ N.\ Kolokol\cprime tsov},
       DOI = {10.1007/s00440-018-0839-8},
       URL = {https://doi.org/10.1007/s00440-018-0839-8},
}

@article {cerrai1,
    AUTHOR = {Cerrai, Sandra},
     TITLE = {A {K}hasminskii type averaging principle for stochastic
              reaction-diffusion equations},
  JOURNAL = {The Annals of Applied Probability},
    VOLUME = {19},
      YEAR = {2009},
    NUMBER = {3},
     PAGES = {899--948},
      ISSN = {1050-5164,2168-8737},
   MRCLASS = {60H15 (35R60 37L55)},
  MRNUMBER = {2537194},
       DOI = {10.1214/08-AAP560},
       URL = {https://doi.org/10.1214/08-AAP560},
}

@article{cerrai2009averaging,
  title={Averaging principle for a class of stochastic reaction--diffusion equations},
  author={Cerrai, Sandra and Freidlin, Mark},
  journal={Probability theory and related fields},
  volume={144},
  number={1},
  pages={137--177},
  year={2009},
  publisher={Springer}
}

@article{cerrai2009normal,
  title={Normal deviations from the averaged motion for some reaction--diffusion equations with fast oscillating perturbation},
  author={Cerrai, Sandra},
  journal={Journal de math{\'e}matiques pures et appliqu{\'e}es},
  volume={91},
  number={6},
  pages={614--647},
  year={2009},
  publisher={Elsevier}
}

@article{pardoux2003poisson,
  title={On the {P}oisson equation and diffusion approximation {II}},
  author={Pardoux, Étienne and Veretennikov, Alexander Y},
  journal={The Annals of Probability},
  volume={31},
  number={3},
  pages={1166--1192},
  year={2003},
  publisher={Institute of Mathematical Statistics}
}

@article {pardoux_3,
    AUTHOR = {Pardoux, Étienne and Veretennikov, Alexander Y},
     TITLE = {On the {P}oisson equation and diffusion approximation {III}},
  JOURNAL = {The Annals of Probability},
    VOLUME = {33},
      YEAR = {2005},
    NUMBER = {3},
     PAGES = {1111--1133},
      ISSN = {0091-1798,2168-894X},
   MRCLASS = {60J60 (35J70 60F17)},
  MRNUMBER = {2135314},
MRREVIEWER = {Nikita\ Y.\ Ratanov},
       DOI = {10.1214/009117905000000062},
       URL = {https://doi.org/10.1214/009117905000000062},
}

@article {pardoux_1,
    AUTHOR = {Pardoux, Étienne and Veretennikov, Alexander Y},
     TITLE = {On the {P}oisson equation and diffusion approximation {I}},
  JOURNAL = {The Annals of Probability},
    VOLUME = {29},
      YEAR = {2001},
    NUMBER = {3},
     PAGES = {1061--1085},
      ISSN = {0091-1798,2168-894X},
   MRCLASS = {60H30 (35J15 60J45 60J60)},
  MRNUMBER = {1872736},
MRREVIEWER = {Jean-Claude\ Zambrini},
       DOI = {10.1214/aop/1015345596},
       URL = {https://doi.org/10.1214/aop/1015345596},
}

@article{kato1988commutator,
  title={Commutator estimates and the {E}uler and {N}avier-{S}tokes equations},
  author={Kato, Tosio and Ponce, Gustavo},
  journal={Communications on Pure and Applied Mathematics},
  volume={41},
  number={7},
  pages={891--907},
  year={1988},
  publisher={Wiley Online Library}
}

@article{gao2018averaging,
  title={Averaging principle for the higher order nonlinear {S}chr{\"o}dinger equation with a random fast oscillation},
  author={Gao, Peng},
  journal={Journal of Statistical Physics},
  volume={171},
  number={5},
  pages={897--926},
  year={2018},
  publisher={Springer}
}

@article{gao2017averaging,
  title={AVERAGING PRINCIPLE FOR THE {S}CHR{\"O}DINGER EQUATIONS.},
  author={Gao, Peng and Li, Yong},
  journal={Discrete \& Continuous Dynamical Systems-Series B},
  volume={22},
  number={6},
  year={2017}
}

@article{brzezniak2013existence,
  title={E{xistence of a martingale solution of the stochastic Navier--Stokes equations in unbounded 2D and 3D domains}},
  author={Brze{\'z}niak, Zdzis{\l}aw and Motyl, El{\.z}bieta},
  journal={Journal of Differential Equations},
  volume={254},
  number={4},
  pages={1627--1685},
  year={2013},
  publisher={Elsevier}
}

@article{kap,
  title={Dynamic stability of the pendulum with vibrating suspension point},
  author={Kapitza, Pyotr Leonidovich},
  journal={Soviet Physics JETP},
  volume={21},
  number={5},
  pages={588--597},
  year={1951}
}

@article{crippa2025zero,
  title={Zero-noise selection and {L}arge {D}eviations in $ {L}^{\infty}_t {L}^{p}_x $ for the stochastic transport equation beyond {D}iPerna-{L}ions},
  author={Crippa, Gianluca and Luongo, Eliseo and Pappalettera, Umberto},
  journal={arXiv preprint arXiv:2506.06947},
  year={2025}
}

@article{fehrman2023non,
  title={Non-equilibrium large deviations and parabolic-hyperbolic {PDE} with irregular drift},
  author={Fehrman, Benjamin and Gess, Benjamin},
  journal={Inventiones mathematicae},
  volume={234},
  number={2},
  pages={573--636},
  year={2023},
  publisher={Springer}
}

@article{RhaGil,
  title={Effective {H}amiltonians for periodically driven systems},
  author={Rahav, Saar and Gilary, Ido and Fishman, Shmuel},
  journal={Physical Review A},
  volume={68},
  number={1},
  pages={013820},
  year={2003},
  publisher={APS}
}

@article{RibPer,
  title={Relativistic ponderomotive force in the regime of extreme focusing},
  author={Ribbing, Johan and Perosa, Giovanni and Goryashko, Vitaliy},
  journal={Optics Letters},
  volume={50},
  number={6},
  pages={2093--2096},
  year={2025},
  publisher={Optica Publishing Group}
}

@article{MulMor,
  title={Beyond the usual approximations: {T}ime-dependent magnetic traps revisited},
  author={M{\"u}ller, Jens H and Morsch, Oliver and Ciampini, Donatella and Anderlini, Marco and Mannella, Riccardo and Arimondo, Ennio},
  journal={Comptes Rendus de l'Acad{\'e}mie des Sciences-Series IV-Physics},
  volume={2},
  number={4},
  pages={649--656},
  year={2001},
  publisher={Elsevier}
}

@article{Has,
  title={Stochastic climate models part {I}. {T}heory},
  author={Hasselmann, Klaus},
  journal={tellus},
  volume={28},
  number={6},
  pages={473--485},
  year={1976},
  publisher={Taylor \& Francis}
}

@article{KugPau,
  title={A magnetic storage ring for neutrons},
  author={K{\"u}gler, K-J and Paul, W and Trinks, U},
  journal={Physics Letters B},
  volume={72},
  number={3},
  pages={422--424},
  year={1978},
  publisher={Elsevier}
}

@article{Leg,
  title={Bose-{E}instein condensation in the alkali gases: {S}ome fundamental concepts},
  author={Leggett, Anthony J},
  journal={Reviews of modern physics},
  volume={73},
  number={2},
  pages={307},
  year={2001},
  publisher={APS}
}

@article{LovMeh,
  title={Magnetic confinement of a neutral gas},
  author={Lovelace, Richard V E and Mehanian, Courosh and Tommila, Timo J and Lee, David M},
  journal={Nature},
  volume={318},
  number={6041},
  pages={30--36},
  year={1985},
  publisher={Nature Publishing Group UK London}
}

@article{MonCam,
  title={Programmable quantum simulations of spin systems with trapped ions},
  author={Monroe, Christopher and Campbell, Wes C and Duan, L-M and Gong, Z-X and Gorshkov, Alexey V and Hess, Paul W and Islam, Rajibul and Kim, Kihwan and Linke, Norbert M and Pagano, Guido and others},
  journal={Reviews of Modern Physics},
  volume={93},
  number={2},
  pages={025001},
  year={2021},
  publisher={APS}
}

@article{Pri,
  title={Cooling neutral atoms in a magnetic trap for precision spectroscopy},
  author={Pritchard, David E},
  journal={Physical Review Letters},
  volume={51},
  number={15},
  pages={1336},
  year={1983},
  publisher={APS}
}

@article{RidDav,
  title={Particle motion in rapidly oscillating potentials: {T}he role of the potential’s initial phase},
  author={Ridinger, A and Davidson, Nir},
  journal={Physical Review A—Atomic, Molecular, and Optical Physics},
  volume={76},
  number={1},
  pages={013421},
  year={2007},
  publisher={APS}
}

@article {kurtzYW,
    AUTHOR = {Kurtz, Thomas G},
     TITLE = {The {Y}amada-{W}atanabe-{E}ngelbert theorem for general
              stochastic equations and inequalities},
  JOURNAL = {Electronic Journal of Probability},
    VOLUME = {12},
      YEAR = {2007},
     PAGES = {951--965},
      ISSN = {1083-6489},
   MRCLASS = {60H99 (60H10 60H15 60H20 60H25)},
  MRNUMBER = {2336594},
MRREVIEWER = {Rainer\ Buckdahn},
       DOI = {10.1214/EJP.v12-431},
       URL = {https://doi.org/10.1214/EJP.v12-431},
}

@article{jakubowski1998almost,
  title={The almost sure Skorokhod representation for subsequences in nonmetric spaces},
  author={Jakubowski, Adam},
  journal={Theory of Probability \& Its Applications},
  volume={42},
  number={1},
  pages={167--174},
  year={1998},
  publisher={SIAM}
}

@article{bagnara2025no,
  title={No blow-up by nonlinear {I}t{\^o} noise for the {E}uler equations},
  author={Bagnara, Marco and Maurelli, Mario and Xu, Fanhui},
  journal={Electronic Journal of Probability},
  volume={30},
  pages={1--29},
  year={2025},
  publisher={The Institute of Mathematical Statistics and the Bernoulli Society}
}

@article{tubaro1984estimate,
  title={An estimate of {B}urkholder type for stochastic processes defined by the stochastic integral},
  author={Tubaro, Luciano},
  journal={Stochastic Analysis and Applications},
  volume={2},
  number={2},
  pages={187--192},
  year={1984},
  publisher={Taylor \& Francis}
}

@article{butori2026background,
  title={{Background Vlasov equations and Young measures for passive scalar and vector advection equations under special stochastic scaling limits}},
  author={Butori, Federico and Flandoli, Franco and Luongo, Eliseo and Tahraoui, Yassine},
  journal={Probability Theory and Related Fields},
  pages={1--62},
  year={2026},
  publisher={Springer}
}

@article {Brzezniak_skoro,
    AUTHOR = {Brze{\'z}niak, Zdzis{\l}aw and Ondrej{\'a}t, Martin},
     TITLE = {Stochastic geometric wave equations with values in compact
              {R}iemannian homogeneous spaces},
  JOURNAL = {The Annals of Probability},
    VOLUME = {41},
      YEAR = {2013},
    NUMBER = {3B},
     PAGES = {1938--1977},
      ISSN = {0091-1798,2168-894X},
   MRCLASS = {60H15 (35R60)},
  MRNUMBER = {3098063},
       DOI = {10.1214/11-AOP690},
       URL = {https://doi.org/10.1214/11-AOP690},
}

@book {Brez_book_f,
    AUTHOR = {Brezis, Ha\"im},
     TITLE = {Analyse fonctionnelle},
    SERIES = {Collection Math\'ematiques Appliqu\'ees pour la Ma\^itrise.
              [Collection of Applied Mathematics for the Master's Degree]},
      NOTE = {Th\'eorie et applications. [Theory and applications]},
 PUBLISHER = {Masson, Paris},
      YEAR = {1983},
     PAGES = {xiv+234},
      ISBN = {2-225-77198-7},
   MRCLASS = {46-01 (47-01)},
  MRNUMBER = {697382},
}

@article{brzezniak2014existence,
  title={The existence of martingale solutions to the stochastic {B}oussinesq equations},
  author={Brze{\'z}niak, Zdzis{\l}aw and Motyl, El{\.z}bieta},
  journal={Global and Stochastic Analysis},
  volume={1},
  number={2},
  pages={175--216},
  year={2014}
}

@book{aliprantis2006infinite,
  title={Infinite dimensional analysis: a hitchhiker’s guide},
  author={Aliprantis, Charalambos D and Border, Kim C},
  year={2006},
  publisher={Springer}
}

@article{diperna1989ordinary,
  title={Ordinary differential equations, transport theory and {S}obolev spaces},
  author={DiPerna, Ronald J and Lions, Pierre-Louis},
  journal={Inventiones mathematicae},
  volume={98},
  number={3},
  pages={511--547},
  year={1989},
  publisher={Springer}
}

@article {LDPweak_convergence,
    AUTHOR = {Budhiraja, Amarjit and Dupuis, Paul and Maroulas, Vasileios},
     TITLE = {Large deviations for infinite dimensional stochastic dynamical
              systems},
  JOURNAL = {The Annals of Probability},
    VOLUME = {36},
      YEAR = {2008},
    NUMBER = {4},
     PAGES = {1390--1420},
      ISSN = {0091-1798,2168-894X},
   MRCLASS = {60H15 (35R60 37L55 60F10)},
  MRNUMBER = {2435853},
MRREVIEWER = {Mohamed\ Mellouk},
       DOI = {10.1214/07-AOP362},
       URL = {https://doi.org/10.1214/07-AOP362},
}

@book{flandoli2023stochastic,
  title={Stochastic partial differential equations in fluid mechanics},
  author={Flandoli, Franco and Luongo, Eliseo},
  volume={2330},
  year={2023},
  publisher={Springer}
}

@book{rozovskii2012stochastic,
  title={Stochastic evolution systems: linear theory and applications to non-linear filtering},
  author={Rozovskii, Boris Lvovich},
  year={2012},
  publisher={Springer Science \& Business Media}
}

@book{flandoli2024regularity,
  title={Regularity Theory and Stochastic Flows for Parabolic ISPDES},
  author={Flandoli, Franco},
  year={2024},
  publisher={CRC Press}
}

@article{davies1974markovian,
  author  = {Davies, Brian E},
  title   = {Markovian Master Equations},
  journal = {Communications in Mathematical Physics},
  volume  = {39},
  pages   = {91--110},
  year    = {1974},
  doi     = {10.1007/BF01608389}
}

@article{gorini1976completely,
  author  = {Gorini, Vittorio and Kossakowski, Andrzej and Sudarshan, E. C. G.},
  title   = {Completely Positive Dynamical Semigroups of {$N$}-Level Systems},
  journal = {Journal of Mathematical Physics},
  volume  = {17},
  number  = {5},
  pages   = {821--825},
  year    = {1976},
  doi     = {10.1063/1.522979}
}

@article{lindblad1976generators,
  author  = {Lindblad, G{\"o}ran},
  title   = {On the Generators of Quantum Dynamical Semigroups},
  journal = {Communications in Mathematical Physics},
  volume  = {48},
  pages   = {119--130},
  year    = {1976},
  doi     = {10.1007/BF01608499}
}

@book{breuer2002theory,
  author    = {Breuer, Heinz-Peter and Petruccione, Francesco},
  title     = {The Theory of Open Quantum Systems},
  publisher = {Oxford University Press},
  address   = {Oxford},
  year      = {2002}
}

@book{barchielli2009quantum,
  author    = {Barchielli, Alberto and Gregoratti, Matteo},
  title     = {Quantum Trajectories and Measurements in Continuous Time: The Diffusive Case},
  series    = {Lecture Notes in Physics},
  volume    = {782},
  publisher = {Springer},
  address   = {Berlin, Heidelberg},
  year      = {2009},
  doi       = {10.1007/978-3-642-01298-3}
}

@article{khasminskii1968,
  author  = {Khasminskii, Rafail Z.},
  title   = {On the Principle of Averaging the {It{\^o}} Stochastic
             Differential Equations},
  journal = {Kybernetika},
  volume  = {4},
  number  = {3},
  pages   = {260--279},
  year    = {1968}
}

@article{Moh,
  title={On the 2{D} quantum tunneling},
  author={Mohammadpour, Hakimeh},
  journal={Acta Physica Polonica A},
  volume={130},
  number={3},
  pages={769--772},
  year={2016},
  publisher={Polska Akademia Nauk. Instytut Fizyki PAN}
}

@article{RakYus,
  title={Wave-packet rectification in graphene with alternating circular electrostatic potential barriers},
  author={Rakhimov, Khamdam and Yusupov, Hammid and Nurmatov, Shavkat and Chaves, Andrey and Berdiyorov, Golibjon},
  journal={Journal of Applied Physics},
  volume={137},
  number={14},
  year={2025},
  publisher={AIP Publishing}
}
\bibliographystyle{plain}

\end{document}